\RequirePackage{snapshot}
\newif\ifproofs\proofsfalse\ifproofs\RequirePackage[displaymath,mathlines]{lineno}\fi

\newif\ifsubsections
\subsectionstrue 

\documentclass[]{pcmi}

\usepackage[sc]{mathpazo}          
\usepackage{eulervm}               
\usepackage[mathscr]{eucal}
\usepackage[scaled=0.86]{berasans} 
\usepackage[scaled=1]{inconsolata} 
\usepackage[T1]{fontenc}

\usepackage[%
	protrusion=true,
	expansion=false,
	auto=false
	]{microtype}

\usepackage{xcolor}
\usepackage{graphicx}

\ifdraft
	\definecolor{linkred}{rgb}{0.7,0.2,0.2}
	\definecolor{linkblue}{rgb}{0,0.2,0.6}
\else
	\definecolor{linkred}{rgb}{0.0,0.0,0.0}
	\definecolor{linkblue}{rgb}{0,0.0,0.0}
\fi

\usepackage{amssymb}
\usepackage{bbm}
\usepackage{bm}
\usepackage{dsfont}
\usepackage{stmaryrd}
\usepackage{cite}
\usepackage{xspace}
\usepackage{tikz}
\usepackage{comment}

\usepackage{tikz,pgfplots,pgfmath}
\newlength\figurewidth 
\newlength\figureheight 
\pgfplotsset{compat=newest}
\usetikzlibrary{plotmarks}
\usetikzlibrary{arrows.meta}
\usepgfplotslibrary{patchplots}

\PassOptionsToPackage{hyphens}{url} 
\usepackage[
    setpagesize=false,
    pagebackref,
	plainpages=false,pdfpagelabels,pageanchor=true,
    pdfstartview={FitH 1000},
    bookmarksnumbered=false,
    linktoc=all,
    colorlinks=true,
    anchorcolor=black,
    menucolor=black,
    runcolor=black,
    filecolor=black,
    linkcolor=linkblue,
	citecolor=linkblue,
	urlcolor=linkred,
]{hyperref}
\usepackage[backrefs,msc-links,nobysame]{amsrefs}

\customizeamsrefs 

\theoremstyle{plain} 
\newtheorem{theorem}[equation]{Theorem}
\newtheorem*{theorem*}{Theorem}
\newtheorem{lemma}[equation]{Lemma}
\newtheorem*{lemma*}{Lemma}
\newtheorem{exercise}[equation]{Exercise}
\newtheorem*{exercise*}{Exercise}

\newtheorem*{corollary*}{Corollary}
\newtheorem{proposition}[equation]{Proposition}
\newtheorem*{proposition*}{Proposition}
\theoremstyle{definition} 
\newtheorem{definition}[equation]{Definition}
\newtheorem*{definition*}{Definition}

\newtheorem*{example*}{Example}
\theoremstyle{remark} 
\newtheorem{remark}[equation]{Remark}

\newtheorem*{remark*}{Remark}
\newtheorem*{remarks*}{Remarks}

\numberwithin{equation}{subsection}
\let\oldsubsubsection\subsubsection
\renewcommand{\subsubsection}[1]{%
\setcounter{subsubsection}{\value{equation}}%
\oldsubsubsection{#1}
\refstepcounter{equation}
}

\makeindex

\newcommand{\bla}{\mbox{\boldmath$\lambda$}}
\newcommand{\bnu}{\mbox{\boldmath$\nu$}}

\newcommand{\bvarphi}{\mbox{\boldmath$\varphi$}}

\def\cH{{\mathcal H}}
\def\cF{{\mathcal F}}

\newcommand{\wt}{\widetilde}
\newcommand{\wh}{\widehat}

\newcommand{\eps}{\varepsilon}

\newcommand{\rd}{{\rm d}}

\newcommand{\bt}{{\bf {t}}}
\newcommand{\ba}{{\bf{a}}}
\newcommand{\bbm}{{\bf{m}}}
\newcommand{\bbf}{{\bf{f}}}
\newcommand{\bx}{{\bf{x}}}
\newcommand{\by}{{\bf{y}}}
\newcommand{\bu}{{\bf{u}}}
\newcommand{\bv}{{\bf{v}}}
\newcommand{\bw}{{\bf{w}}}
\newcommand{\bh}{{\bf{h}}}

\newcommand{\al}{\alpha}

\newcommand{\be}{\begin{equation}}
\newcommand{\ee}{\end{equation}}

\newcommand{\e}{{\varepsilon}}

\newcommand{\la}{\lambda}

\newcommand{\om}{{\omega}}

\newcommand{\cC}{{\cal C}}

\newcommand{\wkto}{\rightharpoonup}

\newcommand{\cal}{\mathcal}

\renewcommand{\le}{\leq}
\renewcommand{\ge}{\geq}

\renewcommand{\P}{\mathbb{P}}
\newcommand{\E}{\mathbb{E}}
\newcommand{\R}{\mathbb{R}}
\newcommand{\C}{\mathbb{C}}
\newcommand{\N}{\mathbb{N}}
\newcommand{\Z}{\mathbb{Z}}

\newcommand{\norm}[1]{\lVert #1 \rVert}

\DeclareMathOperator{\tr}{Tr}

\DeclareMathOperator{\re}{Re}
\DeclareMathOperator{\im}{Im}
\DeclareMathOperator{\sgn}{sgn}

\newcommand{\cS}{{\cal S}}
\newcommand{\cU}{{\cal U}}
\newcommand{\cK}{{\cal K}}

\newcommand{\ov}{\overline}

\newcommand{\nc}{}

\newcommand{\1} {\mspace{1 mu}}
\newcommand{\2} {\mspace{2 mu}}

\newcommand{\Cp} {\mathbb{H}}

\newcommand{\Bounded}{\mathcal{B}}

\newcommand{\xSpace} {\mathfrak{X}} 
\newcommand{\xMeasure} {\pi} 

\begin{document}
\ifproofs\linenumbers\fi

\author{L\'aszl\'o Erd{\H o}s}

\address{Institute of Science and Technology (IST) Austria\\
Am Campus 1\\
A-3400, Klosterneuburg, Austria}
\email{lerdos@ist.ac.at}

\thanks{Partially supported  by ERC Advanced Grant, RANMAT 338804}

\title[The matrix Dyson equation and its applications for random matrices]{The Matrix Dyson Equation \\ and its Applications for Random Matrices}

%
%
\keywords{Park City Mathematics Institute, Random matrix, Matrix Dyson Equation, local semicircle law, Dyson sine kernel, Wigner-Dyson-Mehta conjecture,
Tracy-Widom distribution, Dyson Brownian motion}

\begin{abstract} These lecture notes  are a concise   introduction of recent
techniques to  prove local spectral universality for a large class of  random matrices.
The general strategy  is presented following the recent book with H.T. Yau \cite{ErdYau2017}.
We extend the scope of this book by focusing on new techniques developed to deal with generalizations of 
Wigner matrices that allow for non-identically distributed entries and even for correlated entries.
This requires to analyze a system of nonlinear equations, or more generally a nonlinear matrix
equation called the {\it Matrix Dyson Equation (MDE)}. We demonstrate that stability properties 
of the MDE play a central role in random matrix theory. The analysis of MDE is based upon
joint works with J. Alt, O. Ajanki, D. Schr\"oder and  T. Kr\"uger that are supported by the  ERC Advanced Grant, RANMAT 338804 of the European Research Council. 

The lecture notes were written for the  27th Annual PCMI Summer Session on Random Matrices held
in 2017. The current edited version will appear in the IAS/Park City Mathematics Series, Vol. 26. 
\end{abstract}

%
%
\date{Sep 1, 2017}
\maketitle
\thispagestyle{empty}

\tableofcontents

\section{Introduction}\label{sec:intro}

{\leavevmode\kern27pt
\begin{minipage}[c]{3.9in}
 {\it ``Perhaps I am now too courageous when I try to guess the distribution of the distances between
successive levels (of energies of heavy nuclei).   Theoretically, the 
situation is quite simple if one attacks the problem  in a 
simpleminded fashion.  The question is simply what are the 
 distances  of the characteristic values of  a symmetric
matrix with random coefficients.''  }
\medskip 
\medskip 
\centerline{\hfill Eugene Wigner on 
the Wigner surmise, 1956}   
\medskip 
\end{minipage}
}

The cornerstone of probability theory is the fact that the collective behavior of many independent random  variables 
 exhibits universal patterns; the obvious examples are
the {\it law of large numbers (LLN)} and the {\it central limit theorem (CLT)}. They assert that the normalized sum of $N$
independent, identically distributed (i.i.d.) random variables $X_1, X_2, \ldots , X_N \in \R$ converge to their common
expectation value:
\be\label{lan}
 \frac{1}{N}\big( X_1+ X_2 + \ldots +X_N) \to \E X_1
\ee
as $N\to \infty$, 
and their centered average with a $\sqrt{N}$ normalization converges to the centered Gaussian distribution
with variance $\sigma^2= \mbox{Var}(X)$:
\[
   S_N:=\frac{1}{\sqrt{N}} \sum_{i=1}^N \big( X_i - \E X_i) \Longrightarrow {\cal N}(0, \sigma^2).
\]
The convergence in the latter case is understood in distribution, i.e. tested against any bounded continuous function $\Phi$:
\[
  \E \Phi(S_N ) \to \E \Phi(\xi),
\]
where $\xi$ is an ${\cal N}(0, \sigma^2)$ distributed normal random variable.
  
 These basic results directly  extend to 
random vectors instead of scalar valued random variables. The main question is:  what are their analogues in the non-commutative
setting, e.g. for matrices? Focusing on their spectrum,  what  do eigenvalues of typical large random matrices look like?
Is there a deterministic limit of some relevant random  quantity, like the average in case of the LLN \eqref{lan}. Is there some
stochastic {\it universality pattern} arising, similarly to the ubiquity of the Gaussian distribution in Nature owing
to the central limit theorem?

These natural questions could have been raised from pure curiosity by mathematicians, but historically random matrices first appeared
in statistics (Wishart in 1928 \cite{Wis1928}), where empirical covariance matrices of measured data (samples) naturally form
a random matrix ensemble and the eigenvalues play a crucial role in principal component analysis.
  The question regarding  the universality of eigenvalue  statistics, however,  
appeared  only in the 1950's in the pioneering work \cite{Wig1955} of Eugene Wigner. He was motivated by
a simple observation looking at data from nuclear physics, but he immediately realized 
a very general phenomenon in the background.  He noticed from experimental data that gaps in
energy levels  of large nuclei tend to follow the same statistics irrespective of
the material. Quantum mechanics predicts that energy levels are eigenvalues of a self-adjoint 
operator, but the correct  Hamiltonian  operator describing nuclear forces was not known at that time. 
 Instead of pursuing  a direct solution of  this problem, Wigner appealed to  a
phenomenological model to explain  his observation.   His pioneering idea was to model the complex Hamiltonian  
by a random matrix with independent entries. All physical details of the system 
were ignored except one, the {\it symmetry type}: systems with time reversal symmetry
were modeled by real symmetric random matrices, while complex Hermitian 
random matrices were used for systems without time reversal symmetry (e.g.  with magnetic forces). 
This simple-minded model  amazingly  reproduced the correct  gap statistics.
Eigenvalue gaps carry basic information about possible excitations of the quantum systems.
In fact, beyond nuclear physics,   random matrices enjoyed a renaissance in the theory
of disordered quantum systems, where the spectrum of a non-interacting electron
in a random impure environment was studied. It turned out that eigenvalue
statistics is one of the basic signatures of the celebrated {\it metal-insulator}, or {\it Anderson} transition
in condensed matter physics \cite{And1958}.

\subsection{Random matrix ensembles}

Throughout these notes we will consider  $N\times N$ square matrices of the form
\be\label{Hdef}
H=H^{(N)} = 
\begin{pmatrix} h_{11} &  h_{12} & \ldots & h_{1N} \\ 
 h_{21} &  h_{22} & \ldots & h_{2N} \\
\vdots & \vdots &  & \vdots\\
 h_{N1} &  h_{N2} & \ldots & h_{NN} \\
\end{pmatrix} .
\ee
The entries are  real or complex random variables constrained by
the symmetry 
\[
h_{ij}=  \bar h_{ji}, \qquad i,j=1, \ldots, N,
\]
so that $H=H^*$ is either Hermitian (complex) or symmetric (real). 
In particular, the eigenvalues of $H$, $\lambda_1\le \lambda_2\le \ldots \le \lambda_N$
are real and we will be interested in their statistical
behavior induced by the randomness of $H$ as the size of the matrix $N$ goes to infinity.
Hermitian symmetry is very natural from  the point of view of  physics applications and it makes the problem
much more tractable mathematically.  Nevertheless, there has recently been  an increasing interest
in non-hermitian random matrices as well motivated by systems of ordinary differential equations
with random coefficients arising in biological networks (see, e.g.   \cite{EKR, Nelson2016} and references therein).

There are essentially two customary ways to define a probability measure on the space of $N\times N$ random matrices
that we now briefly introduce. The main point is that
either 
one specifies the distribution of  the matrix elements directly or one aims at a basis-independent measure. 
The prototype of the first case is the Wigner ensembles and we will be focusing on its natural generalizations
in these notes. The typical example of the second case are the invariant ensembles. We will briefly introduce them now.

\subsubsection{Wigner ensemble}

The most prominent
example of the first class is the traditional {\bf Wigner matrix}, where the matrix elements $h_{ij}$ are i.i.d. random
variables subject to the symmetry constraint $h_{ij} = \overline{h_{ji}}$. More precisely, Wigner matrices are defined by  assuming that
\be\label{wignermoment}
     \E h_{ij} =0, \qquad \E |h_{ij}|^2 = \frac{1}{N}.
\ee
In the {\it real symmetric case},  the collection of random variables $\{ h_{ij} \; : \; i\le j\}$ are independent, identically distributed,
while in the {\it complex hermitian case} the distributions of $ \{ \re h_{ij}, \im h_{ij} \; : \; 1\le i<j\le N\}$ and $\{ \sqrt{2} h_{ii} \; : \;  i=1,2,\ldots, N\}$
are independent and identical. 

The common variance of the matrix elements is the single parameter of the model; by a trivial
rescaling  we may fix it conveniently.  The normalization $1/N$ chosen in \eqref{wignermoment}
guarantees that the typical size of the eigenvalues remain of order 1 even as $N$ tends to infinity.
To see this, we may compute the expectation of the trace of $H^2$ in two different ways:
\be\label{aves}
     \E \sum_i \lambda_i^2  =  \E \tr H^2 = \E \sum_{ij} |h_{ij}|^2 = N
\ee
indicating that $\lambda_i^2\sim 1$ on average.  In fact, much stronger bounds hold and one can prove
that 
\[
   \| H\| = \max_i |\lambda_i| \to 2, \qquad N\to \infty,
\]
in probability.

 In these notes we will  focus on Wigner ensembles and their extensions, where 
we will drop the condition of identical distribution and we will weaken the independence condition.
We will call them {\it Wigner type} and {\it correlated ensembles}.  Nevertheless,
for completeness we also present the other class of random matrices.

\subsubsection{Invariant ensembles}

The  ensembles in the second class are defined by the measure
\be
   \P(H) \rd H : = Z^{-1} \exp{ \big(-\frac{\beta}{2} N\tr V(H)\big) }\rd H.
\label{phold}
\ee
Here $\rd H = \prod_{i\le j} \rd h_{ij} $ is the flat Lebesgue measure on $\R^{N(N+1)/2}$
(in case of complex Hermitian matrices and $i<j$, $\rd h_{ij} $ is the Lebesgue measure
on the complex plane $\C$ instead of $\R$).  The (potential) function $V:\R\to\R$ is assumed to
grow mildly at infinity (some logarithmic growth would suffice)
to ensure that the measure defined in \eqref{phold} is finite. 
The parameter $\beta$ distinguishes between the two symmetry classes: $\beta=1$
for the real symmetric case, while $\beta=2$ for the complex hermitian case -- for traditional reason we
factor this parameter out of the potential.

Finally, $Z$ is
the normalization factor to make $P(H)\rd H$ a probability measure.
Similarly to the normalization of the variance in \eqref{wignermoment}, the factor $N$ in the exponent in \eqref{phold} guarantees that 
the eigenvalues remain order one even as $N\to \infty$. This scaling also  guarantees that empirical density of the eigenvalues will
have a deterministic limit without further rescaling.

Probability distributions of the form \eqref{phold}
are called {\bf invariant ensembles} since they are  invariant under
the orthogonal or unitary conjugation (in case of symmetric or Hermitian
matrices, respectively). For example,  in the Hermitian case,
for any fixed unitary matrix $U$, the transformation
\[
    H \to U^* H U
\]
leaves the distribution \eqref{phold} invariant thanks to $\tr V(U^*HU)=\tr V(H)$
and that $\rd (U^*HU) = \rd H$.

An important special case is when  $V$ is a quadratic polynomial, after shift and rescaling 
we may assume that $V(x)=\frac{1}{2} x^2$. In this case
\begin{align*}
    \P(H) \rd H  &  = Z^{-1} \exp{ \big(-\frac{\beta}{4} N \sum_{ij} |h_{ij}|^2\big) }\rd H \\
    & = Z^{-1} \prod_{i<j} \exp{ \big(-\frac{\beta}{2} N |h_{ij}|^2\big) }\rd h_{ij} \prod_i \exp{ \big(-\frac{\beta}{4} N h_{ii}^2\big) }\rd h_{ii},
\end{align*}
i.e. the measure factorizes and it is equivalent to  independent Gaussians for the matrix elements.
The factor $N$ in the definition \eqref{phold}  and the choice of $\beta$ ensure that we recover the normalization 
\eqref{wignermoment}. (A pedantic reader may notice that the normalization of the diagonal element
for the real symmetric case is off by a factor of 2, but this small discrepancy plays no role.)
The invariant Gaussian ensembles, i.e. \eqref{phold} with $V(x) = \frac{1}{2}x^2$, are called {\bf Gaussian orthogonal ensemble (GOE)}
for the real symmetric case $(\beta=1)$
and {\bf Gaussian unitary ensemble (GUE)} for the complex hermitian case $(\beta=2)$.

Wigner matrices and  invariant ensembles form two different universes with
quite different mathematical tools  available for their studies. In fact, these two
classes are almost disjoint because the Gaussian ensembles are the
only invariant Wigner matrices. This is the content of the following lemma:

\begin{lemma}[\cite{Dei1999} or Theorem 2.6.3 \cite{Meh1991}]\label{WH}
Suppose that the real symmetric or complex Hermitian matrix ensembles given in \eqref{phold}
have independent entries $h_{ij}$, $i\le j$. Then $V(x)$ is a quadratic
polynomial, $V(x) = a x^2 + bx + c$ with $a>0$. This means that apart from 
a trivial shift and  normalization, the ensemble is  GOE or GUE.
\end{lemma}

The significance of the Gaussian ensembles  is that they allow for explicit
calculations that are not available for Wigner matrices with general non-Gaussian 
single entry distribution. In particular the celebrated {\bf Wigner-Dyson-Mehta
correlation functions} can be explicitly obtained for the GOE and GUE ensembles.
Thus the typical proof of identifying the eigenvalue correlation function for
a general matrix ensemble goes through universality: one first proves
that the correlation function is independent of the distribution, hence
it is the same as GUE/GOE, and then, in the second step,
one computes the GUE/GOE correlation functions. This second step has
been completed by Gaudin, Mehta and Dyson in the 60's by an ingenious calculation, see e.g. the classical
treatise by Mehta  \cite{Meh1991}.

One of the  key ingredients of the explicit calculations is the surprising fact that the joint (symmetrized) density  function 
of the eigenvalues, $p(\lambda_1, \lambda_2,\ldots , \lambda_N)$ can be computed
explicitly for any invariant ensemble. It is given by
\be
    p_N(\lambda_1, \lambda_2, \ldots , \lambda_N) 
  = \mbox{const.} \prod_{i<j} (\lambda_i-\lambda_j)^\beta 
   e^{- \frac{\beta}{2}N\sum_{j=1}^N V(\lambda_j)}.
\label{expli}
\ee
where the constant ensures the normalization, but its exact value is typically unimportant.

\begin{remark} In other sections of these notes we usually label the eigenvalues in increasing order so that 
their probability density,  denoted by  $\wt p_N(\bla)$,   is defined on the set 
\[
  \Xi^{(N)}:= \{ \lambda_1 \le \lambda_2 \le\ldots \le \la_N\} \subset \R^N.
\] 
For the
purpose of  \eqref{expli}, however, we dropped this restriction and
we consider $p_N(\lambda_1, \lambda_2, \ldots ,\lambda_N)$ to be a
symmetric function of $N$ variables, $\bla =(\la_1, \ldots, \la_N)$
on $\R^N$.  The relation between the ordered and unordered
densities is clearly  
 $\wt p_N(\bla)= N!\, p_N(\bla)\cdot {\bf 1}(\bla\in \Xi^{(N)})$.
\end{remark}

 The emergence of the {\it Vandermonde determinant}
 \index{Vandermonde determinant} in \eqref{expli} is a result of 
integrating out the ``angle'' variables in \eqref{phold}, i.e., the  unitary matrix
 in the diagonalization of $H=U\Lambda U^*$. 
This is a remarkable formula since it gives  a direct access to  the eigenvalue  distribution.
In particular, it shows that the  eigenvalues are strongly correlated.
For example, no two eigenvalues can be too close to each other since the corresponding probability is suppressed by the
factor $\lambda_{j}-\lambda_i$ for any $i\ne j$; this phenomenon is called the {\bf level repulsion}. We remark that level repulsion also holds
for Wigner matrices  with smooth distribution  \cite{ErdSchYau2010} but its proof is much more involved.

In fact, one may  view the ensemble  \eqref{expli}  as a statistical physics question by rewriting $p_N$ 
as a classical Gibbs measure of a $N$ point particles on the line with a logarithmic  mean field interaction:
\be\label{mf}
 p_N(\bla) = \mbox{(const.)} e^{-\beta N \cH (\bla)}
\ee
with a Hamiltonian
\[
   \cH (\bla) = \frac{1}{2}\sum_i V(\lambda_i) - \frac{1}{N}\sum_{i<j} \log |\lambda_j-\lambda_i|.
\]
This ensemble of point particles with logarithmic interactions is  also called {\bf log-gas}. We remark that viewing 
the Gibbs measure \eqref{mf} as the starting point and forgetting about the matrix ensemble behind, the parameter $\beta$ does not have to be 1 or 2; it can be 
any positive number, $\beta>0$, and it has the interpretation of the inverse temperature.  We will not pursue general invariant ensembles in these notes.

\subsection{Eigenvalue statistics on different scales}

The normalization both in \eqref{wignermoment} and \eqref{phold}  is chosen in such a way
 that the typical eigenvalues remain of order 1 even in the large $N$ limit. In particular, the
typical distance between neighboring eigenvalues is of order $1/N$. We distinguish two
different scales for studying eigenvalues: {\it macroscopic} and {\it microscopic} scales.
With our scaling, the macroscopic scale
is order one and on this scale we detect the cumulative effect of $cN$  eigenvalues
with some positive constant $c$. In contrast, on the microscopic scales indiv\-idual eigenvalues
are detected; this scale is typically of order $1/N$. However, near the spectral edges, where
the density of eigenvalues goes to zero, the typical eigenvalue spacing hence the microscopic scale may be larger.
Some phenomena (e.g. fluctuations of linear statistics of eigenvalues)
occur on various {\it mesoscopic} scales that lie between the macroscopic and the microscopic scales.

\subsubsection{Eigenvalue density on macroscopic scales: global laws}

The first and simplest question is to determine the eigenvalue density, i.e. the behavior of the {\it empirical eigenvalue
density} or {\bf empirical density of states}
\be\label{empire}
    \mu_N(\rd x) : = \frac{1}{N}\sum_i \delta(x-\lambda_i)\rd x
\ee
in the large $N$ limit. This is a random measure, but under very general conditions it converges to a deterministic
measure, similarly to self-averaging property encoded in the law of large numbers \eqref{lan}.

\begin{figure}[h]
	\centering
	\includegraphics[width=.86\textwidth]{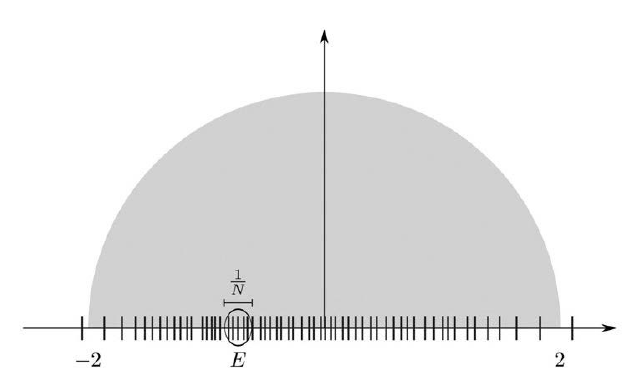}
	\caption{
	Semicircle law and eigenvalues of a GUE random matrix of size $N=60$. }
	\label{Fig:sc}
\end{figure}

For Wigner ensemble, the empirical distribution of eigenvalues converges  to 
the {\bf Wigner semicircle law}. \index{Wigner semicircle law|textbf} To formulate it more precisely, note that
the typical spacing between neighboring eigenvalues is of order $1/N$, so
in a fixed interval $[a,b]\subset \R$, one expects macroscopically many 
(of order $N$) eigenvalues.
More precisely, it can be shown (first proof was given by Wigner \cite{Wig1955})
 that  for any fixed $a\le b$ real numbers,
\be
   \lim_{N\to\infty} \frac{1}{N}\#\big\{ i \; : \; \lambda_i\in [a,b]\big\}
  = \int_a^b \varrho_{sc}(x)\rd x, \qquad \varrho_{sc}(x) : = \frac{1}{2\pi}
  \sqrt{(4-x^2)_+},
\label{sc} 
\ee
where $(a)_+:=\max\{ a, 0\}$ denotes the positive part of the number $a$.
Alternatively, one may formulate the Wigner semicircle law as the weak convergence in probability of
the empirical distribution $\mu_N$ to the semicircle distribution, $\varrho_{sc}(x)\rd x$. This means 
that the limit
\[
  \int_\R f(x) \mu_N(\rd x) = \frac{1}{N}\sum_i f(\lambda_i) \to \int_\R f(x) \varrho_{sc}(x)\rd x, \qquad N\to \infty
\]
holds in probability  for any bounded continuous  function $f$, i.e.,
\[ 
   \P\Big( \Big| \int_\R f(x) \mu_N(\rd x) - \int_\R f(x) \varrho_{sc}(x)\rd x\Big|\ge \e\Big) \to 0
\]
for any $\e>0$ as $N\to \infty$.

Note that the emergence of  the semicircle density is already a certain form of universality: the common distribution
of the individual matrix elements is ``forgotten''; the density of eigenvalues is asymptotically  always the same,
 independently of the details of the distribution of
the matrix elements. 

We will see that for a more general class of Wigner type matrices with zero expectation 
but not identical distribution a similar limit statement holds for the empirical density of
eigenvalues, i.e. there is a deterministic density function $\varrho(x)$ such that
\be\label{lime}
  \int_\R f(x) \mu_N(\rd x) = \frac{1}{N}\sum_i f(\lambda_i) \to \int f(x) \varrho(x)\rd x, \qquad N\to \infty
\ee
holds. The density function $\varrho$  thus approximates the empirical density, so we will call it {\bf asymptotic density (of states).}
In general it is not the semicircle density, but is determined by the second moments of 
the matrix elements and it is independent of other details of the distribution.
For independent entries, the variance matrix
\be\label{Sijdef}
      S=( s_{ij})_{i,j=1}^N, \qquad s_{ij} :  =\E |h_{ij}|^2
 \ee
contains all necessary information. For matrices with correlated entries, all relevant second moments are encoded in the 
linear operator 
\[
  \cS[R]: =  \E HRH, \qquad R\in \C^{N\times N}
\]
acting on $N\times N$ matrices. It is one of the key questions in random matrix theory to compute
the asymptotic density $\varrho$ from the second moments; we will see that the answer requires solving 
a system of nonlinear equations, that will be commonly called the {\bf Dyson equation}.
The explicit solution leading to the semicircle law is available only for Wigner matrices, or a little bit
more generally, for ensembles with the property
\be\label{genuine}
\sum_j s_{ij}=1  \qquad \mbox{for any $i$}.
\ee
These are called {\bf generalized Wigner ensembles} and have been introduced in \cite{EYY}.

For invariant ensembles, the self-consistent density $\varrho=\varrho_V$ depends on the potential function $V$. It can be
computed by solving a convex  minimization problem, namely it is the 
the unique minimizer  of the
functional
\[
I(\nu)=
\int_\R V(t) \nu(t)\rd t-
\int_\R \int_\R \log|t-s| \nu(s) \nu(t) \rd t \rd s.
\]
In both cases, under some  mild conditions on the variances $S$ or on the potential $V$, respectively, 
the asymptotic density $\varrho$ is compactly supported.

\subsubsection{Eigenvalues on mesoscopic scales: local laws}

The Wigner semicircle law in the form \eqref{sc} asymptotically determines the number of eigenvalues
in a fixed interval $[a,b]$. The number of eigenvalues in such intervals is comparable with $N$. However, 
keeping in mind the analogy with  the law of large numbers, it is natural to raise the question whether
the same asymptotic relation holds if the length of the interval $[a,b]$ shrinks to zero as $N\to \infty$.
To expect a deterministic answer, the interval should still contain many eigenvalues, but this would be
guaranteed by  $|b-a|\gg 1/N$. This turns out to be correct and the local semicircle law asserts that
\be
   \lim_{N\to\infty} \frac{1}{2N\eta}\#\big\{ i \; : \; \lambda_i\in [E-\eta, E+\eta]\big\}
  = \varrho_{sc}(E)
  \label{sclocal}
\ee
uniformly in $\eta=\eta_N$ as long as $N^{-1+\e}\le \eta_N\le N^{-\e}$ for any $\e>0$ and $E$ is not at the edge, $|E|\ne 2$.
Here we considered
the interval $[a,b]= [E-\eta, E+\eta]$, i.e. we fixed its
 center and viewed its length as
an $N$-dependent parameter.  (The $N^\e$ factors can be improved to some $(\log N)$-power.)

\subsubsection{Eigenvalues on microscopic scales: universality of local eigenvalue statistics}

Wigner's original observation concerned the distribution of 
the distances between consecutive (ordered) eigenvalues, or {\bf gaps}. In the bulk of the spectrum, i.e. in 
the vicinity of a fixed energy level $E$ with $|E|<2$
in case of the semicircle law, 
 the gaps have a typical size of order $1/N$ (at the spectral edge, $|E|=2$, the relevant microscopic
 scale is of order $N^{-2/3}$, but we will not pursue edge behavior in these notes).
 Thus the corresponding rescaled gaps have the form
 \be\label{gi}
      g_i:=   N\varrho(\lambda_i) \big(\lambda_{i+1}-\lambda_i\big),
 \ee
 where $\varrho$ is the asymptotic density, e.g. $\varrho=\varrho_{sc}$ for Wigner matrices.
  Wigner predicted that the fluctuations of the 
gaps are universal and their distribution is given by a new law, the {\it Wigner surmise}. Thus 
there exists a random variable $\xi$, depending only on the symmetry class $\beta=1,2$, such that
\[
   g_i \Longrightarrow \xi
\]
in distribution, for any gap away from the edges, i.e.,  if $\e N \le i\le (1-\e)N$
with some fixed $\e>0$.
\index{Wigner surmise} 
This might  be viewed as the random matrix analogue of the central limit theorem. 
Note that universality is twofold. First, the distribution of $g_i$  is independent of 
the index $i$ (as long as $\lambda_i$ is away from the edges). Second, more importantly, the limiting  gap distribution 
is independent of the distribution of the matrix elements, similarly to the universal character
of the central limit theorem.

However, the gap universality holds much more generally than the semicircle law: the 
rescaled gaps  \eqref{gi} follow the same distribution as the gaps of the GUE or GOE
(depending on the symmetry class) essentially for any  random matrix ensemble
with ``sufficient'' amount of randomness. In particular, it holds for invariant ensembles, as well as
for Wigner type and correlated random matrices, i.e. for  very broad extensions of 
the original  Wigner ensemble. In fact, it holds much beyond the traditional realm of random
matrices; it is conjectured to hold for any random matrix describing a disordered quantum 
system in the {\it delocalized regime}, see Section~\ref{sec:grand} later.

The universality on microscopic scales can also be expressed in terms of the appropriately rescaled {\bf correlation 
functions}. In fact, in this way the formulas are more  explicit.
 First we define the correlation functions. 
\begin{definition}\label{def:corrfn}
Let $p_N(\lambda_1, \lambda_2, \ldots ,\lambda_N)$ be the joint
symmetrized probability distribution of the eigenvalues.  For any $n\ge 1$,
the $n$-point 
correlation function is defined by
\be
  p^{(n)}_N(\la_1, \la_2, \ldots, \la_n): = \int_{\R^{N-n}} p_N( \la_1, \ldots,\la_n, \la_{n+1},
\ldots \la_N) \rd\la_{n+1} \ldots \rd\la_{N}.
\label{pk}
\ee 
\end{definition}

 The significance of the correlation functions is that with their help
one can compute the expectation value of any symmetrized  observable.  For example, 
for any bounded continuous test function $O$ of two variables we have,
directly from the definition of the correlation functions, that 
\be\label{coro}
\frac 1 { N (N-1)} \E \sum_{i \not = j} O(\lambda_i, \lambda_j)  = \int_{\R\times\R}
 O(\lambda_1, \lambda_2) p^{(2)}_N (\lambda_1, \lambda_2) \rd \lambda_1 \rd 
\lambda_2, 
\ee
where the expectation is w.r.t. the probability density $p_N$ or in this case w.r.t. the original random matrix ensemble. 
Similar formula holds for observables of any number of variables.  In particular, the global law \eqref{lime} implies
 that  the one point correlation 
function converges to the asymptotic density
\[
  p^{(1)}_N (x)\rd x \to \varrho(x) \rd x
\]
weakly, since
\[
   \int_\R O(x)  p^{(1)}_N (x)\rd x = \frac{1}{N} \E \sum_i O(\lambda_i) \to \int O(x)\varrho(x)\rd x.
\]

Correlation functions are difficult to compute in general, even if the joint density function $p_N$ is 
explicitly given as in the case of the invariant ensembles \eqref{expli}. Naively one may think 
that  computing
the correlation functions in this latter case  boils down to an elementary calculus exercise by integrating out
all but a few variables. However, that task is complicated.

As mentioned, one may view the joint density of eigenvalues of invariant ensembles \eqref{expli} as a Gibbs measure of
a log-gas and here $\beta$ can be any positive number (inverse temperature). The universality of
correlation functions is a valid question for all $\beta$-log-gases
 that has been positively answered in \cite{BouErdYau2014, BouErdYau2012, BouErdYau2014-2, BekFigGui2015, Shc2014} 
 by showing that for a sufficiently smooth potential $V$ (in fact $V\in C^{4}$ suffices) the correlation functions
 depend only on $\beta$ and are independent of $V$. We will not pursue general invariant ensembles in these notes.

The logarithmic interaction is of long range, so the system \eqref{mf} is strongly correlated
and standard methods of statistical mechanics to compute correlation functions cannot be applied. 
The computation is quite involved even for the simplest Gaussian case, and it relies on sophisticated 
identities involving Hermite orthogonal polynomials. These calculations have been developed
by Gaudin, Mehta and Dyson in the 60's and can be found, e.g. in Mehta's book \cite{Meh1991}. Here
we just present the result for the most relevant $\beta=1, 2$ cases. 

 We  fix an energy $E$ in the bulk, i.e., $|E| < 2$, and we  rescale 
the correlation functions by a factor $N\varrho$  around $E$ to make the typical distance between neighboring
eigenvalues 1.  These rescaled correlation functions then have a universal limit:

\begin{theorem}
For GUE ensembles, the rescaled correlation functions converge to the determinantal formula with the sine kernel, $S(x): =  \frac{\sin \pi x}{\pi x}$,
 i.e. 
\begin{align}\label{sineres}
  \frac{1}{[\varrho_{sc}(E)]^n}
 p_N^{(n)}  \Big( E+ \frac{\al_1}{N\varrho_{sc}(E)}, & E + \frac{\al_2}{N\varrho_{sc}(E)},
 \ldots ,E+ \frac{\al_n}{N\varrho_{sc}(E)}\Big) \\
& \rightharpoonup  q^{(n)}_{\text{\tiny GUE}}\left({\boldsymbol  \alpha}\right):=
\det \big(  S (\al_i - \al_j)\big)_{i,j=1}^n \nonumber
\end{align}
as  weak convergence of functions in the variables ${\boldsymbol  \alpha}= (\al_1, \ldots, \al_n)$.
\end{theorem}

 Formula \eqref{sineres} holds for the GUE case.  The corresponding expression for GOE  is
more involved \cite{Meh1991, AndGuiZei2010} 
\be \label{pGOE}
q^{(n)}_{\text{\tiny GOE}}\left({\boldsymbol  \alpha}\right):=
\det \big( K (\al_i - \al_j)\big)_{i,j=1}^n, \qquad K(x) := \begin{pmatrix} S(x) & S'(x) \cr 
-\frac{1}{2}\sgn (x)+\int_0^x S(t)\rd t
 & S(x) 
\end{pmatrix}.
\ee
Here the determinant is understood as  the trace of the  {\it quaternion determinant}
\index{Quaternion determinant} after the canonical correspondence between
 quaternions $a\cdot{\bf 1} +b\cdot{\bf i} + c\cdot {\bf j} + d\cdot{\bf k}$, $a,b,c,d\in \C$, and $2\times 2$
 complex matrices given by
 \[
    {\bf 1} \leftrightarrow\begin{pmatrix} 1 &0\cr 0&1 \end{pmatrix}\qquad
    {\bf i} \leftrightarrow\begin{pmatrix} i &0\cr 0&-i \end{pmatrix}\qquad
    {\bf j} \leftrightarrow \begin{pmatrix} 0 &1\cr -1&0 \end{pmatrix}\qquad
    {\bf k} \leftrightarrow\begin{pmatrix} 0 &i\cr i&0 \end{pmatrix}.
 \]

 Note that the limit in \eqref{sineres} is universal in the  sense that it is independent of the energy $E$. 
 However, universality also holds in a much stronger sense, namely that the local statistics
 (limits of rescaled correlation functions) depend only on the symmetry class, i.e. on $\beta$,
 and are independent of any other details. In particular, they are always given by the sine kernel \eqref{sineres}
 or \eqref{pGOE} not only for the Gaussian case but for any Wigner matrices with arbitrary 
 distribution of the matrix elements, as well as for any invariant ensembles with arbitrary potential $V$.
 This is the {\bf Wigner-Dyson-Mehta (WDM) universality conjecture}, formulated precisely in Mehta's book 
 \cite{Meh1991} in the late 60's. 
 
 The WDM conjecture for invariant ensembles has been in the focus 
 of very intensive research on orthogonal polynomials with general weight function
 (the Hermite polynomials arising in the Gaussian setup have Gaussian weight function).
 It motivated the development of the Riemann-Hilbert method \cite{FokItcKit1992},   that was originally
  brought into this subject by Fokas, Its and Kitaev \cite{FokItcKit1992}, and 
the universality of eigenvalue statistics was established  for  large  classes of invariant ensembles
by  Bleher-Its \cite{BleIts1999} and by Deift and collaborators \cite{Dei1999, DeiGio2007-2, DeiKriMcLVen1999}. 
The key element of this success was
that invariant ensembles, unlike Wigner matrices, have explicit formulas \eqref{expli}
for the joint densities of the eigenvalues.  With the help of the Vandermonde structure of these
 formulas, one may express the eigenvalue correlation functions as determinants 
whose  entries are given by functions of 
orthogonal polynomials. 
 
 For Wigner ensembles, 
there  are  no explicit formulas for the joint density of eigenvalues or for the correlation functions statistics 
 and  the WDM conjecture was open for almost fifty years with virtually 
no  progress. The first significant advance  in this direction  was made by Johansson \cite{Joh2001}, who  proved  the universality for {\it complex} Hermitian matrices  under the assumption that
the common distribution of the matrix entries has a substantial Gaussian component, i.e., the random matrix $H$ is of the form 
$H= H_0+ a H^G$ where  $H_0$ is a  general  Wigner matrix,  
$H^G$ is the   GUE matrix,  and $a$ is a certain,  not too small, positive constant  independent of $N$. 
His proof relied on an explicit formula by Br\'ezin and Hikami \cite{BreHik1996, BreHik1997} that  uses a certain version of 
the Harish-Chandra-Itzykson-Zuber formula \index{Harish-Chandra-Itzykson-Zuber formula}
\cite{ItzZub1980}. These formulas are available
for the complex Hermitian case only, which   restricted the method to this symmetry class.

 \begin{exercise}
Verify formula \eqref{coro}.
\end{exercise}
 
 \subsubsection{The three step strategy}\label{sec:3step}

 The WDM conjecture  in full generality has recently been  resolved  by a new approach called the {\bf three step strategy}
 that has been developed in a series of papers by Erd{\H o}s, Schlein, Yau and Yin between 2008 and 2013
 with a parallel development by Tao and Vu.
 A detailed presentation of this method can be found in \cite{ErdYau2017}, while a 
 shorter summary was presented in 
  \cite{ErdYau2012-2}.
 
This approach consists of the following three steps:  \index{Three-step strategy}

\noindent
{\bf Step 1.}  {\bf Local semicircle law:} \index{Local semicircle law}
It provides an a priori estimate  showing that the density of eigenvalues of
generalized Wigner matrices 
is  given by
the semicircle law at  very small microscopic scales, i.e., down to spectral intervals
that contain $N^\e$ eigenvalues. 

\noindent
{\bf Step 2.}
{\bf Universality for
Gaussian divisible ensembles:}  It proves that the local statistics of \index{Gaussian divisible ensemble}
 {\it Gaussian divisible ensembles} $H_0+ a H^G$ 
are the same as those of the Gaussian ensembles $H^G$ as long as $a \ge N^{-1/2+ \e}$, i.e.,
 already for very small $a$.

\noindent
{\bf Step 3.}  {\bf Approximation by a Gaussian divisible ensemble:}
It is a type  of  ``density argument'' 
that  extends the local spectral universality 
 from  Gaussian divisible ensembles to all Wigner ensembles.


The conceptually novel  point  is  Step 2. The eigenvalue distributions of 
the Gaussian divisible ensembles, written in the form 
$e^{-t/2} H_0+ \sqrt {1 - e^{-t}}H^{\rm G}$,   are the  same as that of  
the solution of a {\it matrix valued Ornstein-Uhlenbeck (OU) process~$H_t$ \index{Ornstein-Uhlenbeck process}}
\be\label{OU}
   \rd H_t = \frac{\rd {\bf B}_t}{\sqrt{N}} -\frac{1}{2}H_t \rd t, \qquad H_{t=0}=H_0,
\ee
 for any time $t\ge 0$, where ${\bf B}_t$ is a matrix valued standard Brownian motion
 of the corresponding symmetry class
(The OU process is preferable over its rescaled version $H_0+ a H^G$
since it keeps the variance constant).
 Dyson \cite{Dys1962} observed  half a century ago    that the dynamics of the 
 eigenvalues $\lambda_i=\lambda_i(t)$ of $H_t$ is  
 given by  an interacting stochastic
particle system,  called the {\it Dyson Brownian motion (DBM)},
where the eigenvalues  are the particles:
\be\label{dbmm}
  \rd \lambda_i =\sqrt{\frac{\beta}{2}} \frac{1}{\sqrt{N}} \rd B_i  + \Big( -\frac{\lambda_i}{2} +\frac{1}{N}\sum_{j\ne i}\frac{1}{\lambda_i-\lambda_j}\Big) \rd t, \qquad
  i=1,2,\ldots, N.
\ee
Here $\rd B_i$ are independent white noises.

 \index{Dyson Brownian motion (DBM)}
In addition, the invariant  measure of this 
dynamics  is  exactly the eigenvalue  distribution of GOE or GUE, i.e. \eqref{expli} with $V(x)=\frac{1}{2}x^2$.
 This invariant measure is thus a Gibbs measure 
of point particles in one dimension  interacting via  a long range
logarithmic potential.  In fact, $\beta$ can be any positive parameter, the corresponding DBM  \eqref{dbmm}
may be studied even if there is no invariant matrix ensemble behind.
Using a heuristic physical argument, Dyson  remarked \cite{Dys1962} that the DBM reaches 
its ``local equilibrium'' on a short time scale $t \gtrsim N^{-1}$. 
We call this {\it Dyson's conjecture}, \index{Dyson conjecture} although
it was rather an intuitive physical picture than an exact mathematical statement. 
Step 2 gives a precise mathematical meaning of this vague idea. 
The key point is that by applying 
local relaxation to {\it all} initial states (within a reasonable class) simultaneously, Step 2
generates a large set of random matrix ensembles for which universality holds. 
For the purpose of universality,  this set is sufficiently  dense so that any Wigner matrix $H$ is sufficiently close to a Gaussian divisible ensemble  of the form $e^{-t/2} H_0+ \sqrt {1 - e^{-t}}H^{\rm G}$ with a suitably chosen  $H_0$. 

We note that in the Hermitian case, Step 2 can be circumvented  by using the 
Harish-Chandra-Itzykson-Zuber formula.  This approach was followed by Tao and Vu \cite{TaoVu2011}
who gave an alternative proof of universality for Wigner matrices in the Hermitian symmetry class 
as well as for the real symmetric class but only  under a certain moment matching condition.

The three step strategy has been refined and streamlined in the last years. By now it has reached a stage when  the content of
Step 2 and Step 3 can be presented as a very general ``black-box'' result that is model independent assuming  that Step 1, the local law, holds.
The only model dependent ingredient  is the local law. Hence to prove local spectral universality  for a new ensemble, one needs to verify 
the local law. Thus in these lecture notes we will focus on the recent  developments in the direction of the local laws.

We will discuss generalizations of the original Wigner ensemble to relax the basic conditions {\it ``independent, identically distributed''}.
First we drop the identical distribution and allow the variances $s_{ij}=\E |h_{ij}|^2$ to vary. The simplest class is
the {\bf generalized Wigner matrices}, defined in \eqref{genuine}, which still leads to  the Wigner semicircle law.
 The next level of generality is to allow arbitrary matrix of variances $S$. The density of states is not the semicircle any more
 and we need to solve a genuine {\bf vector Dyson equation} to find the answer. The most general case discussed in these notes
 are correlated matrices, where different matrix elements have nontrivial correlation that leads to a {\bf matrix Dyson equation}.
    In all cases we  keep the mean field
 assumption, i.e. the typical size of the matrix elements is $|h_{ij}|\sim N^{-1/2}$.
 Since Wigner's vision on the universality of local eigenvalue statistics predicts the same universal behavior for 
 a  much larger class of hermitian random matrices (or operators), it is fundamentally important
 to extend the validity of the mathematical  proofs as much as possible beyond the Wigner case. 
 
We remark that  there are several other directions to extend the Wigner ensemble that we will not discuss here in details, we
just mention some of them with a few references, but we do not aim at completeness; apologies for any omissions.
First, in these notes we will assume very high moment conditions on the matrix elements. These make the proofs
easier and the tail probabilities of the estimates stronger. Several works have focused on {\it lowering the moment
assumption} \cite{Joh2012, GotNauTik2015, Aggarwal2016} and even considering {\it heavy tailed distributions}
\cite{BenPech2014, BorGui02016}. An important special case is the class of {\it  sparse matrices} such as  adjacency matrix
of Erd{\H o}s-R\'enyi random graphs and $d$-regular graphs \cite{ErdKnoYauYin2013-2,
ErdKnoYauYin2012, BauHuaKnoYau2015, HuaLanYau2015, BauKnoYau2015}.
Another direction is to remove the condition that the matrix elements are centered;
this ensemble often goes under the name of {\it deformed Wigner matrices}. One typically separates the expectation and
writes $H=A + W$, where $A$ is a deterministic matrix and $W$ is a Wigner matrix with centered entries.
Diagonal deformations  ($A$ is diagonal) are easier to handle, this class was considered even for a large diagonal
in  \cite{ORourkeVu2014, KnoYin2013, KnoYin2014, LeeSchSteYau2016}. The general $A$ was considered in \cite{HeKnowlesRosenthal2016}.
Finally, a very challenging direction is to depart from the mean field condition, i.e. allow some matrix elements to be much bigger than
$N^{-1/2}$. The ultimate example is the {\it random band matrices} that goes towards the random Schr\"odinger operators
\cite{Sch2009, Sod2010, ErdKno2011, ErdKnoYauYin2013, ErdKnoYau2013, TShc2014, TShc2014-2, Shc2015, BaoErd2016}.

\subsubsection{User's guide}

These lecture notes were intended to Ph.D students and postdocs with general interest in analysis and probability;
we assume knowledge of these areas on  a beginning Ph.D. level.
The overall style is informal, the proof of many statements are only sketched or indicated. Several technicalities
are swept under the rug -- for the precise theorems the reader should consult with the original papers.
We  emphasise conveying the main ideas in a colloquial way.

In Section~\ref{sec:tools} we collected basic tools  from analysis such as Stieltjes transform and resolvent.
We also introduce the semicircle law. We outline the moment method that was traditionally important in random
matrices, but we will not rely on it in these notes, so this part can be skipped. 
In Section~\ref{sec:resolvent} we outline
the main method to obtain local laws, the resolvent approach and we explain in an informal way
its two constituents; the probabilistic and deterministic parts.
In Section~\ref{sec:compl} we introduce four models of Wigner-like ensembles with increasing complexity
and we informally explain the novelty and the  additional complications for each model.
Section~\ref{sec:motivation}  on the physical motivations to study
these models is a  detour. Readers interested only in the mathematical aspects
  may skip this section. Section~\ref{sec:results} contains our main results on the local law formulated in a mathematically
  precise form. We did not aim at presenting the strongest results and the weakest possible conditions;
  the selection was guided to highlight some key phenomena. Some consequences of these local laws
  are also presented with sketchy proofs. Section~\ref{sec:vector} and \ref{sec:matrix} contain
  the main mathematical part of these notes, here we give a more detailed  analysis of
  the vector and the matrix Dyson equation and their stability properties.
   In these sections we aim at rigorous presentation although
 not every proof contains all details.
  Finally, in Section~\ref{sec:ideas} we present the main ideas of the proof of the local laws based
on  stability results on the Dyson equation.

These lecture notes are far from being  a comprehensive text on random matrices.
Many key issues are left out and even those we discuss will be presented in their simplest form.
For more interested readers, we refer to the recent book \cite{ErdYau2017}
that focuses on the three step strategy and discusses all steps in details.
For readers
interested in other aspects of random matrix theory,  in addition to 
the classical book of Mehta \cite{Meh1991}, several excellent works are available that
present random matrices in a  broader scope. The  books by Anderson, Guionnet and Zeitouni \cite{AndGuiZei2010}
and Pastur and Shcherbina  \cite{PasShc2011} contain extensive material starting from the basics.
Tao's book \cite{Tao2012} provides a different aspect to this subject and is self-contained  
as a graduate textbook.   Forrester's monograph \cite{For2010} is a handbook
for any explicit formulas related to random matrices. Finally, \cite{AkeBaiDi-2011}
is an excellent comprehensive overview of diverse applications of random matrix theory
in mathematics, physics, neural networks and engineering.


{\it Notational conventions.} In order to focus on the essentials, we will
not follow the dependence of various constants on different parameters. In particular, 
we will use the generic letters $C$ and $c$ to denote positive constants, whose values
may change from line to line and which may depend on some fixed basic parameters of
the model. For two positive quantities $A$ and $B$, we will write $A\lesssim B$ to indicate that there
 exists a constant $C$
such that $A\le CB$. If $A$ and $B$ are comparable in the sense that $A\lesssim B$ and $B\lesssim A$,
then we write $A\sim B$.   In informal explanations, we will often use $A\approx B$ which indicates
closeness in a not precisely specified sense.
We  introduce  the notation $\llbracket A, B\rrbracket : = \Z \cap [A, B]$
for the set of integers between any two real numbers 
 $A<B$. 
We will usually denote vectors in $\C^N$ by boldface letters; $\bx=(x_1, x_2, \ldots, x_N)$.


{\it Acknowledgement. } A special thank goes to Torben Kr\"uger for many  discussions and suggestions
on  the presentation of this material as well as for his careful proofreading and  invaluable comments.
I am also very grateful to both referees for many constructive suggestions, as well as to Ian Morrison
for the excellent editing work.

\section{Tools}\label{sec:tools}

\subsection{Stieltjes transform}

In this section we introduce our basic tool, the {\bf Stieltjes transform} of a measure.
We denote the open upper half  of the complex  plane by
\[
  \Cp : = \{ z\in \C\; : \;\im z>0\,.\} 
\]

\begin{definition}\label{def:sit}
Let $\mu$ be a Borel probability measure on $\R$. Its Stiltjes transform at a  {\bf spectral parameter} $z\in \Cp$ is defined by
\be\label{StDef}
   m_\mu(z): = \int_\R \frac{\rd\mu(x)}{x-z}.
\ee
\end{definition}

\begin{exercise} The following three properties are straightforward to check:
\begin{itemize}
\item[i)]
The Stieltjes transform $m_\mu(z)$ is analytic on $\Cp$ and it maps $\Cp$ to $\Cp$, i.e. $\im m_\mu(z)>0$.
\item[ii)]  We have $-i\eta m_\mu(i\eta) \to 1$ as $\eta\to \infty$.
\item[iii)]  We have the bound
\[
      |m_\mu(z)|\le \frac{1}{\im z}.
\]
\end{itemize}
\end{exercise}
In fact,  properties i)-ii)  characterize the Stieltjes transform in a sense that if a function $m:\Cp\to\Cp$ satisfies i)--ii), then 
there exists a probability measure $\mu$ such that $m=m_\mu$ (for the proof, see e.g. Appendix B of \cite{Weidmann}; it
is also called the Nevanlinna's representation theorem).

From the Stieltjes transform one may recover the measure:

\begin{lemma}[Inverse Stieltjes transform]\label{lm:st}
Suppose that $\mu$ is a probability measure on $\R$ and let $m_\mu$ be its Stieltjes transform. Then
for any $a<b$ we have
\[
  \lim_{\eta\to 0} \frac{1}{\pi}\int_a^b \im m_\mu(E+i\eta) \rd E = \mu(a,b) +\frac{1}{2}\big[ \mu(\{ a\})+ \mu(\{ b\})\big]
  \]
  Furthermore, if $\mu$ is absolutely continuous with respect to  the Lebesgue measure, i.e. $\mu(\rd E)= \mu(E)\rd E$ with some
  density function $\mu(E)\in L^1$, then 
  \[
     \frac{1}{\pi} \lim_{\eta\to 0+} \im m_\mu(E+i\eta) \to \mu(E)
  \]
  pointwise for almost every $E$.
\end{lemma}

In particular, Lemma~\ref{lm:st} guarantees that $m_\mu = m_\nu$ if and only of $\mu=\nu$, i.e. the Stieltjes transform
uniquely characterizes the measure. Furthermore, pointwise convergence of a sequence of Stieltjes transforms 
is equivalent to weak convergence of the measures. More precisely, we have

\begin{lemma}\label{lm:se}
Let $\mu_N$ be a  sequence of probability measures and let $m_N(z) =m_{\mu_N}(z)$ be their Stieltjes transforms.
Suppose that 
\[
  \lim_{N\to \infty} m_N(z) =: m(z)
\]
exists for any $z\in \Cp$ and $m(z)$ satisfies property ii), i.e. $-i\eta m(i\eta) \to 1$ as $\eta\to \infty$.
Then there exists a probability measure $\mu$ such that $m=m_\mu$ and $\mu_N$ converges  to $\mu$ 
in distribution.
\end{lemma}
The proof can be found e.g.  in \cite{GeronimoHill} and it relies on Lemma~\ref{lm:st} and Montel's theorem.
The converse of  Lemma \ref{lm:se} is trivial: if the sequence $\mu_N$ converges in distribution to a probability measure $ \mu$,
then clearly $m_N(z)\to m_\mu(z)$ pointwise, since the Stieltjes transform for any fixed $z\in \Cp$ is just the integral of the continuous
bounded function $x\to (x-z)^{-1}$. Note that the additional condition ii) is a compactness (tightness) condition, it prevents that 
part of the measures $\mu_N$ escape to infinity in the limit.

All these results are very similar to the Fourier transform (characteristic function) 
\[
   \phi_\mu(t): = \int_\R e^{-itx}\mu(\rd x)
\] 
of a probability measure. In fact, there is a direct connection 
between them; 
\[
   \int_0^\infty e^{-\eta t} e^{itE} \phi_\mu(t)\rd t = i \int_\R  \frac{\rd\mu(x)}{x-E-i\eta}  = i m_\mu(E+i\eta)
\]
for any $\eta>0$ and $E\in \R$.
In particular, due to the regularizing factor $e^{-t\eta}$, the large $t$ behavior of the Fourier transform $\phi(t)$ is closely related to the small 
$\eta\sim 1/t$ behavior of the Stieltjes transform.

Especially important is the imaginary part of the Stieltjes transform since
\[
   \im m_\mu(z) =  \int_\R \frac{\eta}{|x-E|^2+\eta^2} \mu(\rd x), \qquad z=E+i\eta,
\]
which can also be viewed as the convolution of $\mu$ with the Cauchy kernel on scale~$\eta$:
\[
    P_\eta(E) = \frac{\eta}{E^2+\eta^2},
\]
indeed
\[
     \im m_\mu(E+i\eta) = (P_\eta \star \mu)(E).
\]
Up to a normalization $1/\pi$, the Cauchy kernel is an approximate delta function on scale $\eta$. Clearly
\[ 
    \int_\R \frac{1}{\pi}\;  P_\eta(E)\rd E =1
\]
and the overwhelming majority of its mass is supported on scale $\eta$:
\[
   \int_{|E|\ge K\eta} \frac{1}{\pi}\; P_\eta(E) \le \frac{2}{K}
\]
for any $K$. Due to standard properties of the convolution,
 the {\it moral of the story} is that {\bf $\im m_\mu (E+i\eta)$ resolves the measure $\mu$ on a scale $\eta$ around an energy $E$}. 
 
Notice that the small $\eta$ regime is critical; it is the regime where the integral in the definition of the Stieltjes transform  \eqref{StDef}
becomes more singular, and properties of the integral more and more depend on the local smoothness properties of the measure.
In general, the regularity of the measure $\mu$ on some scales $\eta>0$ is directly related to the Stieltjes transform $m(z)$ with
$\im z\approx \eta$.

The Fourier transform $\phi_\mu(t)$  of $\mu$ for large $t$ also characterizes the local behavior of the measure $\mu$
on scales $1/t$, 
We will nevertheless work with the Stieltjes transform since for hermitian matrices  (or self-adjoint operators in general) it is directly related to the 
resolvent, it is relatively easy to handle and it has many convenient properties.

\begin{exercise}
Prove Lemma \ref{lm:st} by using Fubini's theorem and Lebesgue density theorem.
\end{exercise}

\subsection{Resolvent}\label{sec:resolve}

Let $H=H^*$ be a hermitian matrix, then its resolvent at spectral parameter $z\in \Cp$ is defined
as
\[
   G=G(z) = \frac{1}{H-z}, \qquad z\in \Cp.
\]
In these notes, the spectral parameter $z$ will always be in the upper half plane, $z\in \Cp$. We usually follow the
convention that $z=E+i\eta$, where $E=\re z$ will often be referred as ``energy'' alluding to the quantum mechanical
interpretation of $E$.

Let $\mu_N$ be the normalized {\it empirical measure of the eigenvalues} of $H$:
\[
  \mu_N(\rd x) = \frac{1}{N}\sum_{i=1}^N \delta(\lambda_i-x) \rd x.
\]
Then clearly the normalized trace of the resolvent is
\[
\frac{1}{N}\tr G(z)  =\frac{1}{N}\sum_{i=1}^N\frac{1}{\lambda_i -z}= \int_\R \frac{\mu_N(\rd x)}{x-z} = m_{\mu_N}(z)=:m_N(z)
\]
exactly the Stieltjes transform of the empirical measure. This relation justifies why we focus on the Stieltjes transform;
based upon Lemma~\ref{lm:se}, if we could identify the (pointwise) limit of $m_N(z)$, then the asymptotic eigenvalue density $\varrho$
would be given by the inverse Stieltjes transform of the limit. 

Since $\mu_N$ is a discrete (atomic) measure on small $(1/N)$ scales, it may behave very badly (i.e.  it is strongly fluctuating and  may blow up)
for $\eta$ smaller than $1/N$, depending on whether there happens to be an eigenvalue in an $\eta$-vicinity of $E=\re z$.
Since the eigenvalue spacing is (typically) of order $1/N$,  for $\eta\ll 1/N$ there is no  approximately deterministic
(``self-averaging'') behavior of $m_N$. 
 However, as long as $\eta\gg 1/N$, we may hope a {\it law of large number phenomenon}; this would be
 equivalent to the fact that the eigenvalue density does not have much fluctuation above its inter-particle
  scale $1/N$. The local law on $m_N$ down to the smallest possible (optimal) scale $\eta\gg 1/N$ will
  confirm this hope.

In fact, the resolvent carries much more information than merely its trace. 
In general the resolvent of a hermitian matrix  is a very rich object: it gives information on the eigenvalues and 
eigenvectors for energies near the real part of the spectral parameter.
For example, 
by spectral decomposition we have
\[
    G(z) = \sum_i \frac{|\bu_i\rangle \langle \bu_i|}{\lambda_i-z}
 \]
 where $\bu_i$ are the ($\ell^2$-normalized) eigenvectors associated with $\lambda_i$. 
 (Here we used the Dirac notation $|\bu_i\rangle \langle \bu_i|$ for the orthogonal projection
 to the one-dimensional space spanned by $\bu_i$.)
  For example, the diagonal matrix elements
 of the resolvent at $z$ are closely related to the eigenvectors  with eigenvalues near $E=\re z$:
 \[
     G_{xx} = \sum_i \frac{ |\bu_i(x)|^2}{\lambda_i-z}, \qquad \im G_{xx}  = \sum_i \frac{\eta}{ |\lambda_i-E|^2+\eta^2} |\bu_i(x)|^2.
      \]    
  Notice that for very small $\eta$, the factor $\eta/(  |\lambda_i-E|^2+\eta^2)$ effectively reduces the sum from all $i=1,2,\ldots , N$ to
  those indices where $\lambda_i$ is $\eta$-close to $E$; indeed this factor changes from the very large value $1/\eta$ to 
  a very small value $\eta$ as $i$ moves away. Roughly speaking
  \[
     \im G_{xx}  = \sum_i \frac{\eta}{ |\lambda_i-E|^2+\eta^2} |\bu_i(x)|^2 \approx  \sum_{i: |\lambda_i-E|\lesssim\eta} \frac{\eta}{ |\lambda_i-E|^2+\eta^2} |\bu_i(x)|^2.
  \]
  
  This idea can be made rigorous at least as an upper bound on each summand.
A physically important consequence will be that one may directly obtain $\ell^\infty$ bounds on the eigenvectors: for any fixed $\eta>0$ we have
 \be\label{delocbound}
     \| \bu_i\|_\infty^2:= \max_x |\bu_i(x)|^2  \le \eta \cdot \max_x \max_{E\in \R}   \im G_{xx}(E+i\eta).   
 \ee
 In other words, if we can control diagonal elements of the resolvent on some scale $\eta=\im z$, then we can prove
 an $\sqrt{\eta}$-sized bound on the max norm of the eigenvector. The strongest result is always the smallest possible scale.
 Since the local law will hold down to scales $\eta\gg 1/N$, in particular we will be able to establish that $\im G_{xx}(E+i\eta)$
 remains bounded as long as $\eta\gg 1/N$, thus we will prove the {\bf complete delocalization}
 of the eigenvectors:
 \be\label{del}
    \| \bu_i\|_\infty \le \frac{N^\e}{\sqrt{N}}
\ee
for any $\e>0$ fixed, independent of $N$, and with very high probability. 
 Note that the bound \eqref{del} is optimal (apart from the  $N^\e$ factor) since clearly
 \[
    \| \bu \|_\infty \ge \frac{\| \bu\|_2}{\sqrt{N}}
\]
for any $\bu\in \C^N$. 

We also    note  that if $\im G_{xx}(E+i\eta)$ can be controlled only for energies
in a fixed subinterval $I\subset\R$, e.g. the local law holds only for all  $E\in I$, the we
can conclude complete delocalization for those eigenvectors whose eigenvalues lie in~$I$.

\subsection{The semicircle law for Wigner matrices  via the moment method}

This section introduces the traditional moment method to identify the semicircle law. We included this material 
for historical relevance, but it will not be needed later hence it can be skipped at first reading.

For large $z$ one can expand $m_N$ as follows
\be
  m_N(z) =  \frac{1}{N}\tr \frac{1}{H-z} =
 -\frac{1}{Nz}\sum_{m=0}^\infty \tr \Big( \frac{H}{z}\Big)^m,
\label{mexp}
\ee
so after taking the expectation, we need to compute traces of high moments of $H$:
\be\label{EmN}
   \E\, m_N(z) = \sum_{k=0}^\infty z^{-(2k+1)} \frac{1}{N} \E \tr H^{2k}.
 \ee
Here we tacitly used that the contributions of   odd powers are  algebraically zero, which clearly holds  at least 
if we assume that $h_{ij}$ have symmetric 
distribution for simplicity. Indeed, in this  case $H^{2k+1}$ and $(-H)^{2k+1}$ have the same
distribution, thus 
\[
  \E \tr H^{2k+1} =  \E \tr (-H)^{2k+1}  = - \E \tr H^{2k+1} .
\]

The computation of even powers, $\E \tr H^{2k}$, reduces to a combinatorial problem. Writing out
\[
  \E \tr H^{2k} =  \sum_{i_1, i_2, \ldots i_{2k}} \E  h_{i_1i_2} h_{i_2 i_3} \ldots h_{i_{2k} i_1},
\]
one notices that, by $\E h_{ij}=0$, all those terms are zero 
 where at least one  $h_{i_j i_{j+1}}$ stands alone, i.e. is not paired 
with itself or its conjugate. This restriction poses a severe constraint on the  relevant index sequences $i_1, i_2, \ldots , i_{2k}$.
For the terms where an exact pairing of all the $2k$ factors is available, we can use $\E |h_{ij}|^2 = N^{-1}$ to see
that all these terms contribute by $N^{-k}$.  There are terms where three or more  $h$'s coincide, giving rise to higher 
moments of $h$, but their combinatorics is of lower order.
Following Wigner's classical calculation  (called the {\bf moment method}, 
see e.g. \cite{AndGuiZei2010}), one needs to compute the number of relevant index sequences that give rise to a perfect pairing and
 one finds that the leading term is given by the Catalan numbers, i.e. 
\be\label{cat}
          \frac{1}{N}\E \tr H^{2k} = \frac{1}{k+1}{\binom{2k}{k}} + O_k\big(\frac{1}{N}\big).
\ee
Notice that the $N$-factors cancelled out in the leading term.

Thus, continuing \eqref{EmN}
 and neglecting the error terms, we get
\be
  \E\, m_N(z) \approx  - \sum_{k=0}^\infty \frac{1}{k+1}{\binom{2k}{k}} z^{-(2k+1)},
\label{anal}
\ee
which, after some calculus, can be identified as the  Laurent 
 series of the function $\frac{1}{2}(-z+\sqrt{z^2-4})$. 
The approximation becomes exact in the $N\to\infty$ limit. 
Although the expansion \eqref{mexp} is valid
only for large $z$, given that the limit is an analytic function of $z$, one
can extend the relation
\be\label{11}
   \lim_{N\to\infty}\E m_N(z) = \frac{1}{2}(-z+\sqrt{z^2-4})
\ee
by analytic continuation to the whole upper half plane $z=E+i\eta$, $\eta>0$.
It is an easy exercise to see that this is exactly the Stieltjes transform of
the semicircle density, i.e.,
\be
  m_{sc}(z):=\frac{1}{2}(-z+\sqrt{z^2-4}) = \int_\R \frac{\varrho_{sc}(x)\rd x}{x-z}, \qquad \varrho_{sc}(x) =\frac{1}{2\pi}\sqrt{ (4-x^2)_+}.
\label{msc}
\ee
The square root function is chosen with a branch cut in the segment $[-2, 2]$
so that $\sqrt{z^2-4}\sim z$ at infinity. This guarantees that 
 $\im m_{sc}(z)>0$ for $\im z>0$.

 \begin{exercise}
 As a simple calculus exercise, verify \eqref{msc}. Either use integration by parts, or compute the moments
 of the semicircle law and verify that they are given by the Catalan numbers, i.e.
 \be\label{cat11}
     \int_\R x^{2k} \varrho_{sc}(x)\rd x = \frac{1}{k+1}{\binom{2k}{k}}.
  \ee
 \end{exercise}

Since the Stieltjes transform identifies the measure uniquely, and pointwise convergence
of Stieltjes transforms implies weak convergence of measures, we obtain 
\be
   \E \,\varrho_N(\rd x) \wkto \varrho_{sc}(x)\rd x.
\label{global:sc}
\ee
The relation \eqref{11} actually holds with high probability, that is,  for any $z$ with $\im z>0$, 
\be
   \lim_{N\to\infty} m_N(z) = \frac{1}{2}(-z+\sqrt{z^2-4}),
\label{limst}
\ee
in probability, implying a similar strengthening  of the convergence in
\eqref{global:sc}.  In the next sections  we will prove this limit  with an effective error term
via the resolvent method.


 The semicircle law can be identified in many different ways. The {\it  moment method} sketched above
  utilized the fact that
the moments of the semicircle density are given by the Catalan numbers
\eqref{cat11}, which also emerged as the normalized traces of powers of $H$, see
\eqref{cat}. The {\it resolvent method} relies on the fact that $m_N$
approximately satisfies a self-consistent equation,
 \be\label{money}
 m_N(z) \approx -\frac{1}{z+m_N(z)},
 \ee
that is very close to the quadratic equation that $m_{sc}$  from \eqref{msc} exactly satisfies:
\be\label{exact}
     m_{sc}(z) = -\frac{1}{z+m_{sc}(z)}.
\ee
Comparing these two equations, one finds that $m_N(z) \approx m_{sc}(z)$. Taking inverse
Stieltjes transform, one concludes the semicircle law. In the next section we give more details on 
\eqref{money}.

In other words, in the resolvent method the semicircle 
density emerges via a specific relation for its Stieltjes transform. The key relation \eqref{exact} 
is the simplest form of the {\bf Dyson equation}, or a {\bf self-consistent equation} for the trace of the resolvent:
later we will see  a Dyson equation for the entire resolvent.
It turns out that the resolvent  approach allows us to perform a much more precise
analysis than the moment method, especially in the short scale regime, where $\im z$ approaches
to 0 as a function of $N$. Since the
Stieltjes transform of a measure at spectral parameter $z=E+i\eta$
essentially identifies the measure around $E$ on scale $\eta>0$,
a precise understanding of $m_N(z)$ for small $\im z$ will
yield a local  version of the semicircle law.

\section{The resolvent method}\label{sec:resolvent}

In this section we sketch the two basic steps of the resolvent method for the simplest Wigner case but we will already
make remarks preparing for the more complicated setup. The first
step concerns the derivation of the approximate equation \eqref{money}. This is a probabilistic step since
$m_N(z)$ is a random object and even in the best case \eqref{money} can hold only with high probability.
In the second step  we compare the approximate equation \eqref{money} with the exact equation  \eqref{exact}
to conclude that $m_N$ and $m_{sc}$ are close. We will view \eqref{money} as a perturbation 
of \eqref{exact}, so this step  is about a stability property of the exact equation and it is a deterministic problem.

\subsection{Probabilistic step}

   There are essentially two ways to obtain \eqref{money}; either by Schur complement formula
or by cumulant expansion.  Typically the Schur method gives more precise results since it 
can be  easier  turned into a full asymptotic expansion, but it heavily relies on the independence
of the matrix elements and that the resolvent of $H$ is essentially diagonal. 
We now discuss these methods separately.

\subsubsection{Schur complement method}\label{sec:schur}

The basic input is the following well-known formula from linear algebra:

\begin{lemma}[Schur formula]\label{2x2}
Let $A$, $B$, $C$ be $n\times n$, $m\times n$ and $m\times m$
 matrices. We define $(m+n)\times (m+n)$
 matrix $D$ as 
\be 
D:=\begin{pmatrix}
    A & B^*  \\
  B& C 
\end{pmatrix}
\ee
and  $n\times n$ matrix $\widehat D$ as
\be\label{defhatD}
\widehat D:=A -B^*C^{-1}B.
\ee
 Then $\wh D$ is invertible if $D$ is invertible \nc
and for any $1\leq i,j\leq n$, we have 
\be
(D^{-1})_{ij}= {(\widehat D^{-1})}_{ij}
\label{Dinv}
\ee
for the 
corresponding matrix elements. \qed
\end{lemma}

We will use this formula for the resolvent of $H$.
Recall that $G_{ij}=G_{ij}(z)$ 
denotes the  matrix element of the resolvent
\[
G_{ij}=\left(\frac1{H-z}\right)_{ij}.
\]
Let $H^{[i]}$ denote the $i$-th minor of $H$, i.e. the $(N-1)\times (N-1)$ matrix obtained from $H$ by
removing the $i$-th row and column:
\[
     H^{[i]}_{ab} := h_{ab}, \qquad a, b\ne i.
\]
Similarly, we set 
\[ G^{[i]}(z): =\frac{1}{H^{[i]}-z}
 \]
to be the resolvent of the minor.
For $i=1$, $H$ has the  block-decomposition
\[
   H = \begin{pmatrix}  h_{11} & [\ba^1]^*  \cr \ba^{1} & H^{[1]} \end{pmatrix},
\]
where $\ba^{i}\in \C^{N-1}$ is the $i$-th column of $H$ without  the $i$-th element.

Using Lemma \ref{2x2} for $n=1$, $m=N-1$ we have
\be
     G_{ii} =
\frac{1}{ h_{ii} - z- [\ba^i]^* G^{[i]}\ba^i},
\label{1row}
\ee
 where
\be\label{hGh}
  [\ba^i]^*G^{[i]}\ba^i  = \sum_{k,l\ne i} h_{ik} G_{kl}^{[i]} h_{li}.
\ee
Here and below, we use the convention that unspecified summations always run from 1 to $N$.

Now we use the fact that for Wigner matrices $\ba^i$ and $H^{[i]}$ are independent. So in the 
quadratic form \eqref{hGh} we can condition on the $i$-th minor and momentarily consider only
the randomness of  the $i$-th column. Set $i=1$ for notational simplicity. Then we have a quadratic form of the type
\[
     \ba^* B \ba = \sum_{k,l=2}^N\bar a_k B_{kl} a_l
 \]
 where $B=G^{[1]}$ is considered as 
 a fixed deterministic matrix and $\ba$ is a random vector with centered  i.i.d. components and $\E |a_k|^2=1/N$.
We decompose it into its expectation w.r.t. $\ba$, denoted by $\E_\ba$, and the fluctuation:
\be\label{Zdef}
   \ba^* B \ba = \E_\ba  \ba^* B \ba + Z, \qquad  Z:=  \ba^* B \ba -\E_\ba  \ba^* B \ba.
\ee
The expectation gives
\[
   \E_\ba  \ba^* B \ba= \E_\ba \sum_{k,l=2}^{N} \bar a_k B_{kl} a_l = \frac{1}{N} \sum_{k=2}^{N}  B_{kk} = \frac{1}{N}\tr B,
\]
where we used that $a_k$ and $a_l$ are independent, $\E_\ba \bar a_k a_l=\delta_{kl}\cdot\frac{1}{N}$, so the double sum collapses to a single sum.
Neglecting the fluctuation $Z$ for a moment (see an argument later), we have from \eqref{1row} that
\be\label{g11}
  G_{11} = -\frac{1}{z + \frac{1}{N}\tr G^{[1]} +\mbox{error}},
\ee
where we also included the small $h_{11}\sim N^{-1/2}$ into the error term. Furthermore, it is easy to see that $\frac{1}{N}\tr G^{[1]}$ and
$\frac{1}{N}\tr G$ are close to each other, this follows from a basic  fact from linear algebra that the eigenvalues of $H$ and its minor $H^{[1]}$ {\bf interlace}
(see Exercise~\ref{exinterlace}).

Similar formula holds for each $i$, not only for $i=1$.
Summing  them up, we have
\[
   \frac{1}{N} \tr G \approx -\frac{1}{z + \frac{1}{N}\tr G},
\]
which is exactly \eqref{money}, modulo the argument that the fluctuation $Z$ is small.
Notice that we were aiming only at $ \frac{1}{N}\tr G$, but in fact the procedure gave us more. After
approximately identifying  $ \frac{1}{N}\tr G\approx \frac{1}{N}\tr G^{[1]}$
with $m_{sc}$, we can feed this information back to \eqref{g11} to obtain information for each diagonal
matrix element of the resolvent:
\[ 
    G_{11}\approx -\frac{1}{z + m_{sc}} = m_{sc},
\]
i.e. not only the trace of $G$ are close to $m_{sc}$, but each diagonal matrix element.

What about the off-diagonals? It turns out that they are small. The simplest argument to
indicate this is using the {\bf Ward identity} that is valid for resolvents of any self-adjoint operator $T$:
\be\label{wardidentity}
       \sum_j \Big|\big( \frac{1}{T-z}\big)_{ij}\Big|^2= \frac{1}{\im z} \im  \big( \frac{1}{T-z}\big)_{ii}.
\ee
We recall that the imaginary part of a matrix $M$ is given by $\im M =\frac{1}{2i}(M-M^*)$ and notice that $(\im M)_{aa} =\im M_{aa}$
so there is no ambiguity in the notation of its diagonal elements. 
Notice that the summation in \eqref{wardidentity}
 is removed at the expense of a factor $1/\im z$. So if $\eta =\im z\gg 1/N$ and diagonal elements are controlled,
 the Ward identity 
is a substantial improvement over the naive bound 
of estimating each of the $N$ terms separately. In particular, applying \eqref{wardidentity} for~$G$, we get
\[
  \sum_j |G_{ij}|^2 = \frac{1}{\im z} \im G_{ii}.
\]
Since the diagonal elements have already been shown to be close to $m_{sc}$, this implies that
\[
  \frac{1}{N} \sum_j |G_{ij}|^2\approx \frac{\im m_{sc}}{N \im z} ,
\]
i.e. on average we have
\[
   |G_{ij}| \lesssim \frac{1}{\sqrt{N\eta}}, \qquad i\ne j.
\]
With a bit more argument, one can show that this relation holds for every $j\ne i$ and not just on average up to a factor $N^\e$
with very high probability.
We thus showed that the resolvent $G$ of a Wigner matrix is close to the  $m_{sc}$ times the identity matrix $I$, very roughly
\be\label{sloppy}
    G(z) \approx m_{sc}(z)I.
\ee
Such relation must be treated with a certain care, since $G$ is a large matrix and the sloppy formulation in
\eqref{sloppy} does not indicate in which sense the closeness $\approx$ is meant. It turns our that it holds
in {\bf normalized trace sense}:
\[   
    \frac{1}{N} \tr G \approx m_{sc},
\]
in {\bf entrywise sense}:
\be\label{int}
    G_{ij} \approx m_{sc} \delta_{ij}
\ee
for every fixed $i,j$; and more generally in {\bf isotropic sense}:
\[
 \langle \bx, G\by\rangle \approx m_{sc} \langle \bx, \by\rangle
 \]
for every fixed (deterministic) vectors $\bx, \by\in \C^N$.  In all cases, these relations are meant
with very high probability.
But \eqref{sloppy} does {\bf not} hold in operator norm sense
since
\[
    \| G(z)\|=\frac{1}{\eta}, \qquad \mbox{while}\quad \| m_{sc} I\| = |m_{sc}|\sim O(1)
\]
even if $\eta\to 0$. One may not invert \eqref{sloppy} either, since the relation
\be\label{wring}
   H-z \approx \frac{1}{m_{sc}} I 
\ee
is very wrong, in fact 
\[
  H-z\approx -z
\]
if we disregard small off-diagonal elements as we did in \eqref{int}.  The point is that
the cumulative effects of many small  off diagonal matrix elements  substantially changes the matrix.
In fact, using   \eqref{exact}, the  relation \eqref{int}  in the form
\be\label{self energy}
   \Big( \frac{1}{H-z}\Big)_{ij} \approx \frac{1}{-z-m_{sc}(z)}\delta_{ij} 
\ee
exactly shows how much  the spectral parameter must be shifted  compared to the naive (and wrong)
approximation $(H-z)^{-1}\approx -1/z$. This amount is $m_{sc}(z)$ and it is often called
{\bf self-energy shift} in the physics literature. On the level of the resolvent
(and in the senses described above), the effect of the random matrix $H$
can be simply described by this shift.


Finally, we indicate the mechanism that makes the fluctuation term $Z$ in \eqref{Zdef} small. We  
compute only  its variance, higher moment calculations are similar but more involved:
\[
   \E_\ba |Z|^2 = \sum_{mn} \sum_{kl} \E_\ba \Big[  a_m \bar B_{mn} \bar a_n - \E_\ba  a_m \bar B_{mn} \bar a_n\Big]\Big[ \bar a_k B_{kl} a_l - \E_\ba \bar a_k B_{kl} a_l\Big].
\]
The summations run for all indices from 2 to $N$.
Since $\E_\ba a_m=0$, in the terms with nonzero contribution we need to pair every $a_m$ to another $\bar a_m$. For simplicity,  
here we assume
that we work with the complex symmetry class and $ \E a_m^2=0$ (i.e. the real and imaginary parts of each matrix elements $h_{ij}$ are
independent and identically distributed). If $a_m$ is paired with $\bar a_n$ in the above sum, i.e. $m=n$, then this pairing is cancelled by
the $\E_\ba  a_m \bar B_{mn} \bar a_n$ term. So $a_i$ must be paired with an $a$ from the other bracket and since $\E a^2=0$, it has to be 
paired with $\bar a_k$, thus $m=k$. Similarly $n=l$ and we get 
\be\label{Zsquare}
  \E_\ba |Z|^2 = \frac{1}{N^2}\sum_{m\ne n} |B_{mn}|^2 + \E_\ba |a|^4\sum_m |B_{mm}|^2,
\ee
where the last term comes from the case when $m=n=k=l$. Assuming that the matrix elements $h_{ij}$ have fourth moments in a sense
that $\E |\sqrt{N} h_{ij}|^4\le C$, we have $\E_\ba |a|^4 = O(N^{-2})$ in this last term and it is negligible. The main term
in \eqref{Zsquare} has a summation
over $N^2$ elements, so {\it a priori} it looks order one, i.e. too 
large. But in our application, $B$ will be the resolvent of the minor, $B= G^{[1]}$, and we can 
use the Ward identity \eqref{wardidentity}.

 In our concrete application with   $B=G^{[1]}$ we get
\begin{align*}
      \E_\ba |Z|^2 =  & \frac{1}{N\eta} \frac{1}{N}\sum_{m} \im B_{mm}  + \frac{C}{N^2}\sum_m |B_{mm}|^2  \\
     \le {}& \frac{C}{N\eta} \frac{1}{N} \im \tr G^{[1]} \le \frac{C}{N\eta} \im m^{[1]}
      = O\Big( \frac{1}{N\eta}\Big),
\end{align*}
which is small, assuming $N\eta\gg 1$. To estimate the second term here we used that for the resolvent of any hermitian matrix $T$ we have
\be\label{squaredresolvent}
    \sum_m \Big| \Big(\frac{1}{T-z}\Big)_{mm}\Big|^2  \le \frac{1}{\eta} \im \tr  \frac{1}{T-z}
\ee
by spectral calculus.  We also used that the traces of $G$ and $G^{[1]}$ are close:
\begin{exercise}\label{exinterlace}
Let $H$ be any hermitian matrix and $H^{[1]}$ its minor. Prove that their eigenvalues interlace, i.e. they satisfy
\[
  \lambda_1\le \mu_1\le \lambda_2\le \mu_2 \le\ldots \le \mu_{N-1}\le \lambda_N,
\]
where the $\lambda$'s and $\mu$'s are the eigenvalues of $H$ and $H^{[1]}$, respectively. Conclude from this that
\[
      \Big| \tr \frac{1}{H-z} - \tr \frac{1}{H^{[1]}-z}\big|\le \frac{1}{\im z}
\]
\end{exercise}

\begin{exercise}
Prove the Ward identity \eqref{wardidentity} and   the estimate  \eqref{squaredresolvent} by using the spectral decomposition of $T=T^*$.
\end{exercise}

\subsubsection{Cumulant expansion}\label{sec:cum}

Another way to prove \eqref{money} starts with the defining identity of the resolvent: $HG = I +zG$ and computes its expectation:
\be\label{eggg}
   \E HG = I + z \E G.
 \ee
 Here $H$ and $G$ are not independent, but it has the structure that the basic random variable $H$ multiplies a function of it
 viewing $G=G(H)$.
 In a single random variable $h$ it looks like $ \E h f(h)$. If $h$ were a centered  real Gaussian, then we could use the basic 
 {\bf integration by parts} identity of Gaussian variables:
 \be\label{IP}
 \E h f(h) = \E h^2 \E f'(h).
 \ee
 In our concrete application, when $f$ is the resolvent whose derivative is its square, in the Gaussian case we have
 the formula
 \be\label{IP1}
   \E HG = -\E \wt\E \big[ \wt H G \wt H \big] G,
 \ee
 where tilde denotes an independent copy of $H$. 
We may define a linear map $\cS$ on the space of $N\times N$ matrices by
\be\label{Sdef}
   \cS[R]: = \wt\E \big[ \wt H R \wt H \big],
\ee
then we can write \eqref{IP1} as
\[
   \E HG = -\E \cS[G] G.
\]
 
 This indicates to smuggle  the $\E HG$ term into $HG = I +zG$ and write it as
\be\label{egg}
    D=  I + \big(z+\cS[G] \big) G, \qquad D: = HG + \cS[G] G.
 \ee
With these notations,  \eqref{IP1} means that $\E D=0$. Notice that the term $\cS[G]G$ acts as a counter-term to
balance $HG$.

Suppose we can prove that $D$ is small with high probability, i.e. 
not only $\E D=0$ but also $\E |D_{ij}|^2$ is small for any $i,j$, then
\be\label{eg1}
   I + \big(z+\cS[G] \big) G\approx 0.
\ee
So it is not unreasonable to hope that the solution $G$ will be, in some sense, close to the solution $M$ of the 
deterministic equation 
\be\label{mdeee}
 I + \big(z+\cS[M] \big) M =0 
 \ee
 with the side condition that $\im M := \frac{1}{2i}(M-M^*)\ge 0$ (positivity in the sense of
 hermitian matrices). It turns out that this equation  in its full generality will play a central role in 
 our analysis for much larger class of random matrices, see Section~\ref{sec:corr} later.
 The operator $\cS$ is called the {\bf self-energy operator} following the analogy explained around \eqref{self energy}.

To see how $\cS$ looks like, in the real  Gaussian Wigner case (GOE)  we have
\[
  \cS[R]_{ij}=  \wt\E \big[ \wt H R \wt H \big]_{ij} = \wt\E \sum_{ab} \wt h_{ia} R_{ab} \wt h_{bj}  = \delta_{ij} \frac{1}{N}\tr R + \frac{1}{N} R_{ji} {\bf 1}(i\ne j).
\] 
Plugging this relation back into \eqref{eg1}  with $R=G$ and neglecting the second term $\frac{1}{N} G_{ji}$ we have
\[
    0\approx I + \big(z+ \frac{1}{N}\tr G\big) G.
\]
 Taking the normalized trace, we end up with
 \be\label{mn}
       1 + (z + m_N)m_N \approx 0,
  \ee
  i.e. we proved \eqref{money}.

\begin{exercise}
Prove \eqref{IP} by a simple integration by parts and then use \eqref{IP} to prove \eqref{IP1}.
Formulate and prove the complex  versions of these formulas (assume that $\re h$ and $\im h$ are
independent).
\end{exercise}

\begin{exercise}
Compute the variance $\E |D|^2$ for a GOE/GUE matrix and conclude that it is small in the regime where $N\eta\gg 1$ (essentially as $(N\eta)^{-1/2}$).
Compute $\E \big| \frac{1}{N}\tr D\big|^2$ as well and show that it is essentially of order $(N\eta)^{-1}$.
\end{exercise}

This argument so far  heavily used that $H$ is Gaussian. However, the basic integration by parts formula
\eqref{IP}  can be extended to non-Gaussian situation. For this, we recall the {\bf cumulants} of  random variables.
We start with a single random variable $h$. As usual, its moments are defined by
\[
   m_k : = \E h^k,
\]
and they are generated by the {\it moment generating function}
\[
    \E e^{th} = \sum_{k=0}^\infty \frac{t^k}{k!} m_k
\]
(here we assume that all moments exist and  even the exponential moment exists at least for small $t$).
The {\bf cumulants}  $\kappa_k$ of $h$ are the Taylor coefficients of the {\it logarithm} of the moment generating function, i.e. they are 
defined by the identity
\[
   \log \E e^{th} = \sum_{k=0}^\infty \frac{t^k}{k!} \kappa_k.
\]
The sequences of $\{ m_k\; : \; k=0,1,2\ldots\}$ and $\{ \kappa_k\; : \; k=0,1,2\ldots\}$
mutually determine each other; these relations can be obtained from formal power series manipulations. For example
\[
    \kappa_0=m_0=1, \qquad \kappa_1=m_1,\qquad \kappa_2  = m_2- m_1^2, \qquad \kappa_3= m_3 - 3m_2m_1+ 2m_1^3, \ldots
\]
and 
\[
    m_1=\kappa_1,\qquad m_2 = \kappa_2+\kappa_1^2, \qquad  m_3= \kappa_3 +3\kappa_2\kappa_1+2\kappa_1^3,\ldots
 \]
 The general relations are given by
 \be\label{inv}
       m_k = \sum_{\pi\in \Pi_k} \prod_{B\in \pi}\kappa_{|B|},\qquad \kappa_k = \sum_{\pi\in \Pi_k}  (-1)^{|\pi|-1} (|\pi|-1)!\prod_{B\in \pi}m_{|B|},
 \ee
 where $\Pi_k$ is the set of all partitions of a $k$-element base set, say $\{1,2,\ldots, k\}$. Such a $\pi$ consists of
 a collection of  nonempty,
 mutually disjoint sets $\pi= \{ B_1, B_2, \ldots B_{|\pi|} \}$
 such that $\cup B_i =\{1,2,\ldots, k\}$  and $B_i\cap B_j=\emptyset$, $i\ne j$.
 
 For Gaussian variables, all but the first and second cumulants vanish, that is,  $\kappa_3=\kappa_4=\ldots =0$, and this is the reason for 
 the very simple form of the relation \eqref{IP}. For general  non-Gaussian $h$ we have 
 \be\label{cumexp}
   \E h f(h) = \sum_{k=0}^\infty \frac{\kappa_{k+1}}{k!} \E f^{(k)}(h).
 \ee
 Similarly to the Taylor expansion, one does not have to expand it up to infinity, there are versions of this
 formula containing only a finite number of cumulants plus a  remainder term.
 
To see the formula \eqref{cumexp}, we use Fourier transform:
\[
   \hat f(t) = \int_\R e^{ith} f(h) \rd h, \qquad \hat \mu(t)  = \int_\R e^{ith} \mu(\rd h) = \E e^{ith},
\]
where $\mu$ is the distribution of $h$, then 
\[
  \log \hat\mu(t) = \sum_{k=0}^\infty \frac{(it)^k}{k!} \kappa_k.
\]

 By Parseval identity (neglecting $2\pi$'s and assuming $f$ is real)
\[
  \E hf(h) = \int_\R hf(h) \mu(\rd h) =  i \int_\R \ov{ \hat f'(t)}\hat \mu(t) \rd t.
\]
Integration by parts gives
\begin{align*}
   i\int_\R \ov{ \hat f'(t)} \hat \mu(t) \rd t  ={}& - i \int_\R \ov{\hat f(t)} \hat \mu'(t) \rd t 
   = -i  \int_\R \ov{\hat f(t)} \hat \mu(t)  \big( \log \hat\mu(t)\big)' \rd t \\
={}& \sum_{k=0}^\infty \frac{\kappa_{k+1} }{k!} \int_\R (it)^{k} \ov{ \hat f(t)} \hat \mu(t) \rd t
   = \sum_{k=0}^\infty \frac{\kappa_{k+1}}{k!} \E f^{(k)}(h)
\end{align*}
by Parseval again.

So far we considered one random variable only, but joint cumulants can also be defined for any number of random variables.
This becomes especially relevant beyond the independent case, e.g. when the entries of the random matrix  have correlations.
For the Wigner case, many of these formulas simplify, but it is useful to introduce joint cumulants in full generality.

If $\bh =(h_1, h_2, \ldots h_m)$ is a collection of random variables (with possible repetition), then 
\[
   \kappa (\bh) = \kappa(h_1, h_2, \ldots h_m)
 \]
 are the coefficients of the logarithm of the moment generating function:
 \[
       \log \E e^{{\bf t}\cdot \bh} = \sum_{{\bf k}=0}^\infty \frac{{\bf t}^{\bf k}}{{\bf k}!} \kappa_{\bf k}.
\]
 Here ${\bf t} = (t_1, t_2, \ldots, t_n)\in \R^n$, and ${\bf k}=(k_1, k_2, \ldots, k_n)\in\N^n$ is a multi index with   $n$ components and 
 \[
    {\bf t}^{\bf k}: = \prod_{i=1}^n t_i^{k_i}, \qquad {\bf k}! = \prod_i k_i!, \qquad \kappa_{\bf k} = \kappa(h_1, h_1, \ldots h_2, h_2, \ldots ),
\]
where $h_j$ appears $k_j$-times (order is irrelevant, the cumulants are fully symmetric functions in all their variables).
The formulas \eqref{inv} naturally generalize, see e.g. Appendix A of \cite{EKScorrelated} for a good summary.
The analogue of \eqref{cumexp} is
\be\label{cumexp1}
   \E h_1 f({\bf h}) = \sum_{{\bf k}} \frac{\kappa_{{\bf k}+{\bf e}_1}}{{\bf k}!} \E f^{({\bf k})}({\bf h}), \qquad {\bf h}=(h_1, h_2, \ldots , h_n),
 \ee
where the summation is for all $n$-multi-indices and 
\[
  {\bf k} + {\bf e}_1 = (k_1+1, k_2, k_3, \ldots ,k_n)
\]
and the proof is the same. 

We use these cumulant expansion formulas to prove that $D$ defined in \eqref{egg} is small
with high probability by computing $\E |D_{ij}|^{2p}$ with large $p$.  Written as
\[
   \E |D_{ij}|^{2p} =  \E \big( HG + \cS[G]G)_{ij} D^{p-1}_{ij} \bar D_{ij}^{p},
 \]
 we may use \eqref{cumexp1} to do an integration by parts in the first $H$ factor, considering
 everything else as a function $f$. It turns out that the $\cS[G]G$ term cancels the second order cumulant
 and naively  the effect of higher order cumulants are negligible since a cumulant of order $k$ is $N^{-k/2}$.
 However, the derivatives of $f$ can act on the  $D^{p-1} \bar D^{p}$ part of $f$, resulting in a
 complicated combinatorics and in fact many cumulants need to be tracked, see \cite{EKScorrelated} for an extensive analysis.

\subsection{Deterministic stability step}

In this step we compare the approximate equation \eqref{money}  satisfied by the
empirical Stieltjes transform and the exact equation \eqref{exact} for the self-consistent Stieltjes transform
\[
 m_N(z) \approx -\frac{1}{z+m_N(z)},\qquad
     m_{sc}(z) = -\frac{1}{z+m_{sc}(z)}.
\]
In fact, considering the format \eqref{eg1} and \eqref{mn}, sometimes it is better to relate the following two equations
\[
    1+ (z+m_N)m_N\approx 0, \qquad  1+ (z+m_{sc})m_{sc} =0.
 \]
This distinction is irrelevant for Wigner matrices, where the basic object to investigate
is $m_N$, a scalar quantity -- multiplying an  equation with it is a trivial operation.
 But  already \eqref{egg} indicates that there is an approximate equation 
for the entire  resolvent $G$ as well and not only for its trace and in general we are interested in resolvent matrix elements as well.
 Since inverting $G$ is a nontrivial
operation (see the discussion after \eqref{sloppy}), the three possible versions of \eqref{egg} are very different:
\[
   I+ (z +\cS[G])G \approx 0, \qquad G \approx -\frac{1}{z+\cS[G]}, \qquad  - \frac{1}{G} \approx z+\cS[G]
\]
In fact the last version is blatantly wrong, see \eqref{wring}. The first version is closer to the
spirit of the cumulant expansion method, the second is closer to Schur formula method.

In both cases, we need to understand the stability of the equation
\[
    m_{sc}(z) = -\frac{1}{z+m_{sc}(z)} \quad \mbox{or}\quad 1+ (z+m_{sc})m_{sc} =0
 \]
 against a small additive perturbation. For definiteness, we look at the second equation and compare
 $m_{sc}$ with $m_\e$, where $m_\e$ solves
 \[
 1+ (z+m_\e)m_\e =\e
\]
for some small $\e$. Since these are quadratic equations, one may write up the solutions explicitly
and compare them, but this approach will not work in the more complicated situations. Instead, we subtract these two equations
and find that
\[
    (z+2m_{sc})(m_\e-m_{sc})  + (m_\e-m_{sc})^2=\e
 \]
 We may also eliminate $z$ using  the equation $1+ (z+m_{sc})m_{sc} =0$ and get
 \be
    \frac{m_{sc}^2-1}{m_{sc}}(m_\e-m_{sc})  + (m_\e-m_{sc})^2=\e.
\label{stabilit}
\ee
This is a quadratic equation for the difference $m_\e-m_{sc}$ and its
stability thus depends on the invertibility of the linear coefficient $(m_{sc}^2-1)/m_{sc}$,
which is determined by the limiting equation only. If we knew that
\be\label{bounds}
    |m_{sc}| \le C, \qquad |m_{sc}^2-1|\ge c
\ee
 with some positive constants $c, C$, then the linear coefficient would be invertible
 \be\label{linst}
      \Bigg| \Big[ \frac{m_{sc}^2-1}{m_{sc}}\Big]^{-1}\Bigg|\le C/c
 \ee
and \eqref{stabilit} would imply that
\[
     |m_\e-m_{sc}|\le C'\e
\]
at least if we had an a priori information that $|m_\e-m_{sc}|\le c/2C$. 
This a priori information can be obtained for large $\eta=\im z$ easily since
in this regime both $m_{sc}$ and $m_\e$ are of order $\eta$ (we still remember that 
$m_\e$ represents a Stieltjes transform). Then we can use a fairly standard continuity argument
to reduce $\eta=\im z$ and keeping $E=\re z$ fixed to see that the bound $|m_\e-m_{sc}|\le c/2C$
holds for small $\eta$ as well, as long as the perturbation  $\e=\e(\eta)$ is small.

Thus the key point of the stability analysis is to show that the inverse of the {\bf stability constant} (later: operator/matrix)
given in \eqref{linst} is bounded. As indicated in  \eqref{bounds}, the control of the stability constant typically will have two ingredients:
we need


(i)  an upper bound on $m_{sc}$, the solution of the deterministic  Dyson equation \eqref{exact};

(ii) an upper bound on the inverse of $1-m_{sc}^2$.


In the Wigner case, when $m_{sc}$ is explicitly given \eqref{msc}, both bounds are easy to obtain.
In fact, $m_{sc}$ remains bounded for any $z$, while $1-m_{sc}^2$ remains separated
away from zero except near two special values of the spectral parameter: $z =\pm 2$. These
are exactly the edges of the semicircle law, where an instability arises since here $m_{sc}\approx \pm 1$ (the same instability can
be seen from the explicit solution of the quadratic equation).

We will see that it is not a coincidence:  the edges of the asymptotic density $\varrho$  are always the critical points where the inverse of the
 stability constant blows up. These regimes require more careful treatment which typically consists
 in exploiting the fact that the error term $D$ is proportional with the local density, hence it is also 
 smaller near the edge. This additional smallness of $D$ competes with the deteriorating upper bound
 on the inverse of the  stability constant near the edge.
 
 In these notes we will focus on the behavior in the bulk, i.e. we consider 
spectral parameters $z=E+i\eta$ where $\varrho(E)\ge c>0$ for fixed positive constants. This will simplify
many estimates. The regimes where $E$ is separated away from the support of $\varrho$ are even easier and we will not consider them here.
The edge analysis is more complicated and we refer the reader to the original papers.

\section{Models of increasing complexity}\label{sec:compl}

\subsection{Basic setup}
In this section we introduce subsequent generalizations of the original Wigner ensemble. 
We also mention the key features of their resolvent that will be proven later along the local laws.
The $N\times N$ matrix 
\be\label{Hdef1}
H=
\begin{pmatrix} h_{11} &  h_{12} & \ldots & h_{1N} \\ 
 h_{21} &  h_{22} & \ldots & h_{2N} \\
\vdots & \vdots &  & \vdots\\
 h_{N1} &  h_{N2} & \ldots & h_{NN} \\
\end{pmatrix} 
\ee
will always be hermitian, $H=H^*$ and centered, $\E H=0$. The distinction between
real symmetric and complex hermitian cases play no role here; both symmetry classes are allowed.
Many quantities, such as the distribution of $H$, 
the matrix of variances $S$,   naturally depend on $N$, but for notational simplicity we will often
omit this dependence from the notation. 

We will always assume that we are in the mean field regime, i.e. the typical size of the 
 matrix elements is of order $N^{-1/2}$ in a high moment sense:
 \be\label{mmm}
 \max_{ij} \E \big| \sqrt{N} h_{ij}\big|^{p} \le \mu_p
 \ee
 for any $p$  with some sequence of constants $\mu_p$. This strong moment  condition 
 can be substantially relaxed but we will not focus on this direction.

\subsection{Wigner matrix}

We assume that the matrix elements of $H$ are independent (up to the hermitian symmetry) and
identically distributed. We choose the normalization such that 
\[
    \E |h_{ij}|^2 =\frac{1}{N},
\]
see \eqref{aves} for explanation.
The asymptotic density of eigenvalues 
 is the
semicircle law, $\varrho_{sc}(x) $ \eqref{sc} and its Stieltjes transform $m_{sc}(z)$ is given explicitly  in \eqref{msc}.
The corresponding self-consistent (deterministic) equation (Dyson equation) is a  {\bf scalar equation}
\[
     1 + (z+m)m=0, \qquad \im m>0,
 \]
 that is solved by $m=m_{sc}$. The  inverse of the stability ``operator'' is just  the constant
 \[
           \frac{1}{1-m^2}, \qquad m=m_{sc}.
 \]
The resolvent $G(z) =(H-z)^{-1}$ is approximately constant diagonal in the {\bf entrywise sense}, i.e.
\be\label{jig}
    G_{ij}(z) \approx \delta_{ij} m_{sc}(z).
\ee
In particular,  the diagonal elements are approximately the same
\[
  G_{ii}\approx G_{jj} \approx m_{sc}(z).
\]
This also implies that the normalized trace (Stieltjes transform of the empirical eigenvalue density)  is close to $m_{sc}$
\be\label{avert}
  m_N (z) =\frac{1}{N}\tr G(z) \approx m_{sc}(z),
\ee
which we often call an approximation in {\bf average (or tracial) sense}.

Moreover, $G$ is also diagonal in {\bf isotropic sense}, i.e. for any vectors $\bx, \by$ (more precisely, any sequence of
vectors $\bx^{(N)}, \by^{(N)}\in \C^N$) we have
\be\label{isotro}
  G_{\bx\by}:= \langle \bx, G\by\rangle \approx m_{sc}(z) \langle \bx, \by\rangle. 
\ee
In Section~\ref{sec:precise}  we will comment on the precise meaning of $\approx$ in this context, incorporating 
the fact that $G$ is random.

If these relations hold for any fixed $\eta=\im z$, independent of $N$, then we talk about {\bf global law}.
If they hold down to $\eta\ge N^{-1+\gamma}$ with some $\gamma\in (0,1)$, then we talk about {\bf local law}.
If $\gamma>0$ can be chosen arbitrarily small (independent of $N$), than we talk about {\bf local law on the
optimal scale}.

\subsection{Generalized Wigner matrix}

We assume that the matrix elements of $H$ are independent (up to the hermitian symmetry), but not necessarily
identically distributed. We define the matrix of variances as
\be
  S: = \begin{pmatrix} s_{11} &  s_{12} & \ldots & s_{1N} \\ 
 s_{21} &  s_{22} & \ldots & s_{2N} \\
\vdots & \vdots &  & \vdots\\
 s_{N1} &  s_{N2} & \ldots & s_{NN} \\
\end{pmatrix}, \qquad s_{ij}: = \E |h_{ij}|^2.
\label{def:Sigmamatrix}
\ee
We assume that
\be
  \sum_{j=1}^N s_{ij} =1, \qquad \mbox{for every} \quad i=1,2,\ldots, N,
\label{normal}
\ee
i.e., the deterministic
$N\times N$  matrix of variances, $S= (s_{ij})$,
 is  symmetric and  doubly stochastic. The key point is
 that the row sums are all the same. 
 The fact that the sum in \eqref{normal} is exactly one is a chosen normalization.
  The original Wigner ensemble is a special case, $s_{ij}=\frac{1}{N}$.

Although generalized Wigner matrices form a bigger class than the Wigner matrices, the 
key results are exactly the same. The asymptotic  density of states is still the semicircle law,
$G$ is constant diagonal in both the entrywise and isotropic senses:
\[ 
    G_{ij} \approx \delta_{ij} m_{sc} \qquad \mbox{and}\quad 
 G_{\bx\by}= \langle \bx, G\by\rangle \approx m_{sc} \langle \bx, \by\rangle .
\]
In particular, the diagonal elements are approximately the same
\[
   G_{ii}\approx G_{jj}
 \]
 and we have the same averaged law
 \[
     m_N (z) =\frac{1}{N}\tr G(z) \approx m_{sc}(z).
\]
However, within the proof some complications arise. Although eventually $G_{ii}$ turns out to be
essentially independent of $i$, there is no a-priori complete  permutation symmetry among the indices.
We will need to consider the equations for each $G_{ii}$ as a coupled system of $N$ equations.
The corresponding Dyson equation is a genuine {\bf vector equation} of the form
\be\label{vectordyson}
     1 + (z + (S\bbm)_i)m_i  =0, \qquad i=1,2,\ldots N
\ee
for the unknown $N$-vector $\bbm=(m_1, m_2, \ldots , m_N)$ with $m_j\in \Cp$ and we will see that $G_{jj}\approx m_j$.
The matrix $S$ may also be called {\bf self-energy matrix} according to the analogy
explained around \eqref{self energy}.
Owing to \eqref{normal}, the solution to \eqref{vectordyson} is still the constant vector $m_i=m_{sc}$, but the stability operator
depends on $S$ and it is given by the matrix
\[
     1-m_{sc}^2S.
\]

\subsection{Wigner type matrix}

We still assume that the matrix elements are independent, but we impose no special algebraic condition on the variances $S$. 
For normalization purposes, we will assume  that $\|S\|$ is bounded, independently of  $N$, this guarantees that the spectrum of $H$ 
also remains 
bounded.
We only
require an upper bound of the form
\be\label{Supp}
     \max_{ij}  s_{ij}\le \frac{C}{N}
\ee
for some constant $C$. This is a typical mean field condition, it guarantees that no matrix element is too big.
Notice that at this stage there is no requirement for a lower bound, i.e. some $s_{ij}$ may vanish.
However, the analysis becomes considerably harder if large blocks of $S$ can become zero,
so for pedagogical convenience later  in these notes we will assume that $s_{ij}\ge c/N$ for some $c>0$.

The corresponding Dyson equation is just the {\bf vector Dyson equation} \eqref{vectordyson}:
\be\label{vectordyson1}
  1 + (z + (S\bbm)_i)m_i  =0, \qquad i=1,2,\ldots N
 \ee
 but the solution is not the constant vector any more.  
 We will see that the system of equations  \eqref{vectordyson1} still has a unique solution
 $\bbm=(m_1, m_2, \ldots , m_N)$ under the side condition $m_j\in \Cp$, but the components of $\bbm$ may differ
 and they are not given by $m_{sc}$ any more.  
 
 The components $m_i$ approximate the diagonal elements of the resolvent $G_{ii}$.
 Correspondingly, their average 
 \be\label{dos}
   \langle \bbm\rangle := \frac{1}{N}\sum_i m_i,
 \ee
 is the  Stieltjes transform of  a measure  $\varrho$ that approximates the empirical density of states.
We will call this measure the {\bf self-consistent density of states} since it is obtained
from the self-consistent Dyson equation. It is well-defined for any finite $N$ and if it has a limit as
$N\to\infty$, then the limit coincides with the asymptotic density introduced earlier (e.g. the semicircle law for Wigner
and generalized Wigner matrices). However, our analysis is more general and it does not need to assume the existence of this limit
(see Remark~\ref{remark:DOS} later).

In general there is no explicit formula for $\varrho$, 
  it has to be computed by taking the inverse Stieltjes transform of $ \langle \bbm(z)\rangle$:
  \be\label{ados}
   \varrho(\rd \tau ) =  \lim_{\eta\to 0+}\frac{1}{\pi}\im \langle\bbm(\tau +i\eta)\rangle\rd\tau.
  \ee
 No simple closed equation is known for the scalar quantity $ \langle \bbm(z)\rangle$, even if one is interested only in
 the self-consistent density of states or its Stieltjes transform, the only known way to compute it is to solve \eqref{vectordyson1} first
 and then take the average of the solution vector. Under some further conditions on $S$, the density of states is
 supported on finitely many intervals, it is real analytic away  from the edges of these intervals and
 it has a specific singularity structure at the edges, namely it can have either square root singularity or cubic root cusp,
 see Section~\ref{sec:properties} later.

The resolvent is still approximately diagonal and it is given by the $i$-th component of $\bbm$:
\[
   G_{ij}(z) \approx \delta_{ij} m_i(z),
\]
but in general
\[
   G_{ii}\not\approx G_{jj}, \qquad i\ne j.
\]
Accordingly, the isotropic law takes the form
\[
   G_{\bx\by}= \langle \bx, G\by\rangle \approx \langle \bar\bx\bbm\by\rangle  
\] 
and the averaged law
\[
   \frac{1}{N}\tr G \approx  \langle \bbm\rangle.
\]
Here $\bar\bx\bbm\by$ stands for the entrywise product of vectors, i.e.,
 $\langle \bar\bx\bbm\by\rangle  = \frac{1}{N}\sum_i \bar x_i m_i y_i$.
 
The stability operator is
\be\label{vectors}
   1- \bbm^2 S,
\ee
where $\bbm^2 $ is  understood as an entrywise multiplication, so the linear operator $\bbm^2 S$  acts on any vector $\bx\in \C^N$
as
\[
      [(\bbm^2 S)\bx]_i := m_i^2 \sum_j s_{ij}x_j.
 \]

{\it Notational convention.}
 Sometimes we write 
 the equation \eqref{vectordyson1} in the concise vector form as
 \[
     -\frac{1}{\bbm}  = z + S\bbm.
 \]
 Here we introduce the convention that for any vector $\bbm\in \C^N$ and for any function $f:\C\to \C$,
 the symbol $f(\bbm)$ denotes the $N$-vector with components $f(m_j)$,  that is,
 \[
   f(\bbm): = \big(f(m_1), f(m_2), \ldots, f(m_N)\big), \qquad \mbox{for any $\bbm = (m_1, m_2, \ldots, m_N)$}.
 \]
 In particular, $1/\bbm$ is the vector of the reciprocals $1/m_i$. Similarly, the entrywise  product of two
 $N$-vectors $\bx, \by$ is denoted by $\bx\by$; this is the $N$-vector with components 
 \[
    (\bx\by)_i: = x_iy_i
 \]
 and similarly for products of more than two factors. Finally $\bx\le \by$ for real vectors  means $x_i\le y_i$ for all $i$.

\subsubsection{A remark on the density of states}\label{remark:DOS}
 The Wigner type matrix is the first ensemble where the various concepts of density of states truly differ.
The wording ``density of states'' has been used slightly differently by various authors in random matrix theory; here we use 
the opportunity to clarify this point. Typically, in the physics literature the  {\bf density of states} 
means  the statistical average of the {\bf empirical density of states} $\mu_N$ defined in  \eqref{empire}, i.e. 
\[
          \E \mu_N(\rd \tau) = \E \frac{1}{N}\sum_{i=1}^N \delta (\lambda_i-\tau).
\]
This object  depends on $N$, but very often it has a limit (in a weak sense) as $N$, the system size,
goes to infinity. The limit, if exists, is often called the  {\bf limiting (or asymptotic) density of states}. 

In general it is not easy to find $\mu_N$ or its expectation;
the vector Dyson equation is essentially the
only way to proceed.  However, the quantity computed in \eqref{ados}, called
 the {\bf  self-consistent density of states}, 
is not exactly the density of states, it is only a good approximation.
The local law states that  the empirical (random) eigenvalue density $\mu_N$
can be very well approximated  by the self-consistent density of states,
computed from the Dyson equation and \eqref{ados}. Here ``very well'' means in high probability and with
 an explicit error bound of size $1/N\eta$, i.e. on larger scales we have more
precise bound, but we still have closeness even down to scales $\eta\ge N^{-1+\gamma}$.
High probability bounds imply that also the density of states $\E \mu_N$ is
close to the self-consistent density of states $\varrho$, but in general they are not the same.
Note that  the significance of the local law is to approximate a random quantity with a
deterministic one if $N$ is large; there is no direct statement about any $N\to \infty$ limit. 
The variance matrix $S$ depends on $N$ and a-priori there is no relation between 
$S$-matrices for different $N$'s.

In some cases a limiting version of these objects also exists. For example, if the variances $s_{ij}$
arise from a deterministic nonnegative  profile function $S(x,y)$ on $[0,1]^2$ with some regularity, i.e.
\[
   s_{ij} = \frac{1}{N} S\Big( \frac{i}{N}, \frac{j}{N}\Big),
\]
then the sequence of the self-consistent density of states $\varrho^{(N)}$ have a limit. 
If the global law holds, then  this limit must be the limiting density of states, defined as the limit
of $\E\mu_N$. This is the case for Wigner matrices in a trivial way: the self-consistent density of states
is always the semicircle for any $N$. However, the density of states for finite $N$ is not the semicircle law;
it depends on the actual distribution of the matrix elements, but decreasingly as $N$ increases.

In these notes we will focus on computing the self-consistent density of states and proving local laws
for fixed $N$; 
we will not consider the possible large $N$ limits of these objects.

\subsection{Correlated random matrix}\label{sec:corr}

For this class we drop the independence condition, so the matrix elements of $H$ may have nontrivial correlations
in addition to the one required by the hermitian symmetry $h_{ij}=\bar h_{ji}$. The Dyson equation is still 
determined by the second moments of $H$, but the covariance structure of all matrix elements 
is not described by a matrix; but by a four-tensor. 
We already introduced in \eqref{Sdef}  the necessary ``super operator''
\[
   \cS[R] : = \E HRH
 \]
 acting linearly on the space of $N\times N$ matrices $R$. Explicitly
 \[
      \cS[R]_{ij} = \E \sum_{ab} h_{ia} R_{ab} h_{bj} = \sum_{ab} \big[ \E h_{ia} h_{bj}\big] R_{ab}.
\]
The analogue of the upper  bound \eqref{Supp} is
\[
     \cS[R] \le C\langle R\rangle
 \]
 for any positive definite matrix $R\ge 0$, where we introduced the notation
 \[
   \langle R\rangle: = \frac{1}{N}\tr R.
\]
In the actual proofs we will need a lower bound of the form $\cS[R]\ge c\langle R\rangle$
and further conditions on the decay of correlations among the matrix elements of $H$.

The corresponding Dyson equation becomes a {\bf matrix equation} 
\be\label{MDE}
   I + (z+\cS[M])M=0
\ee
for the unknown matrix $M=M(z)\in \C^{N\times N}$
under the constraint that $\im M\ge 0$.  Recall that the imaginary part of any
matrix is a hermitian matrix defined by
\[
   \im M = \frac{1}{2i}(M-M^*).
\]
In fact, one may add a hermitian  {\bf external source matrix} $A=A^*$ and consider
the more general equation
\be\label{MDE1}
   I + (z- A+\cS[M])M=0.
\ee
In random matrix applications, $A$ plays the role of the matrix of expectations, $A= \E H$.
We will call \eqref{MDE1} and \eqref{MDE} the {\bf matrix Dyson equation} with  or without external source.
The equation \eqref{MDE1} has a unique solution and in general it is a non-diagonal matrix even if $A$ is diagonal.
Notice that the Dyson equation contains only the second moments of the elements of $H$
via the operator $\cS$; no higher order correlations appear, although in the proofs of the local laws
further conditions on the correlation decay are necessary.

 The Stieltjes transform of the 
density of states is given by
\[ 
       \langle M(z)\rangle= \frac{1}{N}\tr M(z).
\]
The matrix $M=M(z)$ approximates the resolvent in the usual senses, i.e. we have
\[
      G_{ij} (z)\approx M_{ij}(z),
\]
\[
   \langle \bx, G\by\rangle \approx  \langle \bx, M\by\rangle,
\]
and
\[
   \frac{1}{N}\tr G\approx  \langle M\rangle.
\]
Since in general  $M$ is not diagonal, the resolvent $G$ is not approximately diagonal 
any more.
We will call $M$, the solution to the matrix Dyson equation \eqref{MDE1}, the {\bf self-consistent Green function}
or {\bf self-consistent resolvent}.

The stability operator is of the form
\[
   I-\cC_M\cS,
 \]
where $\cC_M$ is the linear map acting on the space of matrices as 
\[ \cC_M[R]: =MRM.
\]
In other words, the stability  operator is the linear map $R\to R- M\cS[R]M$ 
on the space of matrices.

The independent case (Wigner type matrix) is a special case of the correlated ensemble and
it is interesting to exhibit their relation. 
In this case the super-operator $\cS$  maps diagonal matrices to
diagonal matrix. For any vector $\bv\in \C^N$ we denote by $\mbox{diag}(\bv)$ the $N\times N$ 
diagonal matrix with $(\mbox{diag}(\bv))_{ii}=v_i$ in the diagonal. Then we have, for the independent case
with $s_{ab}: = \E |h_{ab}|^2$ as before, 
\[
    \big( \cS [\mbox{diag}(\bv)]\big)_{ij} =\sum_{a} \big[ \E \bar h_{ai} h_{aj}\big] v_a =\delta_{ij} (S\bv)_i ,
\]
thus
\[
   \cS [\mbox{diag}(\bv)] = \mbox{diag}(S\bv).
\]
\begin{exercise}
Check that in the independent case, the solution $M$ to \eqref{MDE}  is diagonal, $M=\mbox{diag}(\bbm)$, where $\bbm$ solves
the vector Dyson equation \eqref{vectordyson1}. Verify that the statements of the local laws formulated in the general 
correlated language reduce to those for the Wigner type problem. Check that the stability operator $I-\cC_M\cS$ 
restricted to diagonal matrices is equivalent to the stability operator \eqref{vectors}.
\end{exercise}


The following table summarizes the four classes of ensembles we discussed.

\begin{table}[!htbp]\scriptsize\renewcommand{\tabcolsep}{0.15cm}
  \centering
 \begin{tabular}{@{}l|c|c|c|c@{}}
\quad Name& Dyson Equation & For & Stability  op &  Feature \\[5pt]
\hline 
\\[-5pt]
\begin{tabular}{l} \textbf{Wigner} \\ {$\E |h_{ij}|^2=s_{ij}=\frac{1}{N}$}\end{tabular}
  & $1+(z+m)m=0$ &  $m \approx \frac{1}{N}\tr G $   &  $1-m^2$ & 
\begin{tabular}{l} {Scalar Dyson equation,} \\ {$m=m_{sc}$ is explicit} \end{tabular}\\[10pt]
\begin{tabular}{l} \textbf{Generalized Wigner} \\ {$\sum_j s_{ij}=1$}\end{tabular}
& $1+(z+S\bbm)\bbm=0$ &  $m_i\approx \frac{1}{N}\tr G $ &   $ 1-m^2S$ & \begin{tabular}{l}  
{Vector  Dyson equation,} \\  {Split  $S$ as  $S^\perp+ |{\bf e}\rangle\langle {\bf e}|$  }
\end{tabular} \\[10pt]
\begin{tabular}{l} \textbf{Wigner-type} \\ {$s_{ij}$ arbitrary}\end{tabular}
 & $1+(z+S\bbm)\bbm=0$ & $m_i \approx G_{ii}$  &  $1-{\bf m}^2S $ & 
\begin{tabular}{l}  {Vector  Dyson equation,} \\ {${\bf m}$  to be determined}\end{tabular}  \\[10pt]
\begin{tabular}{l}  \textbf{Correlated matrix} \\ {$\E h_{xy}h_{uw} \not\asymp\delta_{xw}\delta_{yu}$}\end{tabular}
& $I+(z+\cS[M])M=0$ & $M_{ij} \approx G_{ij}$  &  $1-M\cS[\cdot]M $ & 
\begin{tabular}{l}  {Matrix Dyson equation }\\  {Super-operator } \end{tabular}
\end{tabular}
\end{table}


We remark that in principle the averaged law (density of states) for  generalized Wigner ensemble could be studied 
 via a scalar equation only since the answer is given by the  scalar Dyson  equation, but in practice a vector 
 equation is studied in order to obtain entrywise and isotropic information. However, Wigner-type matrices need 
 a vector Dyson equation even to identify the density of states. Correlated matrices need a full scale matrix
 equation since the answer $M$ is typically a non-diagonal matrix.

\subsection{The precise meaning of the approximations}\label{sec:precise} 

In the previous  sections we used the sloppy notation $\approx$ to indicate that the  (random) resolvent $G$ in
various senses is close to a deterministic object. We now explain what we mean by that. Consider
first \eqref{jig}, the entrywise statement for the Wigner case:
\[
   G_{ij}(z) \approx \delta_{ij} m_{sc}(z).
\]
More precisely, we will see that 
\be\label{precise}
    \big| G_{ij}(z) -  \delta_{ij} m_{sc}(z)\big|\lesssim \frac{1}{\sqrt{N\eta}}, \qquad \eta=\im z
\ee
holds. Here the somewhat sloppy notation $\lesssim$ indicates that the statement holds
with very high probability and with an additional factor $N^\e$. The very precise form 
of \eqref{precise} is the following: for any $\e, D>0$ we have 
\be\label{perch}
  \max_{ij}  \P\Big(  \big| G_{ij}(z) -  \delta_{ij} m_{sc}(z)\big| \ge \frac{N^\e}{\sqrt{N\eta}}\Big) \le \frac{C_{D,\e}}{N^D}
\ee
with some constant $C_{D,\e}$ independent of $N$, but depending on $D, \e$ and the sequence $\mu_p$ bounding the
moments in \eqref{mmm}. We typically consider only spectral parameters with 
\be\label{domain}
  |z|\le C, \qquad  \eta\ge N^{-1+\gamma}
\ee
for any fixed positive constants $C$ and $\gamma$, and we encourage the reader to think of $z$ satisfying these
constraints, although our results are eventually valid for a larger set as well (the restriction $|z|\le C$ can be replaced
with $|z|\le N^C$ and the lower bound on $\eta$ is not necessary if $E=\re z$ is away from the support of the density of states).

Notice that \eqref{perch} is formulated for any fixed $z$, but the probability control is very strong, so
one can extend the same bound to hold simultaneously for any $z$ satisfying \eqref{domain}, i.e.
\be\label{perch1}
   \P\Big(  \exists z \in \C \; : \;   |z|\le C, \im z\ge N^{-1+\gamma},
   \;\;  \max_{ij} \big| G_{ij}(z) -  \delta_{ij} m_{sc}(z)\big| \ge \frac{N^\e}{\sqrt{N\eta}}  \Big) \le \frac{C_{D,\e}}{N^D}.
\ee
Bringing the maximum over $i,j$  inside the probability follows from a simple union bound.
The same trick does not work directly for bringing the maximum over all $z$ inside since there are uncountable
many of them. But notice that the function
\[
    z\to G_{ij}(z) -  \delta_{ij} m_{sc}(z)
\]
is Lipschitz continuous with a Lipschitz constant $C/\eta^2$ which is bounded by $ CN^2$ in the domain \eqref{domain}. Therefore, we
can first choose a very dense, say $N^{-3}$-grid of $z$ values, apply the union bound to them and then argue with
Lipschitz continuity for all other $z$ values.

\begin{exercise} Make this argument precise, i.e. show that \eqref{perch1} follows from \eqref{perch}.
\end{exercise}

Similar argument does not quite work for the isotropic formulation. While \eqref{isotro} holds for any fixed (sequences of)
$\ell^2$-normalized vectors $\bx$ and $\by$, i.e. in its precise formulation we have
\be\label{perch1a}
   \P\Big(  \big| \langle \bx, G(z)\by\rangle  -   m_{sc}(z) \langle \bx,\by\rangle \big| \ge \frac{N^\e}{\sqrt{N\eta}}\Big) \le \frac{C_{D,\e}}{N^D}
\ee
for any fixed $\bx, \by$ with $\|\bx\|_2=\|\by\|_2=1$, 
 we cannot bring the supremum over all $\bx, \by$ inside the probability.
Clearly $\max_{\bx,\by} \langle \bx, G(z)\by\rangle$ would give the norm of $G$ which is $1/\eta$.

Furthermore, a common feature of all our estimates is that the local law in averaged  sense is one order more precise than the
entrywise or isotropic laws, e.g. for the precise form of \eqref{avert} we have
\be\label{perch2}
   \P\Big(  \big| \frac{1}{N}\tr G(z) -   m_{sc}(z)\big| \ge \frac{N^\e}{N\eta}\Big) \le \frac{C_{D,\e}}{N^D}.
\ee

\section{Physical motivations}\label{sec:motivation}

The primary motivation to study local spectral statistics of large random matrices comes
from nuclear and condensed matter physics where the matrix models a quantum 
Hamiltonian and its eigenvalues correspond to energy levels.
Other applications concern statistics
(especially largest eigenvalues of sample covariance matrices of the form $XX^*$ where
$X$ has independent entries), wireless communication  and neural networks. Here we focus only
on physical motivations.

\subsection{Basics of quantum mechanics}

We start with summarizing the basic setup of quantum mechanics. A quantum system is described by
a {\it configuration space} $\Sigma$, e.g. $\Sigma= \{ \uparrow, \downarrow\}$ for a single  spin, 
or $\Sigma = \Z^3$  for an electron hopping on an ionic lattice or $\Sigma=\R^3$ for an electron in vacuum.
Its elements $x\in \Sigma$ are called configurations and it is equipped with a natural measure (e.g. the
counting measure for discrete $\Sigma$ or the Lebesgue measure for $\Sigma=\R^3$).
The {\it state space} is a complex Hilbert space, typically the natural $L^2$-space of $\Sigma$, i.e. $\ell^2(\Sigma)=\C^2$
in case of a single spin or $\ell^2(\Z^3)$ for an electron in a lattice. Its elements are called wave functions, 
these are normalized functions $\psi \in \ell^2(\Sigma)$, with $\|\psi\|_2=1$. The quantum wave function 
entirely describes the quantum state. 
 In fact its overall phase does not carry measurable physical information;
wave functions $\psi$ and $e^{ic}\psi$ are indistinguishable for any real constant $c$.
This is because only quadratic forms of $\psi$ are measurable, i.e. 
only quantities of the form $\langle \psi, O\psi\rangle$ where $O$ is a self-adjoint operator.
The probability density  $|\psi(x)|^2$ on the configuration  space describes the probability to 
find the quantum particle at configuration  $x$.

The dynamics   of the quantum system, i.e. the process how $\psi$ changes in time,
 is described by the {\it Hamilton operator}, which is a self-adjoint operator
acting on the state space $\ell^2(\Sigma)$. If $\Sigma$ is finite, then it is an $\Sigma\times \Sigma$ hermitian matrix. The matrix elements
$H_{xx'}$ describe the quantum transition rates from configuration $x$ to $x'$. The dynamics of $\psi$ is described by the Schr\"odinger equation
\[
   i\partial_t \psi_t = H\psi_t
\]
with a given initial condition $\psi_{t=0}:=\psi_0$.
The solution is given by  $\psi_t = e^{-itH} \psi_0$. This simple formula
is however, quite hard to compute or analyze, especially for large times.
Typically one writes up the spectral decomposition of $H$ in the form 
$H =\sum_n\lambda_n |\bv_n\rangle\langle \bv_n|$, where $\lambda_n$ and $\bv_n$
are the eigenvalues and eigenvectors of $H$, i.e. $H\bv_n=\lambda_n\bv_n$.
Then
\[
    e^{-itH}\psi_0 =\sum_n e^{-it\lambda_n}  \langle \bv_n, \psi_0\rangle \bv_n =: \sum_n e^{-it\lambda_n}  c_n \bv_n.
\]
If $\psi_0$ coincides with one of the eigenvectors, $\psi_0=\bv_n$, then
the sum above collapses and
\[
 \psi_t=  e^{-itH}\psi_0 = e^{-it\lambda_n}  \bv_n.
\]
Since the physics encoded in the  wave function is insensitive to an overall phase, 
we see that eigenvectors remain unchanged along the quantum evolution.

Once $\psi_0$ is a genuine linear combination of several eigenvectors, quadratic forms
of $\psi_t$ become complicated:
\[
   \langle\psi_t, O \psi_t\rangle = \sum_{nm} e^{it(\lambda_m-\lambda_n)}  \bar{c}_m c_n \langle  \bv_m, O  \bv_n, \rangle .
\]
This double sum is highly oscillatory and subject to possible periodic and quasi-periodic behavior depending
on the commensurability of the eigenvalue differences $\lambda_m-\lambda_n$. Thus the statistics of the eigenvalues
carry important physical information on the quantum evolution.

The Hamiltonian $H$ itself can be considered as an observable, and the quadratic form $\langle \psi, H\psi\rangle$ 
describes the {\it energy} of the system in the state $\psi$. Clearly the energy is a conserved quantity
\[
   \langle \psi_t, H\psi_t\rangle = \langle e^{-itH}\psi_0, H e^{-itH} \psi_t\rangle= \langle \psi_0, H  \psi_t\rangle.
\]
The eigenvalues of $H$ are called {\bf energy levels} of the system.

{\it Disordered quantum systems} are described by random Hamiltonians, here the randomness comes from an external source
and is often described phenomenologically. For example, it can represent impurities in the state space (e.g. the
ionic lattice is not perfect) that we do not wish  to (or cannot) describe with a deterministic precision, only their statistical
properties are known.

\subsection{The ``grand'' universality conjecture for disordered quantum systems}\label{sec:grand}

The general belief is that  disordered quantum systems with ``sufficient'' complexity  are subject to  a strong dichotomy. They exhibit
one of the following two behaviors:  they are either in the {\it insulating} or in the {\it conducting} phase. 
These two phases are also called {\bf localization} and {\bf delocalization} regime. The behavior may depend
on the energy range: the same quantum system can be simultaneously in both phases but at different energies.

The {\bf insulator (or localized regime)} is characterized by the following properties:
\begin{itemize}
\item[1)] Eigenvectors are {\bf spatially localized}, i.e. the overwhelming mass of the probability density $|\psi(x)|^2\rd x$ is 
supported in a small subset of $\Sigma$. More precisely, there exists an $\Sigma'\subset \Sigma$, with $|\Sigma'|\ll |\Sigma|$ such that
\[
   \int_{\Sigma\setminus \Sigma'} |\psi(x)|^2 \rd x \ll 1
 \]
 \item[2)] {\bf Lack of transport}: if the state $\psi_0$ is initially localized, then it remains so (maybe on a larger domain) for all times.
 Transport is usually measured with the mean square displacement  if $\Sigma$ has a metric. For example, for $\Sigma=\Z^d$ we consider
 \be\label{mean square}
      \langle x^2 \rangle_t : = \sum_{x\in \Z^3} x^2 |\psi_t(x)|^2,
 \ee
 then localization means that
\[
  \sup_{t\ge 0}   \langle x^2 \rangle_t \le C
\]
 assuming that at time $t=0$ we had $\langle x^2 \rangle_{t=0} <\infty$. Strictly speaking this concept makes
 sense only if $\Sigma$ is infinite, but one can require that the constant $C$ does not depend on some relevant size
 parameter of the model.
 
 \item[3)] Green functions have a {\bf finite  localization length} $\ell$, i.e. the off diagonal  matrix elements of the
 resolvent decays exponentially (again for $\Sigma=\Z^d$ for simplicity)
 \[
     |G_{xx'}|  \le Ce^{-|x-x'|/\ell}.
 \]

 \item[4)] {\bf Poisson local eigenvalue statistics}: Nearby eigenvalues are statistically independent, i.e. they approximately
 form a Poisson point process after appropriate rescaling.
 
\end{itemize}

The {\bf conducting (or delocalized) regime} is characterized by the opposite features:
 
\begin{itemize}
\item[1)] Eigenvectors are {\bf spatially delocalized}, i.e.  the mass of 
 the probability density $|\psi(x)|^2$ is  not concentrated on a much  smaller subset of $\Sigma$.

 \item[2)] {\bf Transport via diffusion}: The mean square displacement \eqref{mean square}  grows diffusively, e.g. 
 for $\Sigma=\Z^d$
 \[
    \langle x^2 \rangle_t \approx Dt
\]
with some nonzero constant $D$ (diffusion constant) for large times. If $\Sigma$ is a finite part of $\Z^d$,
e.g. $\Sigma = [1,L]^d\cap \Z^d$, then this relation 
 should be modified
so that the growth of  $\langle x^2 \rangle_t $ with time can last only until the whole $\Sigma$ is exhausted.

\item[3)] The Green function does not decay exponentially, the localization length $\ell=\infty$.

 \item[4)] {\bf Random matrix local eigenvalue statistics}: Nearby eigenvalues are statistically strongly dependent, in particular
 there is a level repulsion. They 
  approximately
 form a GUE or GOE eigenvalue point process after appropriate rescaling. The symmetry type of the approximation is
 the same as the symmetry type of the original model (time reversal symmetry gives GOE).

\end{itemize}

The most prominent simple example for the conducting regime is the
Wigner matrices or more generally Wigner-type matrices.  They represent a quantum system where hopping from any site $x\in \Sigma$ 
to any other site $x'\in\Sigma$
is statistically equally likely (Wigner ensemble) or at least comparably likely (Wigner type ensemble).

Thus, a convenient way to represent the conducting regime is via a complete graph as illustrated below in Figure~\ref{hh}. This graph has one vertex for each of the $N=| \Sigma|$  states and an edge joins each pair of states.  The edges correspond
to the matrix elements $h_{xx'}$ in \eqref{Hdef1} and they are independent. For Wigner matrices there is no specific spatial structure  present, the 
system is completely homogeneous. Wigner type
ensembles model a system with an inhomogeneous  spatial structure, but it is still 
a {\bf mean field} model since most  transition rates are comparable. However, some results on Wigner type matrices
allow zeros in the matrix of variances $S$ defined in \eqref{def:Sigmamatrix}, i.e. certain jumps are  explicitly forbidden. 
 
\setlength\figurewidth{1.5cm}
\newcommand\numberCorners{7}
{\centering
\begin{figure}[H]
\begin{tikzpicture}[scale=1.8,baseline={(current bounding box.center)}]
\foreach \i in {1,...,\numberCorners} 
\node (circ\i) at ( 360/\numberCorners*\i:\figurewidth) [label=,circle,draw,fill=black,inner sep=.5mm] {};
\foreach \i in {1,...,\numberCorners}
\foreach \j in {\i,...,\numberCorners}
\pgfmathparse{0.4+2.0*rnd}
\draw[line width = \pgfmathresult pt] (circ\i) to (circ\j);
\end{tikzpicture}
\caption{Graph schematically indicating the configuration space of $N=|\Sigma|=7$ states with random 
quantum transition rates}\label{hh}
\end{figure}
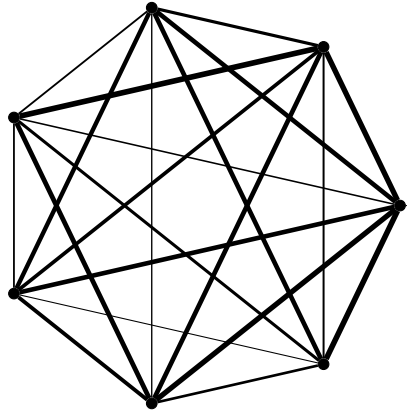
}

The delocalization of the eigenvectors (item 1) was presented in \eqref{delocbound}, while
item 4) is the WDM universality. The diffusive feature (item 2) is trivial since  due to the mean field
character, the maximal displacement is already achieved after $t\sim O(1)$.
Thus the {\bf Wigner matrix is in the  delocalized regime.}

It is not so easy to present a non-trivial example for the insulator regime. A trivial example is if $H$ is a diagonal matrix in the basis given by $\Sigma$,
with i.i.d. entries in the diagonal, then items 1)--4) of the insulator regime clearly hold. Beyond the diagonal, even a short range hopping
can become delocalized, for example the lattice Laplacian on $\Z^d$ has delocalized eigenvectors (plane waves).
However, if the Laplacian is perturbed by a random diagonal, then localization may occur -- this is the celebrated
{\it Anderson metal-insulator transition} \cite{And1958}, which we now discuss.

\subsection{Anderson model}\label{sec:anderson}

The prototype of the random Schr\"odinger operators is the Anderson model  on the $d$-dimensional
square lattice $\Z^d$. It
consists of a Laplacian (hopping term to the neighbors) and a random potential:
\be\label{rs}
    H = \Delta + \lambda V
\ee
acting on $\ell^2(\Z^d)$. The matrix elements of  the Laplacian are given by
\[
      \Delta_{x y} = {\bf 1}( |x-y|=1) 
\] 
and the potential is diagonal, i.e. 
\[
     V_{xy} = \delta_{xy} v_x,
 \]
 where $\{ v_x \; : \; i\in \Z^d\}$ is a collection of  real i.i.d. random variables sitting
 on the lattice sites.  For definiteness we assume that 
 \[
    \E v_x=0, \qquad \E v_x^2 =1
 \]
 and  $\lambda$ is a coupling parameter. 
 Notice that $\Delta$ is self-adjoint and bounded, while the potential at every site is
 bounded almost surely. For simplicity we may assume that the common distribution of $v$
 has bounded support, i.e. $V$, hence $H$ are  bounded operators. This eliminates
 some technical complications related to the proper definition of the self-adjoint extensions.

 \subsubsection{The free Laplacian}
  
 For $\lambda=0$, the spectrum is well known, the eigenvector equation $\Delta f = \mu f$, i.e. 
 \[
        \sum_{|y-x|=1} f_y = \mu f_x, \qquad \forall x\in \Z^d,
   \]
  has plane waves parametrized by the $d$-torus, $k=(k_1, k_2, \ldots , k_d)\in [-\pi, \pi]^d$  as eigenfunctions:
  \[
        f_x =   e^{i k\cdot x}, \qquad \mu = 2\sum_{i=1}^d \cos k_i .
  \]
  Although these plane waves are not $\ell^2$-normalizable, they still form a complete system
  of generalized eigenvectors for the bounded self-adjoint operator $\Delta$.
  The spectrum is the interval $[-2d, 2d]$ and it is a purely absolutely continuous spectrum
  (we will not need its precise definition if you are unfamiliar with it).
   Readers   uncomfortable
  with unbounded domains  can take a large torus $[-L, L]^d$, $L\in \N$, instead of $\Z^d$ 
  as the configuration space. Then everything is finite dimensional, and the wave-numbers $k$ are
  restricted to a finite lattice within the torus $[-\pi, \pi]^d$. Notice that the eigenvectors are still 
   plane waves, in particular they are completely delocalized. 
   
   One may also study the 
  time evolution $e^{it\Delta}$ (basically by Fourier transform) and one finds {\bf ballistic behavior}, i.e.
  for the mean square displacement \eqref{mean square} one finds
  \[
         \langle x^2 \rangle_t  = \sum_{x\in \Z^d} x^2 |\psi_t(x)|^2 \sim Ct^2, \qquad  \psi_t = e^{it\Delta}\psi_0
   \]
   for large $t$.  Thus for $\lambda=0$ the system in many aspects is in the delocalized regime.
   Since randomness is completely lacking,  it is not expected that other features of the delocalized regime hold, e.g.
   the local spectral statistics is not the one from random matrices -- it is  rather related to a lattice
   point counting problem. Furthermore, the eigenvalues have degeneracies, i.e. level repulsion,  a main
   characteristics for random matrices, does not hold.

\subsubsection{Turning on the randomness}

Now we turn on the randomness by taking some $\lambda\ne 0$.  This changes the behavior
of the system drastically in certain regimes. More precisely:

\begin{itemize}
\item In {\bf $d=1$ dimension } the system is in the localized regime as soon as $\lambda\ne 0$,
see \cite{GoldMolPastur1997}

\item In {\bf $d=2$ dimensions } On physical grounds it is conjectured that the system is localized for any $\lambda\ne0$
\cite{Vollhard1980}. No mathematical proof exists.

\item In the most important physical {\bf $d=3$ dimensions } we expect a phase transition: The system is
localized for large disorder, $|\lambda|\ge \lambda_0(d)$ or at the spectral edges 
\cite{FroSpe1983, AizMol1993}. For small disorder and away from the spectral edges  
delocalization is expected but there is no rigorous proof. This is the celebrated {\bf extended states or delocalization conjecture},
one of the few central  holy grails of mathematical physics.
\end{itemize}

Comparing random Schr\"odinger with random matrices, we may write up the matrix of the $d=1$ dimensional
operator $H$  \eqref{rs} in the basis given by $\Sigma=\llbracket 1, L\rrbracket$:
\[
  H = \Delta + \sum_{x=1}^L v_x = 
  \begin{pmatrix}  
v_1 & 1 &&&& \\
1 & v_2 &1&&& \\
&1 & \ddots &&& \\
&&&\ddots & 1 &\\
&&& 1&  v_{L-1} &1 \\
&&&& 1 &v_L
\end{pmatrix}.
\]
It is tridiagonal matrix with i.i.d. random variables in the diagonal
and all ones in the minor diagonal. It is a  {\bf short range model} as
immediate quantum transitions (jumps) are allowed only to the
nearest neighbors. Structurally this $H$  is very different from the
typical Wigner matrix \eqref{hh} where all matrix elements are roughly comparable
({\bf mean field model}).

\subsection{Random band matrices}

\newcommand{\rast}{\ast\nc}
\newcommand{\fS}{{\mathfrak S}}

Random band matrices naturally interpolate between the mean field Wigner ensemble and the
short range random Schr\"odinger operators. Let the state space be 
\[
\Sigma := [1,L]^d\cap\Z^d
\]
a lattice box of linear size $L $  in $d$ dimensions. The total dimension of the state space is $N=|\Sigma|=L^d$.
The entries of $H=H^*$ are centered, independent but not identically distributed -- it is like the Wigner type ensemble,
but without the mean field condition $s_{xy} =\E |h_{xy}|^2\le C/N$. Instead, we introduce a new parameter,
$1\le W\le L$  the {\bf bandwidth} or the interaction range. We assume  that the variances behave as
\[
   \E |h_{xy}|^2 =  \frac{1}{W^d} f\Big(\frac{ |x-y|}{W}\Big).
 \]
 In $d=1$ physical dimension the corresponding matrix is an $L\times L$ matrix with a nonzero band of
 width $2W$ around the diagonal. From any site a direct hopping of size $W$ is possible, see
 the figure below with $L=7$, $W=2$:
 
{\centering
\[
   H = {\begin{pmatrix}
   \rast & \rast & \rast  & 0 & 0&0&0 \\
   \rast & \rast & \rast & \rast & 0&0&0 \\
\rast & \rast & \rast & \rast & \rast&0&0 \\
0&\rast & \rast & \rast & \rast & \rast&0 \\
0&0&\rast & \rast & \rast & \rast & \rast \\
0&0&0&\rast & \rast & \rast & \rast  \\
0&0&0&0&\rast & \rast & \rast  
   \end{pmatrix}}
\qquad\qquad\mbox{
\begin{tikzpicture}[scale=0.8]
	\draw[] (0,0) -- (6,0);
	\foreach \i in {0,...,6}
		{
		\draw[fill=black] (\i,0) circle (0.14);
		}
	\draw (3,0) arc(0:180:1);
	\draw (3,0) arc(0:180:0.5);
	\draw (4,0) arc(0:180:0.5);
	\draw (5,0) arc(0:180:1);
\end{tikzpicture}}
\]
}
\bigskip

Clearly $W=L$ corresponds to the Wigner ensemble, while $W=1$ is very similar to the random Schr\"odinger
with its short range hopping. The former is delocalized, the latter is localized, hence there is a transition 
with can be probed by changing $W$ from 1 to $L$.
The following table summarizes ``facts'' from physics literature on the transition threshold:


   \underline{Anderson metal-insulator transition occurs at the following thresholds:}
\begin{align*}
    W\sim L^{1/2} & \qquad (d=1) \qquad \mbox{Supersymmetry \cite{FyoMir1991} }\\
    W \sim \sqrt{\log L} & \qquad (d=2)  \qquad 
\mbox{Renormalization group scaling \cite{Abrahams1979} }\\
    W \sim W_0(d) & \qquad (d\geq 3) \qquad \mbox{extended states conjecture \cite{And1958}}
\end{align*}
All these conjectures are mathematically  open, the most progress has been done in $d=1$. It is known
that we have localization in the regime  $W\ll L^{1/8}$ \cite{Sch2009} and delocalization for $W\gg L^{4/5}$  \cite{ErdKnoYauYin2013}.
The two point correlation function of the characteristic polynomial  was shown to be given by the Dyson sine kernel up to the threshold
 $W\gg L^{1/2}$ in \cite{TShc2014-2}.

In these lectures we restrict our attention to mean field models, i.e. band matrices will not be discussed. We nevertheless
mentioned them because they are expected to be easier than the short range random Schr\"odinger operators
and they still exhibit the Anderson transition in a highly nontrivial way.

\subsection{Mean field quantum Hamiltonian with correlation}

Finally we explain how correlated random matrices with a certain correlation decay are motivated. We again equip
the state space $\Sigma$ with a metric to be able to talk about ``nearby'' states.
It is then reasonable to assume  that  $h_{xy}$ and $h_{xy'}$ are correlated if $y$ and $y'$ are
close with a decaying correlation as $\mbox{dist}(y,y') $ increases.\\[7pt]
\centerline{
\includegraphics[width=0.33 \textwidth] {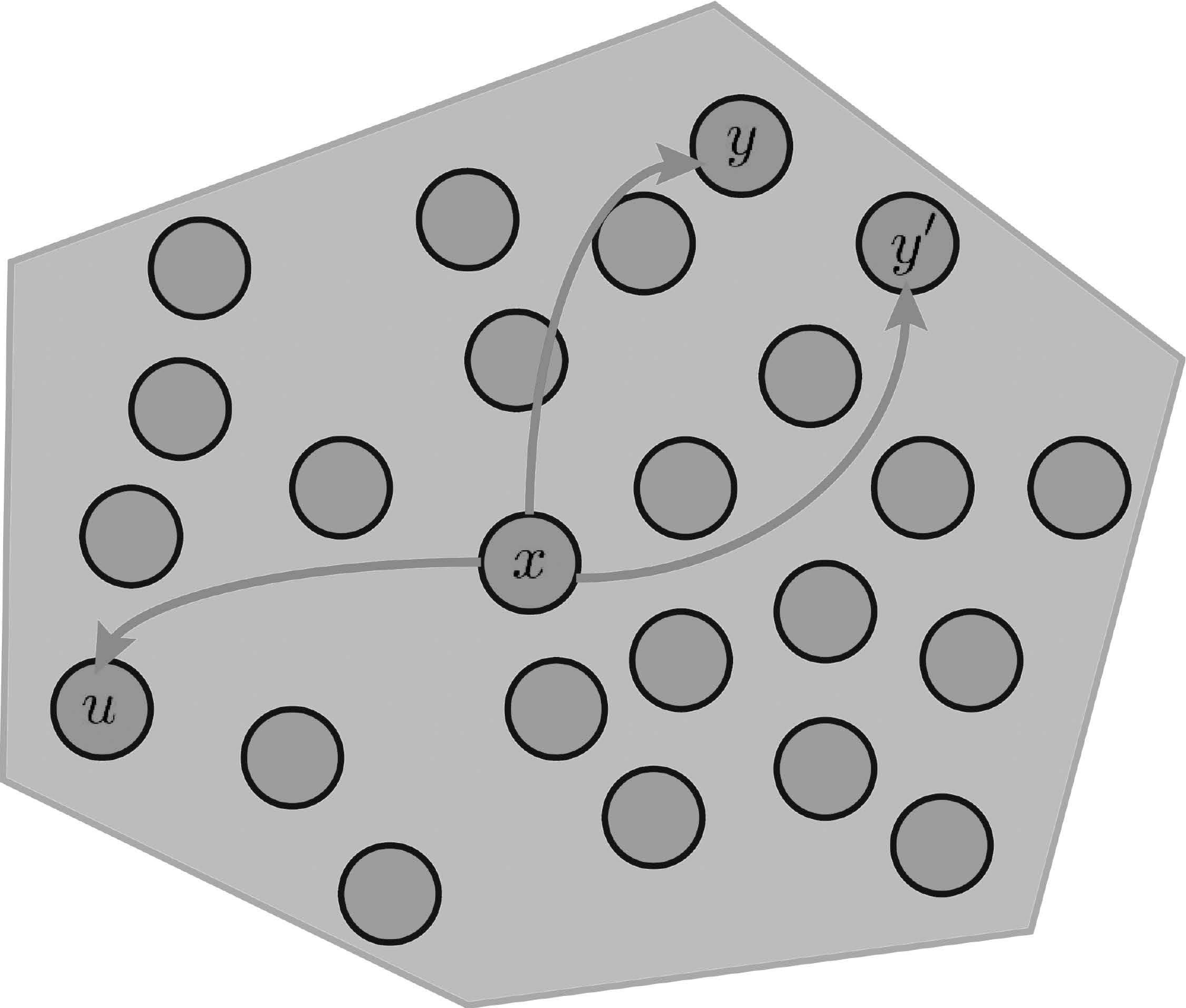}}
For example, in the figure $h_{xy}$ and $h_{xy'}$ 
are strongly correlated but $h_{xy}$ and $h_{xu}$ are not (or only  very weakly) correlated. We can combine this  feature
with an inhomogeneous spatial structure as in the Wigner-type ensembles.

\section{Results}\label{sec:results}

Here we list a few representative results with precise conditions. The results can be divided roughly into three categories:
\begin{itemize}
\item Properties of the solution of the Dyson equation, especially the singularity structure of the density of states and the
boundedness of the inverse of the  stability operator. This part  
of the analysis is deterministic.
\item Local laws, i.e. approximation of the (random) resolvent $G$ by the solution of the corresponding Dyson equation  with very high probability
down to the optimal scale $\eta\gg 1/N$. 
\item Bulk universality of the local eigenvalue statistics on scale $1/N$.
\end{itemize}

\subsection{Properties of the solution to the  Dyson equations}\label{sec:properties}

\subsubsection{Vector Dyson equation}

First we focus on the vector Dyson equation \eqref{vectordyson1} with a general symmetric variance matrix $S$ motivated by Wigner type matrices:
\be\label{vectordyson2}
       -\frac{1}{\bbm} =z+ S\bbm, \qquad \bbm\in \Cp^N, \quad z\in\Cp
\ee
(recall that the inverse of a vector $1/\bbm$ is understood component wise, i.e. $1/\bbm$ is an $N$ vector with components $(1/\bbm)_i=1/m_i$).
We may add an external source which is real vector $\ba\in \R^N$ and the equation is modified to
\be\label{vectordyson5}
       -\frac{1}{\bbm} =z- \ba+S\bbm, \qquad \bbm\in \Cp^N, \quad z\in\Cp,
\ee
but we will consider the $\ba=0$ case for simplicity.
We equip the space $\C^N$ with the maximum norm, 
\[
    \| \bbm\|_\infty : = \max_i |m_i|,
\]
and we let $\| S\|_\infty$ be the matrix norm induced by the maximum norm of vectors.

 We start with the existence and uniqueness
result for \eqref{vectordyson2}, see e.g. Proposition 2.1 in \cite{AjaErdKru2015}:

\begin{theorem}\label{exuniqvector}
The equation \eqref{vectordyson2} has a unique solution $\bbm=\bbm(z)$ for any $z\in\Cp$. For each $i\in\llbracket 1, N\rrbracket$ there is a probability measure
$\nu_i(\rd x)$ on $\R$  (called {\bf generating measure})
 such that $m_i$ is the Stieltjes transform of $v_i$:
\be\label{reps}
        m_i(z) = \int_\R \frac{\nu_i(\rd \tau)}{\tau-z},
\ee
and the support of all $\nu_i$ lie in the interval $[-2\| S\|_\infty^{1/2}, 2\|S\|_\infty^{1/2}]$.
In particular we have the trivial upper bound
\be\label{trivupper}
     \| \bbm(z)\|_\infty\le \frac{1}{\eta}, \qquad \eta=\im z.
 \ee
\end{theorem}
Recalling that the  {\bf  self-consistent density of states} was defined in \eqref{dos}
 via the inverse Stieltjes transform of  $\langle \bbm\rangle =\frac{1}{N}\sum m_i$, we see
that  
\[
  \varrho = \langle \bnu \rangle= \frac{1}{N}\sum_i \nu_i.
\]

We now list two assumptions on $S$, although for some results we will need only one of them:
\begin{itemize}
\item {\bf Boundedness:} We assume that there exists two positive constants $c, C$ such that
\be\label{sbound}
     \frac{c}{N}\le s_{ij}\le \frac{C}{N}
\ee
\item {\bf H\"older regularity:}
\be\label{holder}
    |s_{ij} -s_{i'j'}| \le \frac{C}{N}\, \Big[\frac{ |i-i'| + |j-j'|}{N}\Big]^{1/2} 
 \ee
\end{itemize}

We remark that the lower bound in \eqref{sbound} can be substantially weakened, in particular
large zero blocks are allowed. For example,  we may assume only that $S$ 
 has a substantial diagonal, i.e. $s_{ij}\ge \frac{c}{N}\cdot {\bf 1}( |i-j|\le \e N)$ with some fixed positive $c, \e$,
 but for simplicity of the presentation we follow \eqref{sbound}.  
 
 The H\"older regularity \eqref{holder} expresses a regularity on the order $N$ scale in the matrix. It
 can  be understood in the easiest way if we imagine that the matrix 
 elements $s_{ij}$ come from a macroscopic profile function $S(x,y)$ on $[0,1]\times[0,1]$ by the formula
 \be\label{sit}
    s_{ij} = \frac{1}{N} S\big( \frac{i}{N}, \frac{j}{N}\big).
 \ee
It is easy to check that if $S:[0,1]^2\to \R_+$ is H\"older continuous  with a H\"older exponent $1/2$, then \eqref{holder} holds.
In fact, the H\"older regularity condition can also be weakened to {\bf piecewise 1/2-H\"older regularity} (with finitely many pieces), in that case 
we assume that $s_{ij}$ is of the form \eqref{sit} with a  profile function $S(x,y)$  that is piecewise H\"older continuous with exponent $1/2$, i.e.
 there exists a fixed ($N$-independent) partition  $I_1\cup I_2\cup \ldots \cup I_n=[0,1]$
 of the unit interval into smaller intervals such that
 \be\label{piece holder}
    \max_{ab}\sup_{x, x'\in I_a}\sup_{y,y'\in I_b} \frac{ |S(x,y)- S(x',y')|}{|x-x'|^{1/2}+ |y-y'|^{1/2}} \le C.
 \ee

The main theorems summarizing the properties of the solution to \eqref{vectordyson2} are the following. The first theorem assumes only
\eqref{sbound} and it is relevant in the bulk. We will prove it later in Section~\ref{thm:bulkbound}.

\begin{theorem}\label{thm:bulkbound}
Suppose that $S$ satisfies \eqref{sbound}. Then we have the following bounds:
\be\label{bulk bound}
  \|\bbm(z)\|_\infty\lesssim \frac{1}{\varrho(z)+\mbox{dist}(z, \mbox{supp} \varrho)},  \qquad \varrho(z)\lesssim \im \bbm(z) \lesssim (1+|z|^2)\| \bbm(z)\|^2_\infty \varrho(z).
\ee
\end{theorem}

The second theorem additionally assumes \eqref{holder}, but the result is much more precise, in particular 
a complete analysis of singularities is possible.

\begin{theorem}\label{thm:cusp}[Theorem 2.6 in \cite{AEK1short}]
Suppose that $S$ satisfies \eqref{sbound}
 and it is H\"older continuous  \eqref{holder}
[or piecewise H\"older continuous \eqref{piece holder}]. Then we have the following:

\begin{itemize}
\item[(i)]  The generating measures have Lebesgue density, $\nu_i(\rd \tau)=\nu_i(\tau)\rd \tau$
and the {\bf generating densities} $\nu_i$ are uniformly 1/3-H\"older continuous, i.e.
\be\label{holdereq}
  \max_i \sup_{\tau\ne \tau'} \frac{|\nu_i(\tau)-\nu_i(\tau')|}{|\tau-\tau'|^{1/3}} \le C'.
\ee

\item[(ii)] The set on which $\nu_i$ is positive is independent of $i$:
\[
         \fS:= \{ \tau \in \R\; : \; \nu_i(\tau)>0 \}
 \]
 and it is a union of finitely many open intervals. If $S$ is H\"older continuous in the sense of \eqref{holder}, then $\fS$ consist of a single interval.
 
\item[(iii)] The restriction of $\bnu(\tau)$ to $\R\setminus\partial \fS$ is analytic in $\tau$ (as a vector-valued function).
 
 \item[(iv)]  At the (finitely many) points $\tau_0\in \partial \fS$ the generating density has one of the following two behaviors:
 \begin{itemize}
 \item [CUSP:] If $\tau_0$ is at the intersection of the closure of two connected components of $\fS$, then $\bnu$ has a cubic root singularity, i.e.
 \be\label{cusp}
   \nu_i(\tau_0+\om) = c_i |\om|^{1/3} + O(|\om|^{2/3})
\ee
with some positive constants $c_i$. 
 \item [EDGE:] If $\tau_0$ is not a cusp, then it is the right or left endpoint of a connected component of $\fS$ and $\bnu$ has
 a square root singularity at $\tau_0$:
 \be\label{edge}
      \nu_i(\tau_0\pm\om) = c_i \om^{1/2}+ O(\om), \qquad \om\ge 0,
\ee
with some positive constants $c_i$.
 \end{itemize}
 
 \end{itemize}
 The positive constant $ C'$ in \eqref{holdereq}  depends only on the constants $c$ and $C$ in the conditions \eqref{sbound} and \eqref{holder} [or  \eqref{piece holder}], 
 in particular it is independent of $N$. The constants $c_i$ in \eqref{cusp} and \eqref{edge} are also uniformly 
 bounded from above and below, i.e., $c'' \le c_i \le C''$, with some positive constants $c''$ and  $C''$ that, in addition to $c$ and $C$, may  also depend
 on the distance between the connected components of the generating density.   
\end{theorem}

Some of these statements will be proved in Section~\ref{sec:vector}. 
We now illustrate this  theorem by a few pictures.  The first picture indicates a nontrivial $S$-profile (different shades indicate
different values in the matrix) and the corresponding self-consistent density of states.

\begin{center}
\hbox{\null}
\includegraphics[width=0.5 \textwidth] {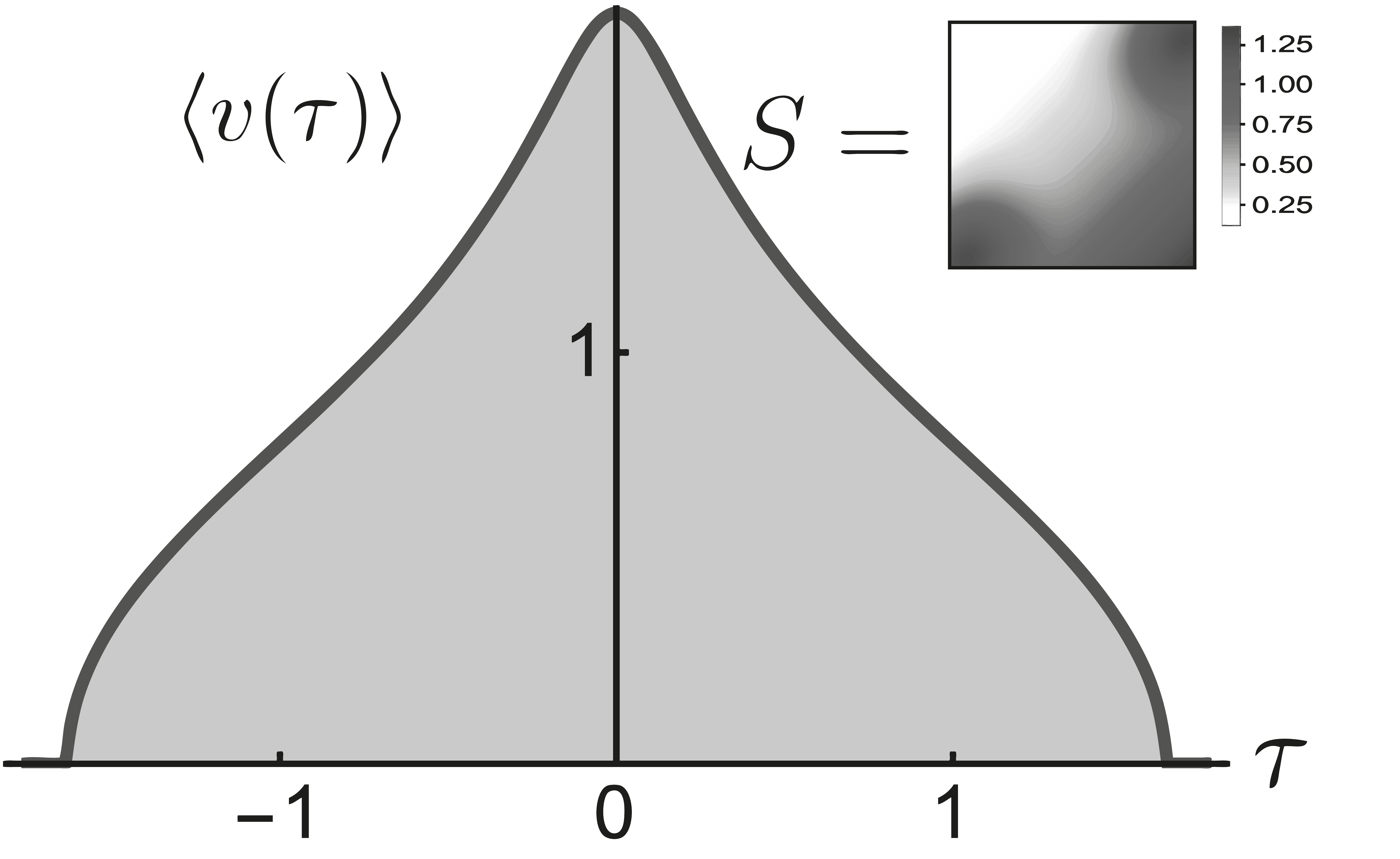}
\end{center}
In particular, we see that in general the density of states is not the semicircle if $\sum_{j} s_{ij} \neq\mbox{const}$.

The next pictures show how the support of the self-consistent density of states splits via cusps as the value of $s_{ij}$ slowly changes.
Each matrix below the pictures is the corresponding variance matrix $S$ represented as a $4\times 4$ block matrix with $(N/4)\times (N/4)$ blocks 
with constant entries.  Notice that the corresponding continuous  profile function $S(x,y)$ is only  piecewise H\"older (in fact, piecewise constant).
As the parameter in the diagonal blocks increases, a small gap closes at a cusp, then it develops a small 
local minimum.

\medskip
\begin{center}
\setlength\figurewidth{.3\textwidth}
\setlength\figureheight{.12\textheight}
%
%
.
\ee

\medskip

Cusps and splitting of the support are possible only if there is a discontinuity in the profile of $S$. If the above profile is smoothed
out (indicated by the narrow shaded region in the schematic picture of the matrix below), then the support  becomes a single interval
with a specific smoothed out ``almost cusp''.
\begin{center}
\includegraphics[width=0.2250 \textwidth] {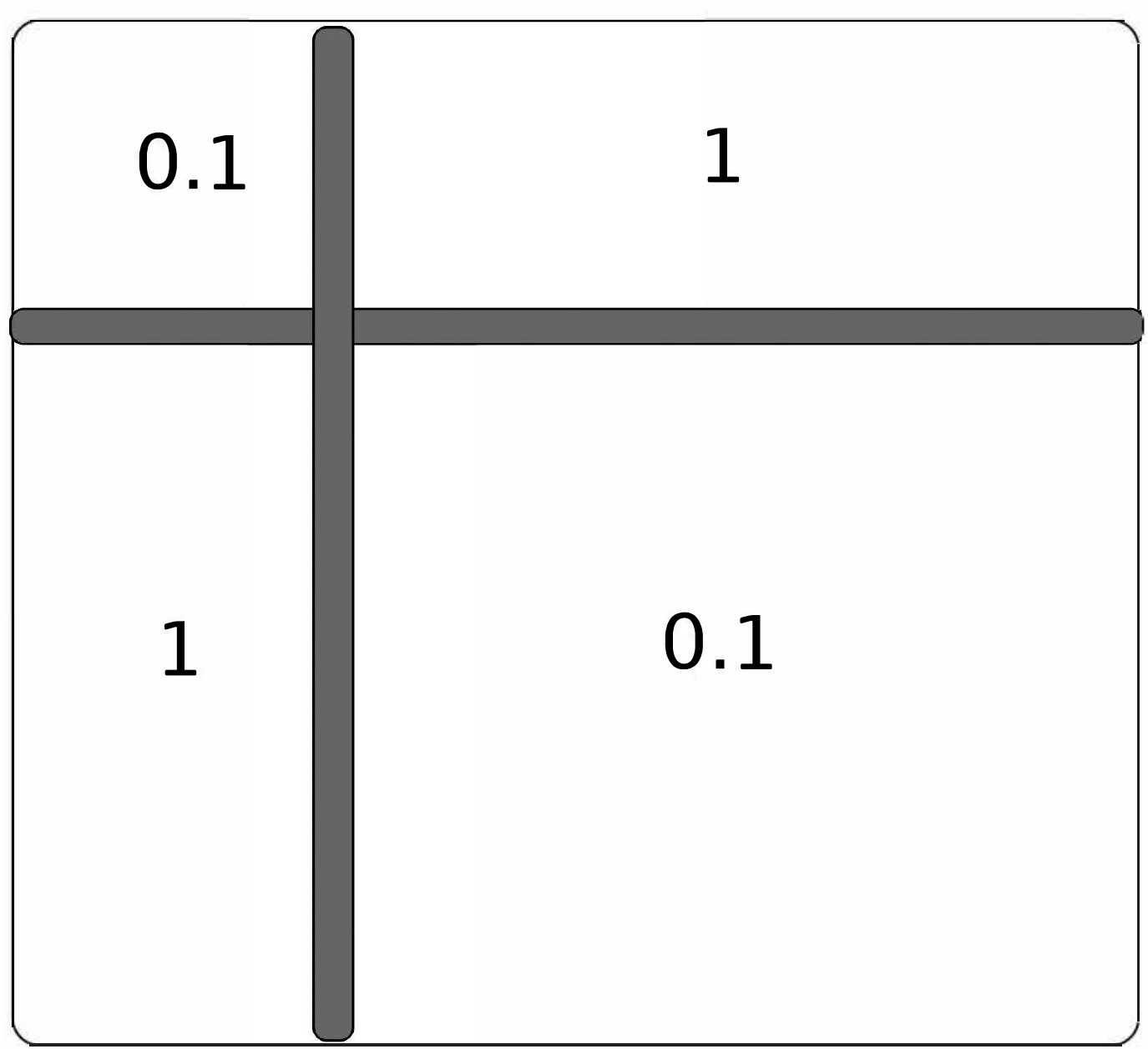}
\end{center}

Finally we show the universal shape of the singularities and near singularities in the self-consistent density of states.
The first two pictures are the edges and cusps, below them the approximate form of the density near
the singularity in terms of the parameter $\om =\tau-\tau_0$, compare with \eqref{edge} and \eqref{cusp}:

\begin{center}
	\medskip
\includegraphics[width=0.4 \textwidth] {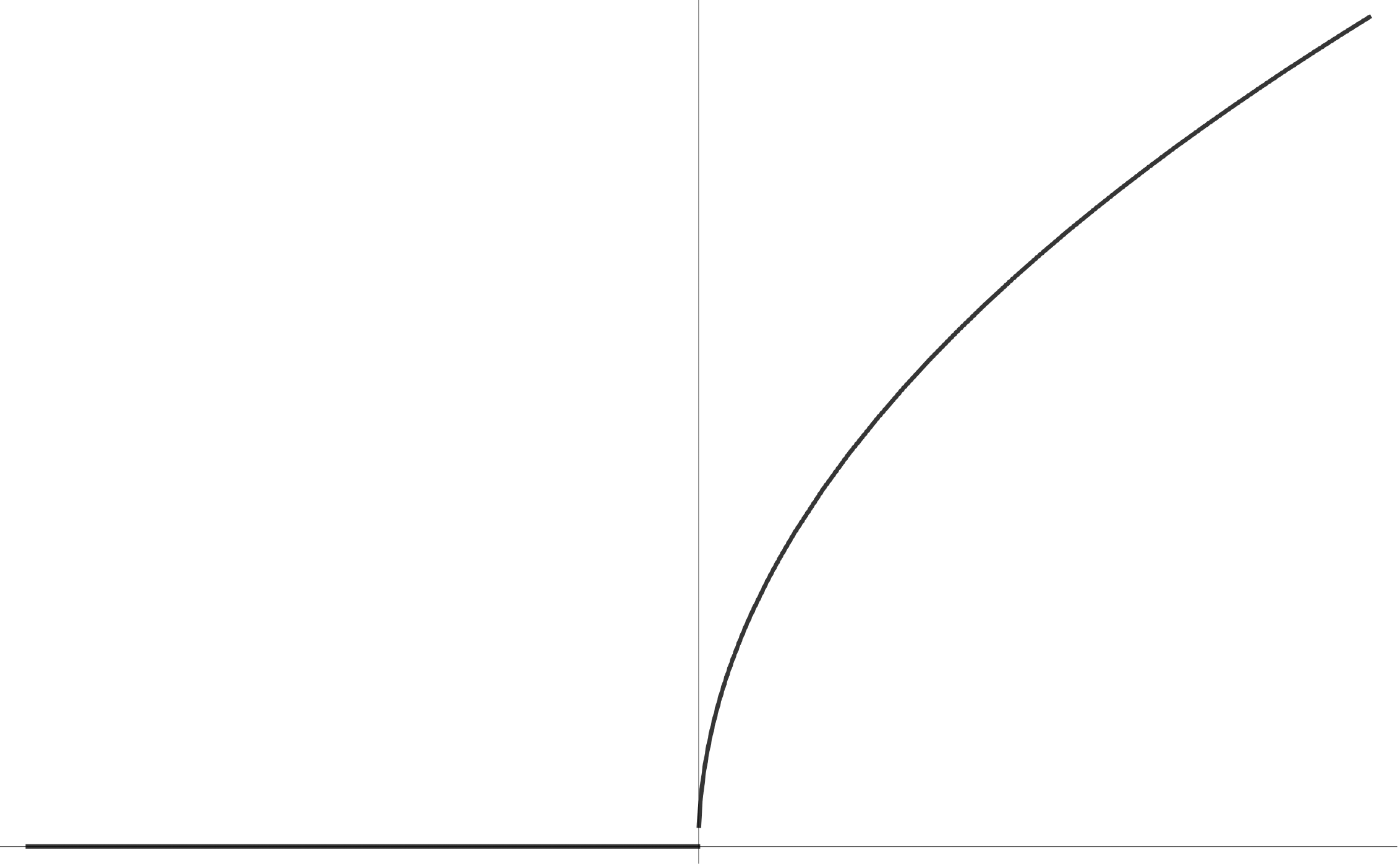} \hspace{2cm}
\includegraphics[width=0.4 \textwidth] {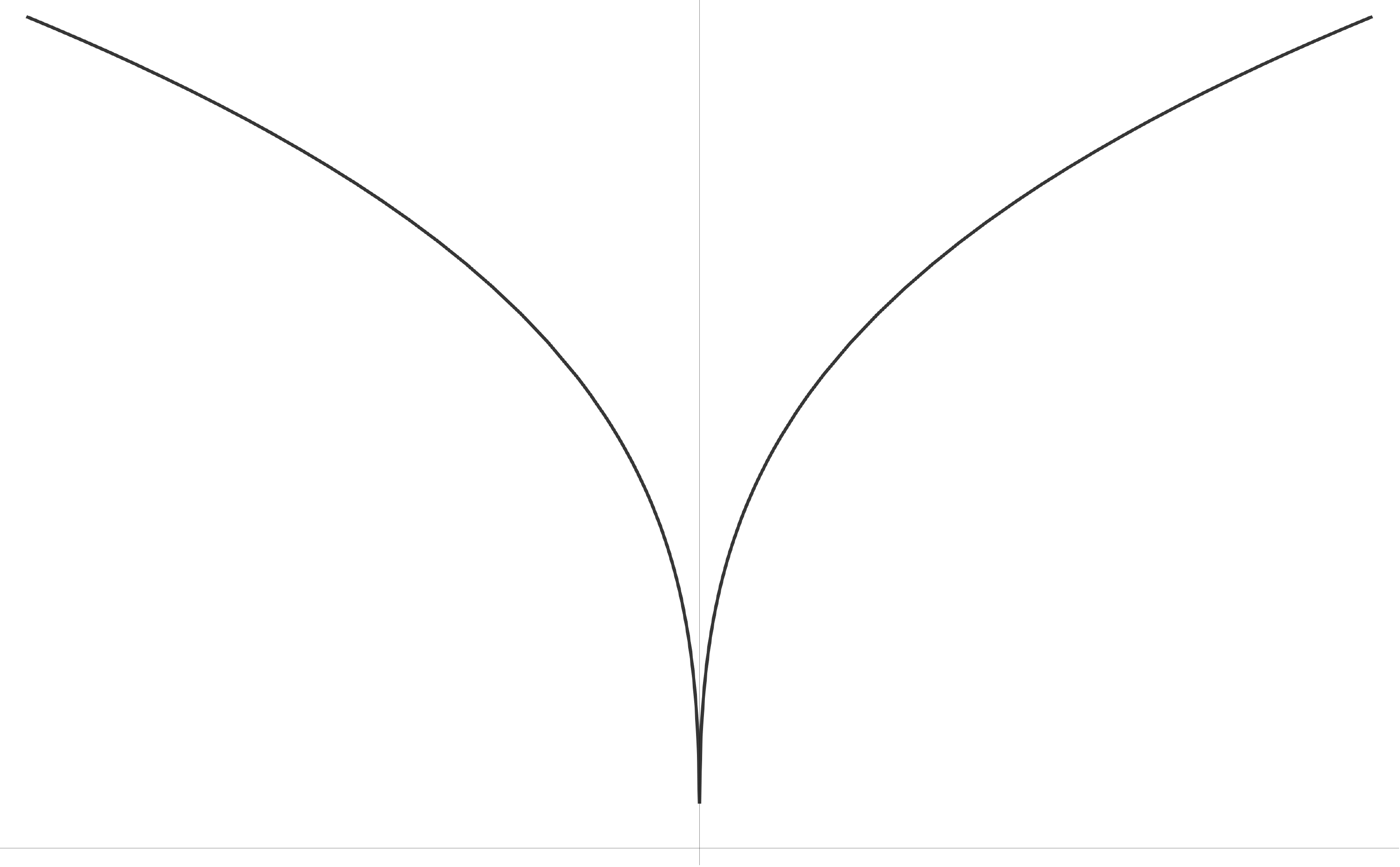}

\centerline{{  Edge, \nc $\sqrt{\omega}$ singularity  \hspace{2.5cm}   Cusp, \nc $|\omega|^{1/3}$ singularity}}
	\medskip
\end{center}

The next two pictures show the asymptotic form of the density right before and after the cusp formation.
The relevant parameter $t$ is an appropriate rescaling of $\omega$; the size of the gap (after the cusp formation)
 and the minimum value of the density (before the cusp formation) set the relevant length scales on which 
 the universal shape emerges:

\begin{center}
\includegraphics[width=0.35 \textwidth] {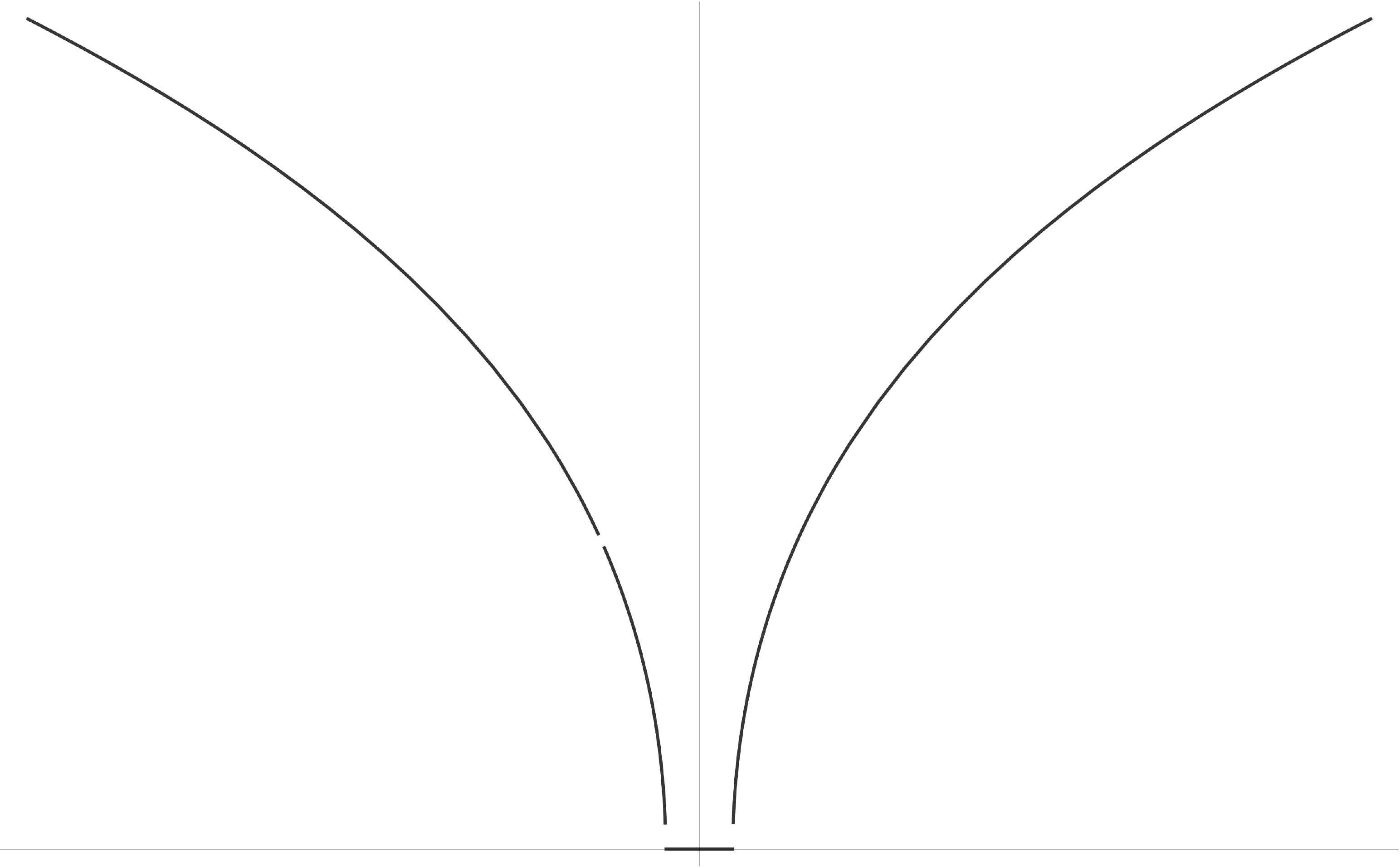} \hspace{2cm}
\includegraphics[width=0.35 \textwidth] {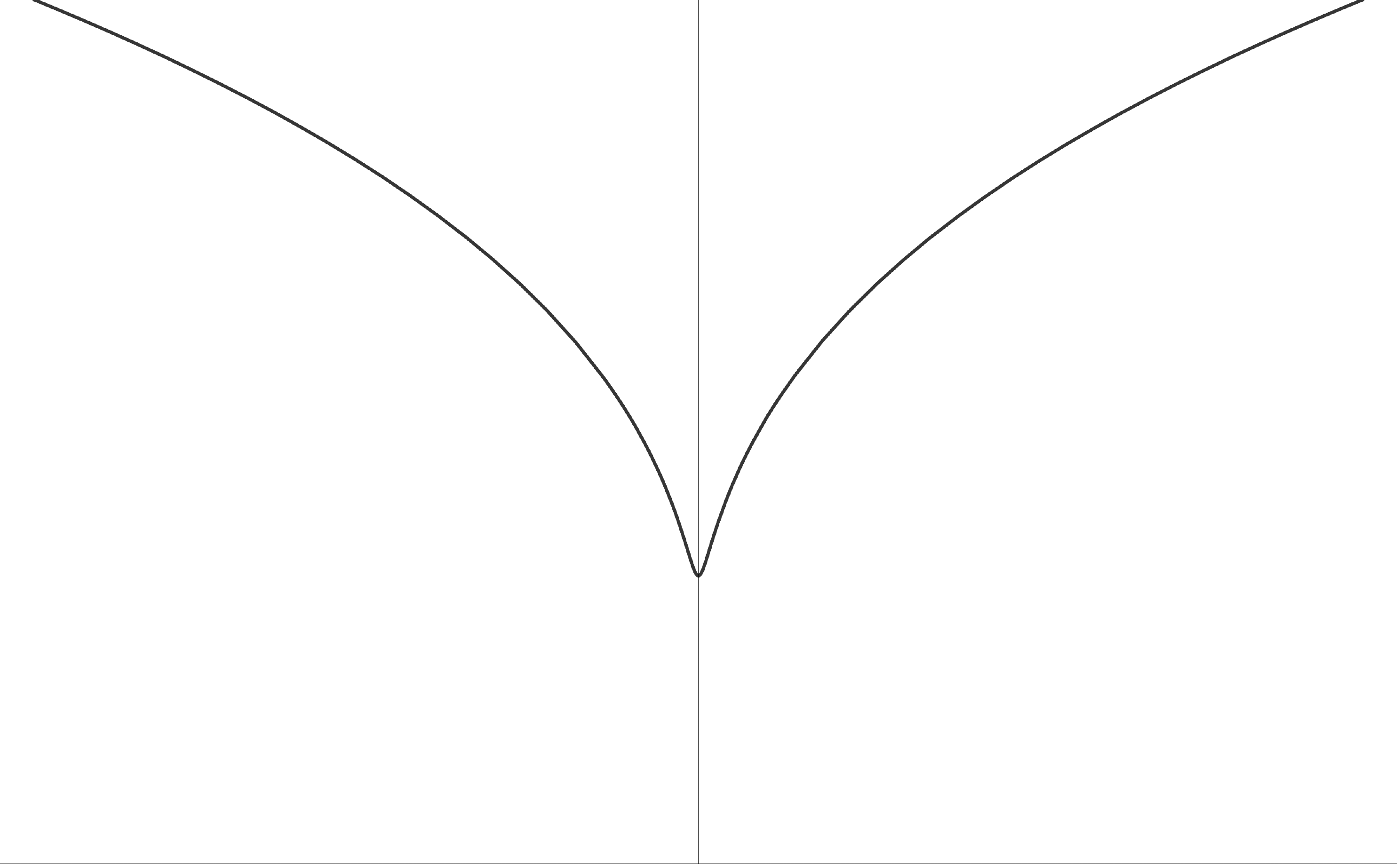}
\centerline{{ \hspace{.3cm} Small-gap \nc \hspace{4cm} \ Smoothed cusp \nc}}\\[6pt]
\centerline{$\frac{(2+t)t}{1+ (1+t +\sqrt{(2+t)t})^{2/3} + (1+t -\sqrt{(2+t)t})^{2/3}}$ \hspace{1cm}
$\frac{\sqrt{1+t^2}}{(\sqrt{1+t^2}+t)^{2/3} + (\sqrt{1+t^2}-t)^{2/3}-1}-1$}
\centerline{\hspace{1.5cm}$t := \frac{|\omega|}{\mbox{\small gap}}$, \hspace{4cm} $t:= \frac{|\omega|}{(\mbox{\small {minimum of $\varrho$} })^{3} }$}
\centerline{}
\end{center}

We formulated the vector Dyson equation in a discrete setup for $N$ unknowns
but  it can be considered in a more abstract setup as follows. 
For a measurable space $\mathfrak{A}$ and a subset $\mathbb{D}\subseteq \C$ of the complex numbers, we denote by $\Bounded(\mathfrak{A},\mathbb{D})$ the space of bounded measurable functions on $\mathfrak{A}$ with values in $\mathbb{D}$. 
Let  $(\xSpace,\xMeasure(\rd x))$ be a measure space with bounded positive (non-zero) measure $\pi$. Suppose we are given a real valued  $a\in \Bounded(\xSpace, \R)$ and a non-negative, symmetric, $s_{xy}=s_{yx}$, function $s\in \Bounded(\xSpace^2, \R_0^+)$. Then we consider the  \emph{quadratic vector equation (QVE)},
\begin{equation}
\label{QVE}
-\frac{1}{m(z)}\,=\, z-a+Sm(z)  \,,\qquad  z \in \Cp\,,
\end{equation}
for a function $m:\Cp \to \Bounded(\xSpace, \Cp), \, z \mapsto m(z)$, where $S:\Bounded(\xSpace, \C)\to \Bounded(\xSpace, \C)$ is the integral operator with kernel $s$,
\[
(Sw)_x\,:=\, \int s_{x y} \1w_y\1\xMeasure(\rd y)\,, \qquad x \in \xSpace\,,\; w \in \Bounded(\xSpace, \C)\,.
\]

We equip the space $\Bounded(\xSpace, \C)$ with its natural  supremum norm,
\[
\norm{w}\,:=\, \sup_{x \1\in\1\xSpace}\2|w_x|\,, \qquad w \in \Bounded(\xSpace,\C)\,.
 \]
With this norm $\Bounded(\xSpace, \C)$ is a Banach space. 
All results stated in Theorem~\ref{thm:cusp} are valid in this more general setup, for details, see
\cite{AjaErdKru2015}. The special case we discussed above corresponds to 
\[
   \xSpace: = \Big\{\frac{1}{N}, \frac{2}{N}, \ldots , \frac{N}{N}\Big\}, \qquad \xMeasure(\rd x) = \frac{1}{N}\sum_{i=1}^N \delta\big(x-\frac{i}{N}\big).
\]
The scaling here differs from \eqref{sit} by a factor of $N$, since now $s_{xy} = S\big( x, y\big)$, $x, y \in \xSpace$
in which case there is an infinite dimensional limiting equation with $\xSpace=[0,1]$
and $\xMeasure(\rd x)$ being the Lebesgue measure. If  $s_{ij}$ comes from a continuous profile, \eqref{sit}, then
in the $N\to\infty$ limit, the vector Dyson equation becomes
\[
   -\frac{1}{m_x(z)} = z+ \int_0^1 S(x,y) m_y(z) \rd y, \qquad x\in  [0,1],\quad z\in \Cp.
\]

\subsubsection{Matrix Dyson equation}

The matrix version of the Dyson equation naturally arises in the study of correlated random matrices, see Section~\ref{sec:cum}
and Section~\ref{sec:corr}.  
 It takes the form
\be\label{mde4}
      I + (z+\cS[M])M=0, \qquad \im M > 0, \quad \im z>0,   \qquad (MDE)
\ee
where we assume that $\cS: \C^{N\times N}\to \C^{N\times N}$ is a linear operator that is
\begin{itemize}
\item[1)] symmetric with respect to the Hilbert-Schmidt scalar product. In other words,  $\tr R^* \cS[T]=\tr \cS[R]^* T$
for any matrices $R, T\in \C^{N\times N}$;
\item[2)] positivity preserving, i.e. $\cS[R]\ge 0$ for any $R\ge 0$.
\end{itemize}
Somewhat informally we will refer to linear maps on the space of matrices as {\bf superoperators} to distinguish them
from usual matrices.

Originally, $\cS$ is defined in \eqref{Sdef} as a covariance operator of a hermitian random matrix $H$,
but it turns out that \eqref{mde4} can be fully analyzed solely under these two conditions 1)--2).
It is straightforward  to check that $\cS$ defined in \eqref{Sdef} satisfies the conditions 1) and 2).

Similarly to the vector Dyson equation \eqref{vectordyson5}  one may add an external source  $A=A^*\in \C^{N\times N}$
and
consider
\be\label{mde7}
    I + (z-A+\cS[M])M=0, \qquad \im M> 0
\ee
but these notes will be restricted to $A=0$. We remark that instead of finite dimensional matrices, 
a natural extension of \eqref{mde7} can be considered on a general  von Neumann algebra, see \cite{AEK2018}
for an extensive study.

The matrix Dyson equation \eqref{mde7} is a generalization of the vector Dyson equation  \eqref{vectordyson5}.
Indeed, if $\mbox{diag}(\bv)$ denotes the diagonal matrix  with the components of the vector $\bv$ in the
diagonal, then \eqref{mde7} reduces to \eqref{vectordyson5} 
with the identification $A=\mbox{diag}(\ba)$, $M= \mbox{diag}(\bbm)$ and  $\cS (\mbox{diag}(\bbm)) = \mbox{diag}(S\bbm)$.
The solution $\bbm$ to the vector Dyson equation was controlled in the maximum norm $\|\bbm\|_\infty$; for the matrix Dyson equation 
the analogous natural norm is the Euclidean matrix (or operator) norm, $\| M\|_2$, given by
\[ 
    \| M\|_2 = \sup\{ \| M\bx\|_2 \; : \; \bx\in\C^N, \; \|\bx\|_2=1\}.
\]
Clearly, for diagonal matrices we have
$\| \mbox{diag}(\bbm)\|_2 = \|\bbm\|_\infty$.
Correspondingly, the natural norm on the superoperator $\cS$
is the norm $\|\cS\|_2: =\| \cS\|_{2\to 2}$ induced by the Euclidean norm on matrices, i.e.
\[
  \| \cS\|_2 := \sup\{ \| \cS[R]\|_2 \; : \; R\in \C^{N\times N}, \; \| R\|_2=1\}.
\]

Similarly to  Theorem~\ref{exuniqvector}, we  have an existence and uniqueness result for the solution (see \cite{Helton2007-OSE})
 moreover,  we have a Stieltjes transform representation (Proposition 2.1 of \cite{AEK5}):

\begin{theorem}\label{mdeexuniq}
For any $z\in\Cp$, the MDE \eqref{mde4} with the side condition $\im M>0$ has a unique solution $M=M(z)$
that is analytic in the upper half plane.  The solution admits a Stieltjes transform representation
\be\label{rep matrix}
  M(z) = \int_\R \frac{V(\rd\tau)}{\tau-z}
\ee
where $V(\rd\tau)$ is a positive semidefinite matrix valued measure on $\R$ with normalization $V(\R)=I$.
In particular
\be\label{Mup}
   \| M(z)\|_2\le \frac{1}{\im z}.
\ee
The support of this measure lies in $[-2\|\cS\|_2^{1/2}, 2\| \cS\|_2^{1/2}]$. 
\end{theorem}
The solution $M$ is called the {\bf self-consistent Green function} or {\bf self-consistent resolvent} since it will be
used as a computable deterministic approximation to the random Green function $G$.

From now on we assume the following {\it flatness} condition on $\cS$ that is the matrix analogue of the
 boundedness condition \eqref{sbound}:


{\bf Flatness condition:} The operator $\cS$ is called {\bf flat} if there exists two positive constants, $c, C$, independent of $N$, such that
\be\label{flatness}
    c \langle R \rangle \le \cS[R] \le C\langle R \rangle, \qquad \mbox{where} \quad \langle R \rangle : = \frac{1}{N}\tr R
\ee
holds for any positive definite matrix $R\ge 0$.

Under this condition we have the following quantitative results on the solution $M$ (Proposition 2.2 and  Proposition 4.2 of \cite{AEK5}):


\begin{theorem}
Assume that $\cS$ is flat, then the holomorphic function $\langle M\rangle : \Cp\to \Cp$ is the Stieltjes transform of a H\"older continuous 
probability density $\varrho$  w.r.t. the Lebesgue measure:
\[
  \langle V(\rd \tau)\rangle =\varrho(\tau)\rd \tau
\]
i.e.
\be\label{matrixholder}
   |\varrho(\tau_1)-\varrho(\tau_2)|\le C|\tau_1-\tau_2|^\e
\ee
with some H\"older regularity exponent $\e$, independent of $N$ ($\e=1/100$ would do). 
The density $\varrho$ is called the {\bf self-consistent density of states.} Furthermore, $\varrho$ is real analytic on the 
open set $\fS: =\{ \tau\in \R\; ; \;\varrho(\tau)>0\}$ which is called the {\bf self-consistent bulk spectrum}.
For the solution itself we also have
\be\label{Mupper}
  \| M(z)\|_2\le  \frac{C}{\varrho(z) + \mbox{dist}(z, \fS)}
  \ee
  and
  \[
   c\varrho(z) \le \im M(z)\le C\| M(z)\|^2_2 \varrho(z),
 \]     
  where $\varrho(z)$ is the harmonic extension of $\varrho(\tau)$ to the upper half plane.
  In particular, in the bulk regime of spectral parameters, 
  where $\varrho(\re z)\ge \delta$ for some fixed $\delta>0$, we see that  $M(z)$ is bounded and $\im M(z)$ is comparable
   (as a positive definite matrix)
  with $\varrho(z)$.
\end{theorem}

Notice that unlike in the analogous Theorem~\ref{thm:cusp} for the vector Dyson equation, 
here we do not assume any regularity on $\cS$, but the conclusion is weaker. We do not 
get H\"older exponent 1/3 for the self-consistent density of states $\varrho$. Furthermore,
cusp and edge analysis
would also require further conditions on $\cS$. Since in the correlated case
we focus on the bulk spectrum, i.e. on spectral parameters $z$ with $\re z\in \fS$, 
we will not need detailed information about the density near the spectral edges.
A detailed analysis of the singularity structure of the solution to \eqref{mde7},
in particular a theorem analogous to Theorem~\ref{thm:cusp}, has been given in \cite{AEK2018}.
The corresponding edge universality for correlated random matrices  was proven in \cite{AEKS}.

\subsection{Local laws for Wigner-type and correlated random matrices}

We now state the precise form of the local laws. 

\begin{theorem}[Bulk local law for Wigner type matrices, Corollary 1.8 from \cite{AEK2}]\label{thm:type}
Let $H$ be a centered Wigner type matrix with  bounded variances $s_{ij} =\E |h_{ij}|^2$ 
i.e. \eqref{sbound} holds. Let $\bbm(z)$ be the solution to the vector Dyson equation
\eqref{vectordyson5}.  If the uniform moment condition \eqref{mmm} for the matrix elements,
then the local law in the bulk holds. If we fix positive constants $\delta, \gamma, \e$ and $D$, then 
for any spectral parameter $z=\tau +i\eta$ with 
\be\label{domain12}
  \varrho(\tau)\ge\delta, \qquad  \eta\ge N^{-1+\gamma}
\ee
we have the entrywise local law
\be\label{perch12}
  \max_{ij}  \P\Big(  \big| G_{ij}(z) -  \delta_{ij} m_i(z)\big| \ge \frac{N^\e}{\sqrt{N\eta}}\Big) \le \frac{C}{N^D} \,,
  \ee
and, more generally, the isotropic law that for non-random normalized vectors  $\bx, \by\in \C^N$,
\be\label{perch14}
  \max_{ij}  \P\Big(  \big| \langle \bx, G(z)\by\rangle  -  \langle \bx, \bbm(z)\by\rangle \big| \ge \frac{N^\e}{\sqrt{N\eta}}\Big) \le \frac{C}{N^D} .
  \ee
Moreover for any non-random vector $\bw=(w_1, w_2, \ldots)\in \C^N$ with $\max_i |w_i|\le 1$ we have the averaged local law
\be \label{perch13}
\P\Big(  \big| \frac{1}{N} \sum_i w_i\big[G_{ii}(z) -  m_i(z)\big]\big| \ge \frac{N^\e}{N\eta}\Big) \le \frac{C}{N^D},
\ee
in particular (with $w_i=1$) we have
\be \label{perch16}
\P\Big(  \big| \frac{1}{N} \tr G (z) -  \langle \bbm(z) \rangle )\big| \ge \frac{N^\e}{N\eta}\Big) \le \frac{C}{N^D}.
\ee
The constant $C$  in \eqref{perch12}--\eqref{perch16} is
independent of $N$ and the choice of $w_i$, but it depends on $\delta, \gamma, \e, D$, the
constants in \eqref{sbound}  and the sequence $\mu_p$ bounding the
moments in \eqref{mmm}.
\end{theorem}

As we explained around \eqref{perch1},   in the entrywise local law \eqref{perch12} one may bring both superma
on  $i,j$ and on the spectral parameter $z$
inside the probability, i.e. one can guarantee that $G_{ij}(z)$ is close to $m_i(z)\delta_{ij}$ simultaneously for
all indices and spectral parameters in the regime \eqref{domain12}. Similarly, $z$ can be brought inside the probability in \eqref{perch14} and 
\eqref{perch13}, but  the  isotropic law \eqref{perch14} cannot hold simultaneously for all $\bx,\by$ and similarly, the
averaged law \eqref{perch13}  cannot simultaneously hold
for all $\bw$.

We formulated the local law only under the boundedness condition \eqref{sbound} but only  in the bulk of the spectrum for simplicity.
Local laws near the edges and cusps require much more delicate analysis and  some type of regularity on $s_{ij}$, e.g.
the 1/2-H\"older regularity introduced in \eqref{holder} would suffice. Much easier is the regime outside of the spectrum. The precise statement
is found in Theorem 1.6 of \cite{AEK2}.

For the correlated matrix we have the following local law
from  \cite{AEK5}:

\begin{theorem}[Bulk local law for correlated matrices]\label{thm:univcorr}
Consider a  random   hermitian  matrix  $H\in \C^{N\times N}$ with correlated  entries. 
Define the self-energy super operator $\cS$ as
\be\label{css}
  \cS[R] = \E \big[ H R H\big]
\ee
acting on any matrix $R\in \C^{N\times N}$.
Assume that the flatness condition \eqref{flatness} and the moment 
condition \eqref{mmm} hold. We also assume an exponential decay of  correlations in the form
\be\label{coroskari}
   \mbox{Cov}\Big( \phi( W_A); \psi(W_B)\Big)\le C(\phi,\psi)\; \; e^{-d(A,B)}.
  \ee
Here $W=\sqrt{N}H$ is the rescaled random matrix, $A, B$ are two subsets of the index set
 $\llbracket 1, N \rrbracket\times \llbracket 1, N \rrbracket $,
 the distance $d$ is the usual Euclidean distance between the sets $A\cup A^t$ and $B\cup B^t$
 and $W_A=(w_{ij})_{(i,j)\in A}$, see figure below.
  Let $M$ be the self-consistent Green function, i.e. the solution of the matrix Dyson equation \eqref{mde4} with $\cS$ given in \eqref{css}, 
   and consider a spectral parameter in the bulk, i.e.
   $z=\tau +i\eta$ with 
\be\label{domain14}
  \varrho(\tau)\ge\delta, \qquad  \eta\ge N^{-1+\gamma}
\ee
Then 
for any non-random normalized vectors  $\bx, \by\in \C^N$ we have the isotropic local law
\be\label{perch14cor}
   \P\Big(  \big| \langle \bx, G(z)\by\rangle  -  \langle \bx, M(z)\by\rangle \big| \ge \frac{N^\e}{\sqrt{N\eta}}\Big) \le \frac{C}{N^D} ,
  \ee
  in particular we have the entrywise law
  \be\label{perch16cor}
   \P\Big(   \big| G_{ij}(z)  -  M_{ij}(z) \big| \ge \frac{N^\e}{\sqrt{N\eta}}\Big) \le \frac{C}{N^D} ,
  \ee
  for any $i,j$.
Moreover for any fixed (deterministic) matrix $T$ with $\|T\|\le 1$, we have the averaged local law
\be \label{perch13cor}
\P\Big(  \Big| \frac{1}{N} \tr T\big[G(z) - M(z)\big] \Big| \ge \frac{N^\e}{N\eta}\Big) \le \frac{C}{N^D}.
\ee
The constant $C$  is
independent of $N$ and the choice of $\bx,\by $, but it depends on $\delta, \gamma, \e, D$, the constants in \eqref{flatness}
 and the sequence $\mu_p$ bounding the
moments in \eqref{mmm}.
\end{theorem}

In our recent paper \cite{EKScorrelated}, we substantially relaxed the condition on the correlation decay \eqref{coroskari}
to the form
\[
   \mbox{Cov}\Big( \phi( W_A); \psi(W_B)\Big)\le \frac{ C(\phi,\psi)}{1+ d^2} e^{-d/N^{1/4}}, \qquad d=d(A,B),
 \]
and a similar condition on higher order cumulants, see \cite{EKScorrelated} for the precise forms.

In Theorem~\ref{thm:univcorr},  we again formulated the result only in the bulk, but similar (even stronger) local law is available
for energies $\tau$  that are separated away from the support of $\varrho$.

\centerline{
\includegraphics[width=0.3 \textwidth] {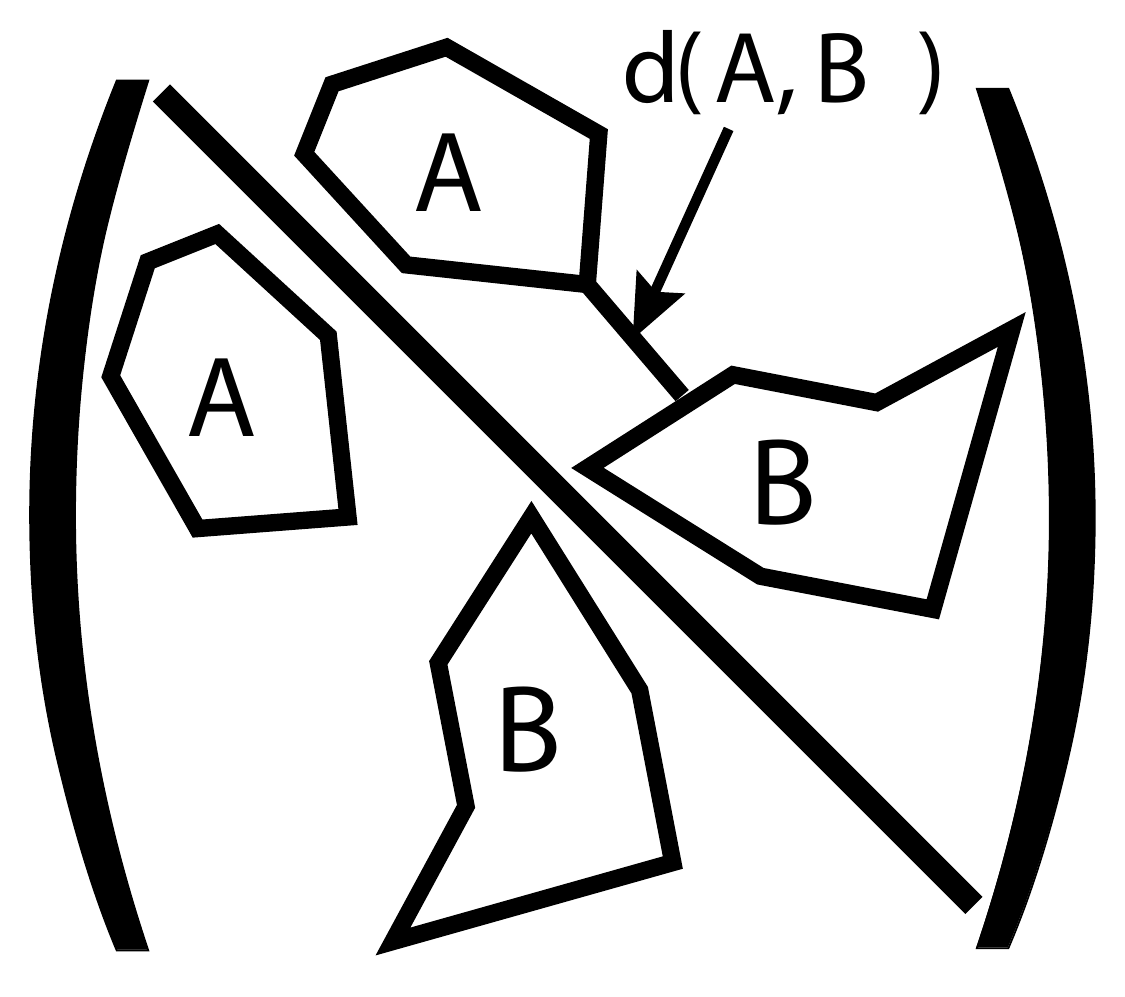}}

In these notes we will always assume that $H$ is centered, $\E H=0$ for simplicity, but our result holds in the 
general case as well. In that case $\cS$ is given by 
\[
 \cS[R] = \E \big[ (H-\E H) R (H-\E H)\big]
\]
and $M$ solves the MDE with external source $A:= \E H$, see \eqref{mde7}.

\subsection{Bulk universality and other consequences of the local law}

In this section we give precise theorems of three important consequences of the local
law. We will formulate the results in the simplest case, in the bulk. We give some sketches
of the proofs. Complete arguments for these results can be found 
in the papers \cite{AEK2} and \cite{AEK5, EKScorrelated}.

\subsubsection{Delocalization}

The simplest consequence of the entrywise local law is the  delocalization of the
eigenvectors as explained in Section~\ref{sec:resolve}. The precise formulation goes as follows:

\begin{theorem}[Delocalization of bulk eigenvectors] Let $H$ be a Wigner type or, more generally, a correlated random matrix, satisfying the conditions
of Theorem~\ref{thm:type} or Theorem~\ref{thm:univcorr}, respectively. Let $\varrho$ be the self-consistent density of states obtained from solving
the corresponding Dyson equation.
Then for any $\delta, \gamma>0$  and $D>0$ we have
\[
  \P \Big( \exists \bu,\lambda, \;  H\bu= \lambda\bu, \;\; \|\bu\|_2=1\;\; \varrho(\lambda) \ge \delta, \;\;  \| \bu\|_\infty \ge N^{-\frac{1}{2}+\gamma} \Big) \le 
  \frac{C}{N^D}.
\]
\end{theorem}
\begin{proof}[{\it Sketch of the proof.}] The proof was basically given in \eqref{delocbound}. The local laws guarantee that $\im G_{jj}(z) $  is close to
its deterministic approximant, $m_i(z)\delta_{ij}$ or $M_{ij}(z)$, these statements hold for any $E=\re z$ in the bulk and for $\eta\ge N^{-1+\gamma}$.
Moreover,  \eqref{bulk bound} and \eqref{Mupper} show that in the bulk regime both $|\bbm|$ and $\| M\|$ are bounded. From these two information
we conclude that $\im G_{jj}(z) $ is bounded with very high probability. 
\end{proof}

\subsubsection{Rigidity}

The next  standard consequence of the local law is the {\bf rigidity of eigenvalues}. It states that with very high probability
the eigenvalues in the bulk are at most $N^{-1+\gamma}$-distance away from their {\bf classical locations} predicted
by the corresponding quantiles of the self-consistent density of states, for any $\gamma>0$. This error bar $N^{-1+\gamma}$ reflects 
that typically the eigenvalues are almost as close to their deterministically prescribed locations as the typical 
level spacing $N^{-1}$. This is actually an indication of a very strong correlation;  e.g. if the eigenvalues were completely uncorrelated, i.e.
given by a Poisson point process with intensity $N$, then the typical fluctuation of the location of the points would be $N^{-1/2}$.

Since local laws at spectral parameter $z=E+i\eta$ determine the local eigenvalue density
on scale $\eta$, it is very natural that a local law on scale $\eta=\im z$ locates individual eigenvalues with $\eta$-precision.
Near the edges and cusps the local spacing is different ($N^{-2/3}$ and $N^{-3/4}$, respectively), and
the corresponding rigidity result must respect this. For simplicity, here we state only the bulk result, as we did for the
local law as well; for results at the edge and cusp, see \cite{AEK2}.

Given the self-consistent density $\varrho$, for any energy $E$, define
\be\label{iedef}
  i(E): = \Big\lceil  N \int_{-\infty}^E \varrho(\om)\rd\om \Big\rceil
\ee
to be the index of the $N$-quantile closest to $E$. Alternatively, for any $i\in \llbracket 1, N \rrbracket$ one could 
define $\gamma_i =\gamma_i^{(N)}$ to be the $i$-th  $N$-quantile of $\varrho$ by the relation
\[
   \int_{-\infty}^{\gamma_i} \varrho(\om)\rd\om = \frac{i}{N},
\]
then clearly $\gamma_{i(E)}$ is (one of) the closest $N$-quantile  to $E$ as long as $E$ is in the bulk, $\varrho(E)>0$.

\begin{theorem}[Rigidity of bulk eigenvalues] Let $H$ be a Wigner type or, more generally, a correlated random matrix, satisfying the conditions
of Theorem~\ref{thm:type} or Theorem~\ref{thm:univcorr}, respectively. Let $\varrho$ be the self-consistent density of states obtained from solving
the corresponding Dyson equation.  Fix any $\delta,\e,  D>0$. For any energy $E$ in the bulk, $\varrho(E)\ge \delta$, we have
\be\label{rigged}
  \P \Big( |\lambda_{i(E)} - E|\ge \frac{N^\e}{N} \Big) \le  \frac{C}{N^D}.
\ee
\end{theorem}

\begin{proof}[\it Sketch of the proof.]
The proof of rigidity from the local law is a fairly standard procedure by now, see Chapter 11 of \cite{ErdYau2017},  or  Lemma 5.1 \cite{AEK2} 
especially tailored to our situation. The key step is the following {\it Helffer-Sj\"ostrand formula} that expresses integrals of a compactly supported
function $f$ on the real line against
a (signed) measure $\nu$ with bounded variation in terms of the Stieltjes transform of $\nu$. (Strictly speaking we defined
Stieltjes transform only for probability measures, but the concept can be easily extended  since  any signed measure with
bounded variation can be written as a difference of two non-negative measures, and thus Stieltjes transform extends by linearity).

 Let $\chi$ be a compactly supported smooth cutoff function on $\R$ such that $\chi\equiv 1$ on $[-1,1]$. Then the
Cauchy integral formula implies
\be\label{cauchy}
  f(\tau) = \frac{1}{2\pi} \int_{\R^2} \frac{ i\eta f''(\sigma) \chi(\eta) + i(f(\sigma) + i\eta f'(\sigma))\chi'(\eta)}{\tau-\sigma -i\eta} \rd \sigma\,\rd\eta.
\ee
Thus for any real valued smooth $f$ the {\bf Helffer-Sj\"ostrand formula} states that
\be\label{HS}
   \int_\R f(\tau)\nu(\rd\tau) = -\frac{1}{2\pi} \big( L_1 + L_2 +L_3\big)
\ee
with
\   \begin{align*}
	   L_1 &=\int_{\R^2}\eta f''(\sigma)\chi(\eta) \im m (\sigma+i\eta) \rd\sigma\,\, \rd \eta\\
	   L_2 &=\int_{\R^2} f'(\sigma)\chi'(\eta) \im m (\sigma+i\eta) \rd\sigma\, \rd \eta\\
	   L_3 &=\int_{\R^2} \eta f'(\sigma)\chi'(\eta) \re m (\sigma+i\eta) \rd\sigma\, \rd \eta
   \end{align*}
where $m(z)=m_\nu(z)$ is the Stieltjes transform of $\nu$. Although this formula is a simple identity,
it plays an essential role in various problems of spectral analysis. One may apply it
to develop functional calculation (functions of a given self-adjoint operator) in terms of
the its resolvents \cite{Dav1995}.

For the proof of the eigenvalue rigidity, the formula \eqref{HS} is used for $\nu: = \mu_N -\varrho$, i.e.
for the difference of the empirical and the self-consistent density of states.  Since the normalized trace of the
resolvent is the Stieltjes transform of the empirical density of states,  the averaged local
law \eqref{perch16} (or \eqref{perch13cor} with $T=1$) states that 
\be\label{lll}
   |m_\nu(\tau+i\eta)| \leq \frac{N^\e}{N\eta}   \qquad \eta \ge N^{-1+\gamma}
 \ee
with very high probability for any $\tau$ with $\varrho(\tau)\ge \delta$.
Now we fix two energies, $\tau_1$ and $\tau_2$ in the bulk and  define $f$ to be
the characteristic function of the interval $[\tau_1,\tau_2]$ smoothed out
on some scale $\eta_0$ at the edges, i.e. 
\[
   f|_{[\tau_1,\tau_2]}=1,\qquad  f|_{\R\setminus [\tau_1-\eta_0,\tau_2+\eta_0]}=0
\]
with derivative bounds $|f'|\le C/\eta_0$, $|f''|\le C/\eta_0^2$ in the transition regimes 
\[J:=[\tau_1-\eta_0, \tau_1]\cup [\tau_2,\tau_2+\eta_0].
\]  We will choose $\eta_0=N^{-1+\gamma}$.  Then it is easy to see
that $L_2$ and $L_3$ are bounded by $N^{-1+\eps + C\gamma}$ since
$\chi'(\eta)$ is supported far away from 0, say on $[1,2]\cup [-2,-1]$, hence, for example
\[ 
    |L_2| \lesssim \int_1^2  \rd\eta \int_J \rd\sigma  \; \frac{1}{\eta_0} \; |\chi'(\eta)|  \frac{N^\e}{N\eta} \le N^{-1+\eps+2\gamma}
\]
using that $|J|\le 2\eta_0\lesssim N^{-1+\gamma}$.
A similar direct estimate does not work for $L_1$ since it would give 
\be\label{l1}
    |L_1| \lesssim \int_0^\infty \rd\eta \int_J \rd\sigma  \; \eta \; \frac{1}{\eta_0^2} \; \chi(\eta)  \frac{N^\e}{N\eta} \le N^{\eps+3\gamma}.
\ee
Even this estimate would need a bit more care since the local law \eqref{lll} does not hold for $\eta$ smaller than $N^{-1+\gamma}$,
but here one uses the fact that for any positive measure $\mu$, the (positive) function $\eta \to \eta \im m_\mu(\sigma+i\eta)$
is monotonously increasing, so the imaginary part of the Stieltjes transform at smaller $\eta$-values can be controlled
by those at larger $\eta$ values. Here it is crucial that $L_1$ contains only the imaginary part of the Stieltjes transforms 
and not the entire Stieltjes transform. The argument \eqref{l1}, while does not cover the entire $L_1$, it gives a sufficient bound
on the small $\eta$ regime:
\begin{align*}
\int_0^{\eta_0} \rd\eta & \int_J \rd\sigma\, \eta f''(\sigma)\chi(\eta) \im m (\sigma+i\eta) \\
\le  & \int_0^{\eta_0} \rd\eta \int_J \rd \sigma\   |f''(\sigma)| \eta_0 \im m (\sigma+i\eta_0) \\
 \le {}& N^{-1+\eps+3\gamma}.
\end{align*}

To improve \eqref{l1} by a factor $1/N$ for $\eta\ge \eta_0$,
 we  integrate by parts before estimating. First we put one $\sigma$-derivative from $f''$ to $m_\nu(\sigma+i\eta)$,
then the $\partial_\sigma$ derivate is switched to $\partial_\eta$ derivative, then  another integration by 
parts, this time in $\eta$ removes the derivative from $m_\nu$. The boundary terms, we obtain 
formulas similar to $L_2$ and $L_3$ that have already been estimated.

The outcome is that 
\be\label{outcome}
   \int_\R f(\tau)\big[ \mu_N(\rd\tau) -\varrho(\tau)\rd \tau] \le N^{-1+ \e'}
\ee
for any $\e'>0$  with very high probability,
since $\e$ and $\gamma$ can be chosen arbitrarily small positive numbers
in the above argument. If $f$ were exactly the characteristic function,  then  \eqref{outcome} would imply that
\be\label{counting}
  \frac{1}{N} \#\big\{ \lambda_j \in [\tau_1, \tau_2]\big\} = \int_{\tau_1}^{\tau_2} \varrho(\om)\rd\om + O( N^{-1+ \e'})
\ee
i.e. it would identify  the eigenvalue counting function down to the optimal scale. 
Estimating the effects of the smooth  cutoffs is an easy technicality. Finally,  \eqref{counting}
can be easily turned into \eqref{rigged}, up to one more catch. So far we assumed that $\tau_1, \tau_2$
are  both in the bulk since the local law was formulated in the bulk
 and \eqref{counting} gave the number of eigenvalues in any interval with
endpoints in the bulk. 

The quantiles appearing in \eqref{rigged}, however, involve semi-infinite intervals,
so one also needs a local law well outside of the bulk. Although in Theorems \ref{thm:type} and \ref{thm:univcorr} we formulated  local laws in
the bulk, similar, and typically even easier estimates are available for energies far away from the support of $\varrho$.
In fact,  in the regime where $\mbox{dist} (\tau, \mbox{supp}\varrho) \ge \delta$ for some fixed $\delta>0$, 
the analogue \eqref{lll} is improved to 
\be\label{lll1}
   |m_\nu(\tau+i\eta)| \leq \frac{N^\e}{N}   \qquad \eta > 0,
 \ee
makingo the  estimates on $L_j$'s  even easier when $\tau_1$ or $\tau_2$ is far  from the bulk. 
\end{proof}

\subsubsection{Universality of local eigenvalue statistics}

The universality of the local distribution of the eigenvalues is the main coveted goal of random matrix theory.
While local laws and rigidity are statements where random quantities are compared with deterministic ones, i.e. they
are, in essence, law of large number type results (even if not always formulated in that way), the
universality is about the emergence and ubiquity of a new distribution.

We will formulate universality in two forms: on the level of  correlation functions and on the level
of individual gaps. While these formulations are ``morally'' equivalent, technically they require quite
different proofs.

We need to strengthen a bit the assumption on the lower bound on the variances in \eqref{sbound} 
for complex hermitian Wigner type matrices $H$. In this case we define the real symmetric $2\times 2$ matrix
\[
  \sigma_{ij} := \begin{pmatrix} \E (\re h_{ij})^2 & \E (\re h_{ij})(\im h_{ij}) \\\E (\re h_{ij})(\im h_{ij}) &
  \E (\im h_{ij})^2
  \end{pmatrix}
\]
for every $i,j$ and we will demand that
\be\label{flange}
     \sigma_{ij}\ge \frac{c}{N}
\ee
with some $c>0$ uniformly for all $i,j$ in the sense of  quadratic forms on $\R^2$.
Similarly,  for correlated matrices the  flatness condition \eqref{flatness} is strengthened to the
requirement that there is a constant $c>0$ such that
\be\label{fullness}
  \E  |\tr BH|^2 \ge c \tr B^2
\ee
for any real symmetric (or complex hermitian, depending on the symmetry class of $H$) deterministic matrix $B$.

\begin{theorem}[Bulk universality]\label{thm:bu}
 Let $H$ be a Wigner type or, more generally, a correlated random matrix, satisfying the conditions
of Theorem~\ref{thm:type} or Theorem~\ref{thm:univcorr}, respectively. ~For Wigner type matrices in the complex hermitian symmetry class we additionally assume \eqref{flange}.
For correlated random matrices, we additionally assume~\eqref{fullness}.

Let $\varrho$ be the self-consistent density of states obtained from solving
the corresponding Dyson equation.  Let $k\in \N$, $\delta>0,E\in \R$ with $\varrho(E)\ge \delta$ and let $\Phi\colon \R^k\to\R$ be a compactly supported smooth test function. Then for some positive constants $c$ and $C$, depending on   $\Phi,\delta,k$, we have the following:

(i) [Universality of correlation functions]
 Denote the $k$-point correlation function of the eigenvalues of $H$ by $p^{(k)}_N$ (see \eqref{pk})
  and denote the corresponding $k$-point correlation function of the 
GOE/GUE-point process by $\Upsilon^{(k)}$. Then 
\be\label{univcorr}
\Bigg|\int_{\R^k} \Phi(\bt) \left[ \frac{1}{\rho(E)^k}p^{(k)}_N\Big(E+\frac{\bt}{N\rho(E)}\Big)-\Upsilon_k(\bt) \right]\rd\bt\Bigg|
\le C N^{-c}.
\ee
(ii) [Universality of gap distributions]   Recall that $i(E)$ is the index of the   $N$-th quantile
in the density $\varrho$  that is closest to the energy $E$  \eqref{iedef}. Then
\begin{equation}\label{univgap}
\begin{split}
	\Bigg| \E \Phi\Big( \big(N & \rho(\lambda_{k(E)})[\lambda_{k(E)+j}-\lambda_{k(E)}]\big)_{j=1}^k \Big) \\
	&- \E_{\text{GOE/GUE}}\Phi\Big( \big(N \rho_{sc}(0)[\lambda_{\lceil N/2\rceil+j}-\lambda_{\lceil N/2\rceil}]\big)_{j=1}^k \Big) \Bigg|\le C N^{-c},
\end{split}
\end{equation}
where the expectation $\E_{\text{GOE/GUE}}$ is taken with respect to the Gaussian matrix ensemble in the same symmetry class as $H$.
\end{theorem}

\begin{proof}[{\it Short sketch of the proof.}]
The main method to prove universality is the {\it three-step strategy} outlined in Section \ref{sec:3step}.
The first step is to obtain a local law which serves as an {\it a priori} input for the
other two steps and it is the only model dependent step.  The second step is to show that a small Gaussian component in the distribution  
already produces the desired universality. The third step is a perturbative argument to show that removal of the Gaussian
component does not change the local statistics. 
There  have been many theorems of increasing generality to complete the second and third  steps and by now
very general ``black-box'' theorems exist that are model-independent.

The {\bf second step} relies on the local equilibration properties of the Dyson Brownian motion
introduced in  \eqref{dbmm}.
The  latest and most general formulation
of this idea  concerns universality of deformed Wigner matrices of the form
\[
     H_t = V + \sqrt{t} W,
\]
where $V$ is a deterministic matrix and $W$ is a GOE/GUE matrix.   In applications $V$ itself is a random matrix 
and in $H_t$ an additional independent Gaussian component is added.
 But for the purpose of local equilibration of the DBM,
hence for the emergence of the universal local statistics, only the randomness of $W$ is used, hence one may
condition on $V$. The main input of the following result  is that the local eigenvalue density of $V$ must be controlled in a
sense of lower and upper bounds on the imaginary part  of the Stieltjes transform $m_V$ of
the empirical eigenvalue density of $V$. In practice this is obtained from the local law with very high probability in 
the probability space of $V$.

\begin{theorem}[\cite{LanYau2015, LSY2016}]\label{thm:u}
Choose two $N$-dependent parameters,
$L, \ell$ for which we have   $1\gg L^2 \gg \ell \gg N^{-1}$ (here the notation $\gg$ indicates separation by an $N^\e$ factor
for an arbitrarily small $\e>0$).
Suppose that around a fixed energy $E_0$
in a window of size $L$  the local eigenvalue density of $V$ on scale  $\ell$ 
is controlled, i.e.
\[
  c\le \im m_V(E+i\eta) \le C, \qquad E\in (E_0-L, E_0+L), \qquad \eta\in [\ell, 10]
\]
(in particular, $E_0$ is in the bulk of $V$). Assume also that $\| V\|\le N^C$.
Then for any $t$ with $N^\e \ell \le t\le N^{-\e}L^2$ the bulk universality of $H_t$ around $E_0$ holds both in the
sense of correlation functions  at fixed energy \eqref{univcorr} and in sense of gaps \eqref{univgap}.
\end{theorem}
Theorem~\ref{thm:u}  in this general form appeared in  \cite{LanYau2015}  (gap universality) and in \cite{LSY2016} (correlation functions
universality at fixed energy).  These ideas have been developed in several papers.
Earlier  results concerned  Wigner or generalized
Wigner matrices and proved correlation function universality with a small energy averaging 
\cite{ErdSchYau2011, ErdSchYauYin2012}, fixed energy universality  \cite{BouErdYauYin2015} and gap universality \cite{ErdYau2015}.
Averaged energy and gap universality for random matrices with general density profile were also proven  in  \cite{ErdSch2015} assuming 
more precise information on $m_V$ that are  available from the optimal local laws.

Finally, the {\bf third step} is to remove the small  Gaussian component by realizing that the family of matrices of the form $H_t= V+\sqrt{t}W$
to which Theorem~\ref{thm:u} applies
is sufficiently rich so that for any given random matrix $H$ there exists a matrix $V$ and a small $t$ so that the local
statistics of $H$ and $H_t= V+\sqrt{t}W$ coincide. We will use this result for some $t$ with $t= N^{-1+\gamma}$
with a small $\gamma$. 
The time $t$ has to be much larger than $\ell$ and $\ell$ has to be much  larger than $ N^{-1}$ since below that
scale the local density of $V$ (given by $ \im m_V(E+i\eta)$) is not bounded. But $t$ cannot be too large either otherwise 
the comparison result cannot hold.

Note that the local statistics is not compared directly with that of $V$;
this would not work even for Wigner matrices  $V$ and even
if we used the Ornstein Uhlenbeck process, i.e. $H_t = e^{-t/2} V + \sqrt{1-e^{-t}}W$
(for  Wigner matrices  $V$ the OU process has the advantage that it preserves  not only the first 
but also the second moments of $H_t$). But for any given Wigner-type ensemble $H$ one can find a random $V$ and 
an independent Gaussian $W$ so that the first three moments of $H$ and $H_t = e^{-t/2} V + \sqrt{1-e^{-t}}W$ coincide
and the fourth moments  are very close;  this freedom is guaranteed by the  lower bound on $s_{ij}$ and  $\sigma_{ij}$ \eqref{flange}.

The main perturbative result is the following {\it Green function comparison theorem} that allows us to compare 
{\it expectations} of reasonable functions of the Green functions of two different ensembles
whose first four moments (almost) match (the idea of matching four moments in random matrices was introduced in \cite{TaoVu2011}).
 The key point is that $\eta=\im z$ can be slightly below the critical
threshold $1/N$: the expectation regularizes the possible singularity. 
 Here is the prototype of such a theorem:

\begin{theorem}[Green function comparison]\cite{EYY}\label{eye}
 Consider two Wigner type ensembles $H$ and $\wt H$ such that their first two moments are the
same, i.e. the matrices of variances coincide, $S=\wt S$ and the third and fourth  moments almost match in a sense that
\be\label{match}
   \big| E h_{ij}^s -\E\wh h_{ij}^s\big| \le N^{-2-\delta}, \qquad s=3,4
 \ee
 (for the complex hermitian case all mixed moments  of  order 3 and 4 should match).  Define
 a sequence of interpolating  Wigner-type matrices  $H_0, H_1, H_2, \ldots $ such that $H_0=H$,
 then in $H_1$ the $h_{11}$ matrix element is replaced with $\wt h_{11}$, in $H_2$
 the $h_{11}$ and $h_{12}$ elements are replaced with $\wt h_{11}$ and $\wt h_{12}$, etc., i.e.
 we  replace one by one  the distribution of the matrix elements. Suppose that the Stieltjes transform 
 on scale $\eta = N^{-1+\gamma}$ is bounded   for all
 these interpolating matrices and for any $\gamma>0$. Set now $\eta' : = N^{-1-\gamma}$ and let
 $\Phi$ a smooth function with moderate growth. Then
 \be\label{comparison}  \Big| \E \Phi\big( G(E+i\eta')\big) - \wt\E \Phi\big( G(E+i\eta')\big)\Big| \le N^{-\delta+C\gamma}
 \ee
 and similar multivariable versions also hold.
 \end{theorem}
  
In the applications, choosing $\gamma$ sufficiently small, we could conclude that the distribution of the Green functions
of $H$ and $\wt H$ on scale even
{\bf below the eigenvalue spacing} are close. On this scale local correlation functions can be identified,
so we conclude that the local eigenvalue statistics of $H$ and $\wt H$ are the same. This will conclude
step 3 of the three step strategy and finish the proof of bulk universality, Theorem~\ref{thm:bu}.
\end{proof}


\begin{proof}[{\it Idea of the proof of Theorem~\ref{eye}.}]
The proof of \eqref{comparison} is a ``brute force'' resolvent and Taylor expansion. For simplicity, we first replace  $\Phi$
by its finite Taylor polynomial. Moreover, we consider only the linear term for illustration in this proof. 
 We estimate the change
of $\E  G(E+i\eta')$ after each replacement; we need to bound each of them by $o(N^{-2})$ 
since there are of order $N^2$ replacements.  Fix an index pair $i, j$. 
Suppose  we are at the step when we change the $(ij)$-th matrix element $h_{ij}$ 
to $\wt h_{ij}$. Let $R$ denote the resolvent of the matrix with $(ij)$-th  and $(ji)$-th
elements being zero, in particular $R$ is independent of $h_{ij}$.
 It is
easy to see from the local law that $\max_{ab}|R_{ab}(E+i\eta)|\lesssim 1$ for any $\eta\ge N^{-1+\gamma}$
and therefore, by the monotonicity of $\eta\to \eta \im m(E+i\eta)$ we find that $|R_{ab}(E+i\eta')\lesssim N^{2\gamma}$.
Then simple resolvent expansion gives, schematically, that
\be\label{Gexp}
   G = R + Rh_{ij}R + Rh_{ij}Rh_{ij} R +Rh_{ij}Rh_{ij} S h_{ij}R + Rh_{ij}Rh_{ij} R h_{ij}R h_{ij}R+\ldots
 \ee
 and a similar expansion for $\wt G =G(h_{ij}\leftrightarrow \wt h_{ij})$ where all $h_{ij}$ is replaced with $\wt h_{ij}$
 (strictly speaking we need to replace $h_{ij}$ and $h_{ji}=\bar h_{ij}$ simultaneously due to hermitian symmetry, but we neglect this).
We do the expansion up to the fourth order terms (counting the number of $h$'s). The naive size of a third order 
term, say, $Rh_{ij}Rh_{ij} Rh_{ij}R$
 is of order $N^{-3/2+8\gamma}$ since every $h_{ij}$ is of order $N^{-1/2}$. However, the difference in $\E$ and $\wt\E$-expectations
 of these terms are of order $N^{-2-\delta}$ by \eqref{match}.  Thus for the first four terms  (fully expanded ones) in \eqref{Gexp}
 it holds that 
 \[
    \E G - \wt \E \wt G = O(N^{-2-\delta + C\gamma}) + \mbox{fifth and higher order terms}
 \]
 But all fifth and higher order terms have at least five $h$ factors  so their size is  essentially $N^{-5/2}$, i.e. negligible,
 even without any cancellation between $G$ and $\wt G$. Finally, we need to repeat this one by one replacement $N^2$ times,
 so we arrive at a bound of order $N^{-\delta+C\gamma}$. This proves \eqref{comparison}.
 \end{proof}

\begin{exercise}
For a given real symmetric matrix $V$ let  $H_t$ solve the SDE
\[
  \rd H_t = \frac{\rd B_t}{\sqrt{N}}, \qquad H_{t=0}=V
\]
where $B_t=B(t)$ is a standard real symmetric matrix valued  Brownian motion, i.e. the matrix elements
$b_{ij}(t)$ for $i<j$ as well as $b_{ii}(t)/\sqrt{2}$ are independent standard Brownian motions and $b_{ij}(t)= b_{ji}(t)$.
Prove that the eigenvalues of $H_t$ satisfy the following coupled system of stochastic differential equations
(Dyson Brownian motion):
\[
  \rd\lambda_a=\sqrt{ \frac{2}{N}} \rd B_a +\frac{1}{N}\sum_{b\ne a}\frac{1}{\lambda_a-\lambda_b}\rd t,
  \qquad a\in \llbracket 1, N\rrbracket
\]
where $\{ B_a\; : \; a\in \llbracket 1, N \rrbracket\} $ is a collection of independent standard Brownian motions
with initial condition $\bla_a(t=0)$ given by the eigenvalues of $V$. {\it Hint:} Use first and second order perturbation
theory to differentiate the eigenvalue equation $H\bu_a= \lambda_a \bu_a$  with the side condition $\langle \bu_a,  \bu_b\rangle =\delta_{ab}$,
then use Ito formula (see Section 12.2 of \cite{ErdYau2017}). Ignore the complication that Ito formula cannot be
directly used due to the singularity; for a fully rigorous proof, see Section 4.3.1 of \cite{AndGuiZei2010}.

\end{exercise}

\section{Analysis of the vector Dyson equation}\label{sec:vector}

In this section we outline the proof of a few results concerning the vector Dyson equation \eqref{vectordyson2}
\be\label{vectordyson3}
  -\frac{1}{\bbm} = z+ S\bbm, \qquad z\in \Cp, \quad \bbm\in \Cp^N,
\ee
where $S=S^t$ is symmetric, bounded, $\| S\|_\infty\le C$  and has nonnegative entries.

We recall the convention that $1/\bbm$ denotes a vector in $\C^N$ with components $1/m_j$.
Similarly, the relation $\bu\le \bv$ and the product $\bu\bv$ of two vectors are understood in coordinate-wise sense.

\subsection{Existence and uniqueness}

We sketch the existence and uniqueness result, i.e. Theorem~\ref{exuniqvector}, a detailed proof can
be found in Chapter 4 \cite{AjaErdKru2015}. To orient the reader here we only mention that it is a fix-point argument for the map
\be\label{Phidef}
   \Phi(\bu): = -\frac{1}{z+ S\bu}
\ee
that maps $\Cp^N$ to $\Cp^N$ for any fixed $z\in \Cp$. Denoting by
\[
   D(\zeta,\om): = \frac{|\zeta-\om|^2}{ (\im \zeta)(\im\omega)}, \qquad \zeta, \om\in\Cp
\]
the standard hyperbolic metric on the upper half plane, one may check that $\Phi$ is a contraction in this metric. More precisely,
for any fixed constant $\eta_0$, we have the bound
\be\label{Dphi}
      \max_j D\Big( \Phi(\bu)_j , \Phi(\bw)_j\Big) \le \Big(1 + \frac{\eta_0^2}{\| S\|}\Big)^{-2} \max_j D(u_j, w_j)
\ee
assuming that $\im z\ge \eta_0$ and both $\bu$ and $\bw$ lie in a large compact set
\be\label{Beta}
    B_{\eta_0}: = \Big\{ \bu\in \Cp^N\; : \;   \| \bu\|_\infty \le \frac{1}{\eta_0},\quad \inf_j  \im u_j \ge \frac{\eta_0^2}{(2+\|S\|)^2}\Big\},
\ee
that is mapped by $\Phi$ into itself. Here $\|\bu\|= \max_j |u_j|$. Once setting up the contraction properly, the rest is a straightforward fixed point theorem.
The representation \eqref{reps} follows from  the Nevanlinna's theorem as mentioned after Definition~\ref{def:sit}.

Given  \eqref{reps}, we recall that $\varrho = \langle \bnu\rangle= \frac{1}{N}\sum_j \nu_j$ is the self-consistent density of states. We consider its harmonic 
extension to the upper half plane and continue to denote it by  $\varrho$:
\be\label{vr}
  \varrho= \varrho(z) : = \frac{\eta}{\pi}\int_\R \frac{\varrho(\rd \tau)}{|x-E|^2 +\eta} = \frac{1}{\pi} \langle \im \bbm(z) \rangle, \qquad z=E+i\eta.
 \ee

 \begin{exercise}
Check directly from \eqref{vectordyson3} that the solution satisfies the additional condition of the Nevanlinna's theorem,
i.e. that for every $j$ we have  $i\eta m_j(i\eta) \to -1$ as $\eta\to \infty$. Moreover, check    that $|m_j(z)|\le 1/\im z$.
\end{exercise}

\begin{exercise} Prove that the support of all measures $\nu_i$ lie in $[-2\sqrt{\| S\|_\infty}, 2\sqrt{\| S\|_\infty}]$.  \\
{\it Hint: }suppose $|z|> 2\sqrt{\| S\|_\infty}$, then 
check  the following implication:
\[
   \mbox{If} \quad \| \bbm(z)\|_\infty < \frac{|z|}{2\| S\|_\infty}, \quad\mbox{then} \quad \| \bbm(z)\|_\infty < \frac{2}{|z|}
\]
and apply a continuity argument to conclude that $\| \bbm(z)\|_\infty < \frac{2}{|z|}$ holds unconditionally. Taking the imaginary part of
\eqref{vectordyson3} conclude that $\im \bbm(E+i\eta)\to 0$ as $\eta\to 0$ for any  $|E| > 2\sqrt{\| S\|_\infty}$.
\end{exercise}

\begin{exercise}
Prove the inequality \eqref{Dphi}, i.e. that $\Phi$ is indeed a contraction on $B_{\eta_0}$. {\it Hint:}  Prove and then use the following
properties of the metric $D$:
\begin{itemize}
\item[1)] The metric $D$ is invariant under linear fractional transformations of $\Cp$ of the form 
\[
  f(z)= \frac{az+b}{cz+d}, \qquad z\in \Cp, \qquad \begin{pmatrix} a & b\\ c & d \end{pmatrix} \in SL_2(\R).
 \]
\item[2)] Contraction: for any $z, w\in \Cp$ and $\lambda>0$ we have
\[
    D(z+i\lambda, w+i\lambda) = \Big(1+\frac{\lambda}{\im z}\Big)^{-1}\Big(1+\frac{\lambda}{\im w}\Big)^{-1} D(z, w);
\]
\item[3)] Convexity: Let $\ba=(a_1, \ldots, a_N)\in\R_+^N$, then
\[
   D\big(\sum_i a_iu_i ,  \sum_i a_i w_i\big)\le  \ \max_i D(u_i, w_i), \qquad \bu,\bw\in \Cp^N.
\]

\end{itemize}
\end{exercise}

\subsection{Bounds on the solution}

Now we start  the quantitative analysis of the solution and we start with a result on the boundedness  in the bulk.
We introduce the maximum norm and the $\ell^p$ norms  on $\C^N$ as follows:
\[
\|\bu\|_\infty = \max_j |u_j|, \qquad \|\bu\|_p^p : =\frac{1}{N}\sum_j |u_j|^2 = \langle |\bu|^p\rangle.
\]

The  procedure to bound $\bbm$ is that we first obtain an $\ell^2$-bound which usually requires less 
conditions. Then we enhance it to an $\ell^\infty$ bound. First we obtain a bound that is useful in the bulk
but deteriorates as the self-consistent density vanishes, e.g. at the edges and cusps. Second, we 
improve this bound to one that is also useful near the edges/cusps but 
 this requires some additional regularity condition on $s_{ij}$. In these notes
 we will not aim at the most optimal conditions, see \cite{AEK1short} and \cite{AjaErdKru2015}
 for the detailed analysis.

\subsubsection{Bounds useful in the bulk}\label{sec:useful}

\begin{theorem}\label{thmbulk}[Bounds on the solution]
Given lower and upper bounds of the form
\be\label{sbound1}
  \frac{c}{N}\le s_{ij} \le \frac{C}{N},
\ee
as in \eqref{sbound}), we have
\[
  \| \bbm\|_2\lesssim 1, \qquad   |\bbm(z)|\lesssim\frac{1}{\varrho(z)+\mbox{dist}(z, \mbox{supp}\varrho)}, \qquad  \frac{1}{|\bbm(z)|} \lesssim 1+|z|, 
\]
and 
\[    \varrho(z) \lesssim \im \bbm(z) \lesssim (1+|z|)^2 \| \bbm(z)\|^2_\infty\varrho(z),
\]
where we recall that $\lesssim$ indicates a bound up to  an unspecified multiplicative constant 
that is independent of $N$ (also, recall that the last three inequalities are understood in coordinate-wise sense). 
\end{theorem}

\begin{proof} For simplicity, in the proof we assume that $|z|\lesssim 1$; the
large $z$ regime is much easier and follows directly from the Stieltjes transform representation of $\bbm$.
Taking the imaginary part of the Dyson equation \eqref{vectordyson3}, we have
\be\label{impart}
  \frac{\im \bbm}{|\bbm|^2} = \eta + S\im \bbm.
\ee
Using the lower bound from \eqref{sbound1}, we get 
\[
     S\im \bbm \ge c \langle \im \bbm \rangle \ge c\varrho
\]     
thus
\be\label{346}
    \im \bbm \ge c|\bbm|^2 \varrho.
\ee
Taking the average of both sides and dividing by $\varrho>0$, we get $\|\bbm\|_2\lesssim 1$.
Using $\im \bbm\le |\bbm|$, we immediately get an upper bound on $|\bbm|\lesssim 1/\varrho$.
The alternative bound 
\[
  |\bbm(z)|\lesssim \frac{1}{\mbox{dist}(z, \mbox{supp}\varrho)}
\]
follows from the Stieltjes transform representation \eqref{reps}.

Next, we estimate the rhs. of \eqref{vectordyson3} trivially, we have
\[
   \frac{1}{|m_i|} \le |z| + \sum_j s_{ij} |m_j|\lesssim |z| + \| \bbm\|_1 \le |z|+\| \bbm\|_2 \lesssim 1
\]
using  H\"older inequality in the last but one step. This gives the upper bound on $1/|\bbm|$.

Using this bound, we can conclude from \eqref{346} that $\varrho\lesssim \im \bbm$. The upper bound on $\im \bbm$
also follows from \eqref{impart} and \eqref{sbound1}:
\[
   \frac{\im \bbm}{|\bbm|^2} \le \eta  + S\im \bbm \le \eta + \langle \im \bbm \rangle \lesssim \eta + C\varrho.
\]
 Using that
 \[
   \varrho(z) \gtrsim \frac{\eta}{(1+ |z|)^2}
 \]
 which can be easily checked from \eqref{vr} and the boundedness of the support of $\varrho$, 
 we conclude the two-sided bounds on $\im \bbm$. 
 \end{proof}
 
Notice two weak points when using  this relatively simple argument. First, the lower bound in \eqref{sbound1} was heavily used,
although much less assumption is sufficient. We will not discuss  these generalizations in these notes, but see 
Theorem 2.11 of \cite{AjaErdKru2015} and remarks thereafter addressing this issue.
 Second, the upper bound on $|\bbm|$ for small $\eta$ is useful only inside the self-consistent bulk
spectrum or away from the support of $\varrho$,  it deteriorates near the edges of the spectrum.
In the next sections we remedy this situation. 

\subsubsection{Unconditional $\ell^2$-bound away from zero}\label{sec:uncond}

Next, we present a somewhat surprising result that shows that an $\ell^2$-bound on the solution,  $\| \bbm(z)\|_2$,
away from the only critical point $z=0$ is possible {\it without any condition on $S$}. The spectral parameter $z=0$ is clearly
critical, e.g. if $S=0$, the  solution $\bbm(z) = -1/z$  blows up.  Thus to control the behavior of $\bbm$ around $z\approx 0$ 
one needs some non degeneracy condition on $S$. We will not address  the issue of $z\approx 0$  in these
notes, but we remark that a fairly complete picture was obtained in Chapter 6 of \cite{AjaErdKru2015}
using the concept of {\it fully indecomposability}.

Before presenting the $\ell^2$-bound away from zero, 
we introduce 
an important  object, the {\bf saturated self-energy operator}, that will also play a key role later in the stability analysis:

\begin{definition}\label{def:F}
Let $S$ be a symmetric matrix with nonnegative entries and let $\bbm=\bbm(z)\in\Cp^N$ solve the vector Dyson equation  \eqref{vectordyson3}
for some fixed spectral parameter $z\in \Cp$. The matrix $F=(F_{ij})$ with 
\[
   F_{ij} := |m_i| s_{ij} |m_j|
\]
 acting as
\be\label{Fdef}
      F\bu  = |\bbm| S\big(|\bbm| \bu), \quad \mbox{i.e.} \quad  \big( F\bu\big)_i = |m_i| \sum_j s_{ij} |m_j| u_j
\ee
on any vector $\bu\in\Cp^N$, is called the {\bf saturated self-energy operator}. 
\end{definition}

Suppose that $S$ has strictly positive entries. Since $m_i\ne0$ from \eqref{vectordyson3},  clearly
$F$ has also positive entries, and 
 $F=F^*$. Thus the Perron-Frobenius theorem applies to $F$, and it guarantees that $F$ has a single largest eigenvalue $r$
 (so that for any other eigenvalue $\lambda$ we have $|\lambda|<r$) and the corresponding eigenvector $\bbf$ 
 has positive entries: $  F\bbf = r \bbf, \qquad \bbf> 0$.
 Moreover, since $F$ is symmetric, we have $\| F\|_2=r$ for the usual Euclidean matrix norm of $F$.

\begin{proposition}
Suppose that $S$ has strictly positive entries and let $\bbm$ solve \eqref{vectordyson3} for some $z=E+i\eta\in \Cp$. Then the norm of the saturated
self-energy operator is given by
\be\label{Fnorm}
    \|F\|_2= 1- \eta\frac{\langle \bbf |\bbm| \rangle}{\big\langle \bbf \frac{\im \bbm}{|\bbm|}\big\rangle},
\ee
in particular $\|F\|_2< 1$. Moreover,
\be\label{trivl2}
\| \bbm(z)\|_2\le \frac{2}{|z|}.
\ee
\end{proposition}
We remark that for the bounds $\| F\|_2< 1$ and \eqref{trivl2} it   is sufficient if $S$ has nonnegative entries 
instead of positive entries; the proof requires a bit more care, see Lemma 4.5 \cite{AjaErdKru2015}.


\begin{proof} Taking the imaginary part of \eqref{vectordyson3} and multiplying it by $|\bbm|$, 
we have
\be\label{perron}
   \frac{\im \bbm}{|\bbm|} = \eta |\bbm| + |\bbm| S\Big( |\bbm| \frac{\im \bbm}{|\bbm|}\Big) = \eta |\bbm| + F\, \frac{\im \bbm}{|\bbm|}.
 \ee
 Scalar multiply this equation by $\bbf$, use the symmetry of $F$ and $F\bbf = \|F\|_2\bbf$ to get
 \[
   \big\langle \bbf, \frac{\im \bbm}{|\bbm|}\big\rangle = \eta \langle \bbf |\bbm|\rangle +  \big\langle \bbf, F \frac{\im \bbm}{|\bbm|}\big\rangle
   = \eta \langle \bbf |\bbm|\rangle + \| F\|_2   \big\langle \bbf, \frac{\im \bbm}{|\bbm|}\big\rangle,
 \]
 which is equivalent to \eqref{Fnorm} (note that $\langle, \rangle$  as a binary operation is the scalar product while $\langle\cdot  \rangle$
 is the averaging).

 For the bound on $\bbm$, we write \eqref{vectordyson3} as $-z\bbm = 1+ \bbm S\bbm$, so taking the $\ell^2$-norm, we have
 \[
   \|\bbm\|_2 \le \frac{1}{|z|}\big( 1 + \| \bbm S\bbm\|_2\big) \le  \frac{1}{|z|}\big( 1 + \big\| |\bbm| S|\bbm|\big\|_2\big)
    =  \frac{1}{|z|}\big( 1 + \| F{\bf 1}\|_2\big)\le\frac{2}{|z|},
 \]
 where ${\bf 1} = (1,1,1,\ldots )$, note that $\| {\bf 1}\|_2=1$ and we used \eqref{Fnorm} in the last step. 
 \end{proof}

 \subsubsection{Bounds valid uniformly in the spectrum}

 In this section we introduce an extra regularity assumption that enables us to control $\bbm$ uniformly  throughout the
spectrum, including edges and cusps. 
For simplicity, we restrict our attention to the special case when $s_{ij}$ originates from a piecewise continuous
nonnegative profile function $S(x,y)$
defined on $[0,1]\times [0,1]$, i.e. we assume
\be\label{profile2}
     s_{ij} = \frac{1}{N} S\Big( \frac{i}{N}, \frac{j}{N}\Big).
 \ee
 We will actually need  that $S$ is piecewise 1/2-H\"older  continuous \eqref{piece holder}.
  
 \begin{theorem}\label{thm:unbound}
 Assume that $s_{ij}$ is given  by \eqref{profile2} with a piecewise H\"older-1/2 continuous function $S$
 with uniform lower and upper bounds  $c\le S(x,y)\le C$. Then for any $R>0$ and for any  $|z|\le R $ we have
 \[
     |\bbm(z)|\sim 1, \qquad  \im m_i(z)\sim \im m_j(z),
 \]
 where the implicit constants in the $\sim $ relation depend only on $c, C$ and $R$. In particular, 
 all components of $\im \bbm$ are comparable, hence 
 \be\label{compare}
  \im m_i \sim \langle  \im \bbm \rangle = \varrho. 
 \ee
 \end{theorem}
 We mention that this theorem also holds under weaker conditions.  Piecewise 1/2-H\"older  continuity
 can be replaced by a a weaker condition called {\it component regularity}, see  Assumption (C) in \cite{AEK1short}.
 Furthermore, the uniform lower bound 
 of $S(x,y)$ can be replaced with
 a condition  called   {\it diagonal positivity} see  Assumption (A) in \cite{AEK1short}
 but we omit these generalizations here.
 
 \begin{proof} We have already obtained an $\ell^2$-bound $\| \bbm\|_2\lesssim 1$ in Theorem~\ref{thmbulk}. Now
 we consider any two indices $i, j$, evaluate \eqref{vectordyson3} at these points and subtract them. From 
 \[
    -\frac{1}{m_i} = z+  (S\bbm)_i,\qquad    -\frac{1}{m_j} = z+  (S\bbm)_j  
 \]
 we thus obtain
 \[
   \Big|\frac{1}{m_i}\Big|\le  \Big|\frac{1}{m_j}\Big| + \sum_k |s_{ik}-s_{jk}| |m_k| 
   \le \Big|\frac{1}{m_j}\Big| + \|\bbm\|_2  \Big( N \sum_k |s_{ik}-s_{jk}|^2\Big)^{1/2}.
\] 
 Using \eqref{profile2} and the H\"older continuity (for simplicity assume $n=1$), we have
 \[
 N \sum_k |s_{ik}-s_{jk}|^2 \le \frac{1}{N} \sum_k \Big| S\Big(\frac{i}{N}, \frac{k}{N}\Big) 
  -S\Big(\frac{j}{N}, \frac{k}{N}\Big)\Big|^2 \le C\frac{|i-j|}{N},
 \]
 thus
 \[
   \Big|\frac{1}{m_i}\Big|\le  \Big|\frac{1}{m_j}\Big| + C' \sqrt{\frac{|i-j|}{N}}.
 \]
Taking the reciprocal and squaring it we have for every fixed $j$ that 
\[
     \frac{1}{N} \sum_i \Bigg[ \frac{1}{  \Big|\frac{1}{m_j}\Big| + C' \sqrt{\frac{|i-j|}{N}}    } \Bigg]^2 \le \frac{1}{N}\sum_i |m_i|^2 = \| \bbm\|_2^2\lesssim 1.
 \]
 The left hand side is can be estimated from below by
 \[
     \frac{1}{N} \sum_i \Bigg[ \frac{1}{  \Big|\frac{1}{m_j}\Big| + C' \sqrt{\frac{|i-j|}{N}}    } \Bigg]^2 \gtrsim
      \frac{1}{N} \sum_i  \frac{1}{  \frac{1}{|m_j|^2} + \frac{|i-j|}{N}}    \gtrsim \log |m_j|.
 \]
 Combining the last two inequalities, 
 this shows the uniform upper bound
 \[  
      |\bbm|\lesssim 1.
  \]
  The lower bound is obtained from
  \[
     \Big| \frac{1}{m_i}\Big| = \big| z + \sum_j s_{ij} m_j \big| \le |z| + \frac{C}{N}\sum |m_j| \lesssim 1
  \]
  using the upper bound  $|\bbm|\lesssim1$  and  $s_{ij}\lesssim 1/N$. This proves $|\bbm|\sim 1$.
            
  To complete the proof, note that comparability of the components of  $\im \bbm$ now follows from the imaginary part of \eqref{vectordyson3}, $|\bbm|\sim 1$  and from
  $S(\im \bbm) \sim \langle \im\bbm \rangle$:
    \begin{equation*}
	    \frac{\im \bbm}{|\bbm|^2} = \eta + S\big( \im \bbm\big) \quad \Longrightarrow \quad \im \bbm \sim \eta +\langle \im\bbm \rangle . \qedhere
    \end{equation*}
\end{proof}
 
  \subsection{Regularity of the solution and the stability operator}\label{sec:reg}

  In this section we prove some parts of the regularity Theorem~\ref{thm:cusp}.
  We will not go into the details of the edge and cusp analysis here, see \cite{AEK1short} for a shorter qualitative
  analysis and \cite{AjaErdKru2015} for the full quantitative analysis of all possible singularities. Here we
  will only show the 1/3-H\"older regularity \eqref{holdereq}. We will use this opportunity to introduce 
  and analyze the key stability operator of the problem which then will also be used in the random matrix part
  of our analysis.
  
It is to keep in mind that the small $\eta=\im z$ regime is critical; typically bounds of order $1/\eta$ 
 or $1/\eta^2$ are easy to obtain but these are useless for local analysis (recall that $\eta$ 
 indicates the scale of the problem).  For the fine regularity  properties of the solution, one needs 
 to take $\eta\to0$ with uniform controls. For the random matrix part, we will take $\eta$ down to $N^{-1+\gamma}$
 for any small $\gamma>0$, so any $1/\eta$ bound would not be affordable. 
  
  
\begin{proof}[{\it Proof of (i) and (iii) from Theorem~\ref{thm:cusp}.}]  We differentiate \eqref{vectordyson3} with respect to $z$ (note
 that $\bbm(z)$ is real analytic by \eqref{reps} for any $z\in \Cp$).
 \be\label{diff}
    -\frac{1}{\bbm} = z + S\bbm \quad \Longrightarrow \quad \frac{\partial_z \bbm}{\bbm^2} = 1 + S\partial_z\bbm \quad \Longrightarrow \quad 
    \partial_z\bbm = \frac{1}{1-\bbm^2 S} \bbm^2.
 \ee
 The  ($z$-dependent) linear operator  $1-\bbm^2 S$ is called the {\bf stability operator}.  We will later prove the following main bound
 on this operator:
 
 \begin{lemma}[Bound on the stability operator]\label{lm:stab}
 Suppose that for any $z\in\Cp$ with $|z|\le C$ we have  $|\bbm(z)|\sim 1$. Then 
\be\label{stab bound}
      \Big\| \frac{1}{1- \bbm^2 S}\Big\|_2\lesssim \frac{1}{\varrho(z)^2}= \frac{1}{ \langle \im \bbm \rangle ^2}.
 \ee
 In fact, the same bound also holds in the $\ell^\infty\to \ell^\infty$ norm, i.e.
 \be\label{stab bound infty}
      \Big\| \frac{1}{1- \bbm^2 S}\Big\|_\infty\lesssim \frac{1}{\varrho(z)^2}= \frac{1}{ \langle \im \bbm \rangle ^2}.
 \ee
 \end{lemma}
\medskip
 \noindent By Theorem~\ref{thm:unbound} we know that under conditions of Theorem~\ref{thm:cusp}, we have $\| \bbm\|\sim1$, so the lemma is
 applicable.
 
 Assuming this lemma for the moment, and using that $\bbm$ is analytic on $\Cp^N$,  we conclude from \eqref{stab bound infty} that
 \[
     |\partial_z \im \bbm|  = \frac{1}{2}|\partial_z \bbm| \lesssim \frac{1}{ \langle \im \bbm \rangle ^2}\sim\frac{1}{  (\im \bbm)^2},
 \]
 i.e. the derivative of $(\im \bbm(z))^3$ is bounded.
 Thus $z\to \im \bbm(z)$ is a 1/3-H\"older regular function on the open upper half plane with a 
 uniform H\"older constant. Therefore  $\im \bbm(z)$ extends to the real axis as a 1/3-H\"older continuous function.
 This proves \eqref{holdereq}. Moreover, it is real analytic away from the edges of the self-consistent spectrum
 $\fS = \{ \tau \in \R\; : \; \varrho(\tau)>0\}$; indeed on $\fS$ it  satisfies an analytic ODE \eqref{diff} with bounded coefficients
 by \eqref{stab bound infty} while outside of the closure of $\fS$ the density is zero. 
  \end{proof}
 
  \begin{exercise}
 Assume the conditions  of Theorem~\ref{thm:cusp}, i.e.  \eqref{sbound} and that $S$ is piecewise H\"older continuous \eqref{piece holder}.
  Prove that the saturated self-energy operator has norm 1 on the imaginary axis
 exactly on the support of the self-consistent density of states. In other words, 
 \[
      \lim_{\eta\to 0+}\| F(E+i\eta)\|_2=1 \quad \mbox{if and only if} \quad E\in \textup{supp}\, \varrho.
  \]
  {\it Hint:} First prove that the Stieltjes transform of a 1/3-H\"older continuous function 
  with compact support is itself 1/3-H\"older continuous up to the real line. 
 \end{exercise}

 \subsection{Bound on the stability operator}
 \begin{proof}[{\it Proof of Lemma~\ref{lm:stab}}]
 
 The main mechanism for the stability bound \eqref{stab bound} goes through the operator  $F =|\bbm|S\big( |\bbm| \cdot\big)$ 
 defined in \eqref{Fdef}.
 We know that $F$ has a single largest eigenvalue, but in fact under the condition \eqref{sbound1} this matrix has a substantial gap
 in its spectrum below the largest eigenvalue. To make this precise,
 we start with a definition:

\begin{definition}
For a  hermitian  matrix  $T$ the {\bf spectral gap} $\mbox{Gap}(T)$ is the difference 
between the two largest eigenvalues of $|T|=\sqrt{TT^*}$. If $\| T\|_2$ is a degenerate eigenvalue of $|T|$, then the
gap is zero by definition.
\end{definition}

The following simple lemma shows that matrices with nonnegative entries tend to have a positive gap:

\begin{lemma}\label{lm:gap} Let $T=T^*$ have nonnegative entries, $t_{ij} = t_{ji}\ge 0$ and let $\bh$ be
the Perron-Frobenius eigenvector, $T\bh = \| T\|_2\bh$ with $\bh\ge 0$. Then
\[
    \mbox{Gap}(T)\ge \Big(\frac{ \| \bh\|_2}{\|\bh\|_\infty}\Big) \cdot \min_{ij} t_{ij}.
\]
\end{lemma}
\begin{exercise}
Prove this lemma. Hint: Set $\|T\|_2=1$ and take a vector $\bu\perp\bh$, $\|\bu\|_2=1$. Verify that
\[
  \langle \bu, (1\pm T)\bu\rangle = \frac{1}{2}\sum_{ij}  t_{ij}\Big[ u_i \Big(\frac{h_j}{h_i}\Big)^{1/2}\pm u_i\Big(\frac{h_i}{h_j}\Big)^{1/2}\Big]^2
\]
and estimate it  from below.
\end{exercise}

Applying this lemma to $F$, we have the following:

\begin{lemma}\label{PFlemma}
Assume \eqref{sbound1} and let $|z|\le C$. Then $F$ has norm of order one, it has uniform spectral gap;
\[ 
    \|F\|_2\sim 1, \qquad \mbox{Gap}(F)\sim 1;
\]
and its $\ell^2$-normalized Perron-Frobenius eigenvector, $\bbf$ with $F\bbf=\| F\|_2\bbf$, has comparable components
\[
     \bbf\sim 1.
\]
\end{lemma}

\begin{proof} We have already seen that $\|F\|_2\le 1$. The  lower bound $\| F\|_2\gtrsim 1$ follows from 
$F_{ij} = |m_i|s_{ij}|m_j|\gtrsim 1/N$, in fact $F_{ij}\sim N^{-1}$, thus $\| F{\bf 1}\|_2\gtrsim 1$.
For the last statement, we write $\bbf = \|F\|^{-1}_2 F\bbf\sim F\bbf \sim \langle \bbf\rangle$
and then by normalization obtain $1=\| \bbf\|_2\sim \langle \bbf\rangle\sim \bbf$.
Finally the statement on the gap follows from Lemma~\ref{lm:gap} and that $\|\bbf\|_\infty\sim \|\bbf\|_2$. 
\end{proof}

 Armed with this information on $F$, we explain how $F$ helps to establish a bound on the stability operator.
 Using the polar decomposition  $\bbm =e^{i\bvarphi }|\bbm|$, we can write for any vector $\bw$
 \be\label{keysymm}
    (1-\bbm^2S)\bw = |\bbm|\big(1-e^{2i\bvarphi} F\big)|\bbm|^{-1}\bw .
 \ee
 Since $|\bbm|\sim 1$, it is sufficient to invert $1 -e^{2i\bvarphi} F$ or $e^{-2i\bvarphi}- F$. Since $F$ has a real spectrum,
 this latter  matrix should intuitively  be invertible unless $\sin 2\bvarphi \approx 0$. This intuition is indeed correct 
 if $\bbm$ and thus $e^{2i\bvarphi}$ were constant; the general case is more complicated.
 
 Assume first that we are in the generalized Wigner case, when $\bbm =m_{sc}\cdot {\bf 1}$, i.e. the solution 
 is a constant vector with components $m:= m_{sc}$. Writing $m= |m|e^{i\varphi}$ with some phase $\varphi$, we see that
 \[   
    1- m^2 S = 1- e^{2i\varphi}F.
 \]
 Since $F$ is hermitian and has norm bounded by 1, it has spectrum in $[-1,1]$. So without the phase the inverse of $1-F$ would
 be quite singular (basically, we would have  $\| F\|_2\approx 1- c\eta$, see \eqref{Fnorm} at least in the bulk spectrum). The phase $e^{2i\varphi}$
 however rotates $F$ out of the real axis, see the picture.
  \bigskip
 
 \centerline{\includegraphics[width=7cm]{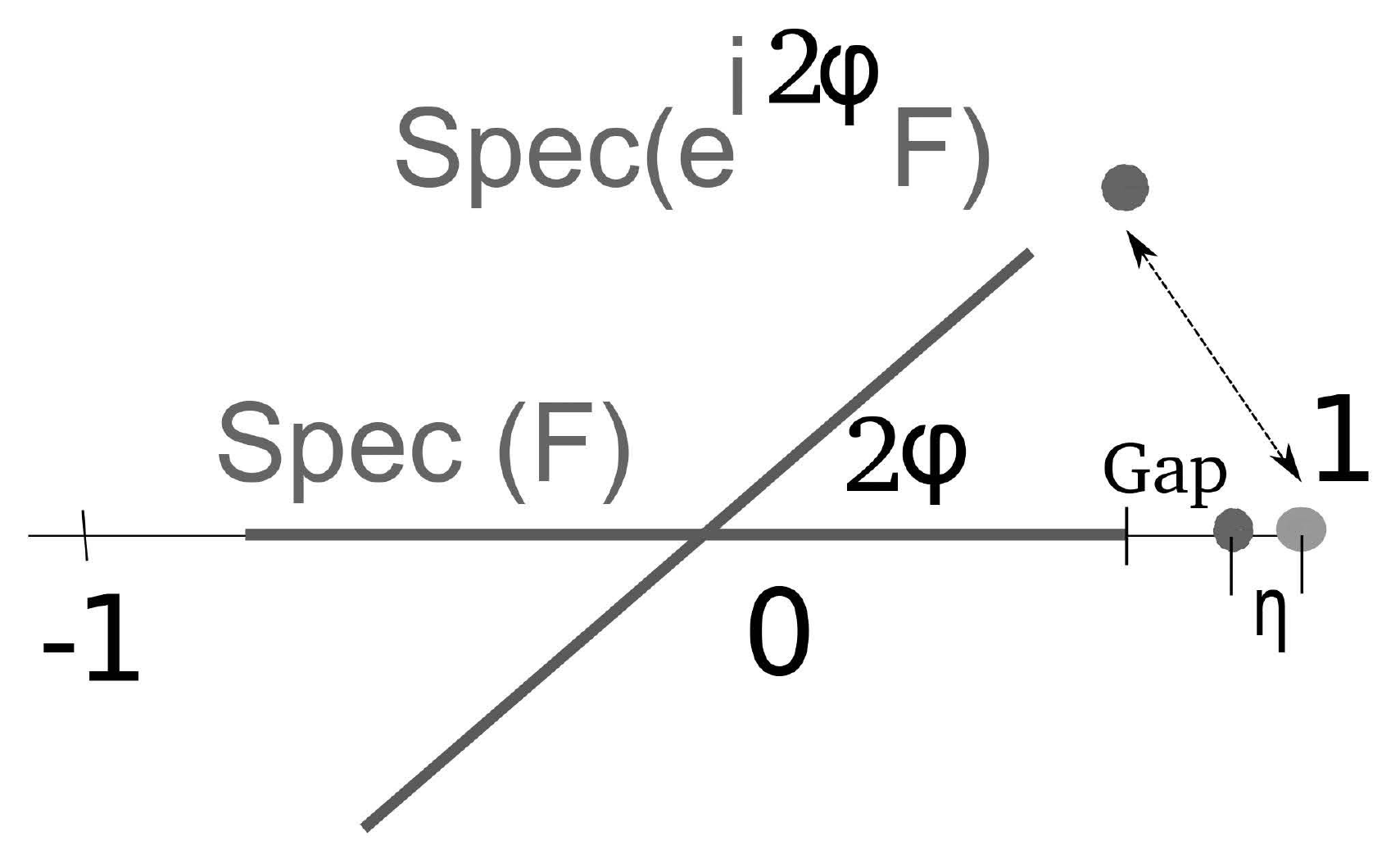}}
 
 
  The distance of 1 from the spectrum of $F$ is tiny, but from the spectrum of 
 $e^{2i\varphi}F$ is comparable with $\varphi\sim \im m=\varrho$:
 \[
   \Big\| \frac{1}{1- m^2 S}\Big\|_2 = \Big\| \frac{1}{1- e^{2i\varphi}F}\Big\|_2 \sim \frac{C}{|\varphi|}  \sim \frac{C}{ \varrho}
  \]
  in the regime where $|\varphi|\le\pi/2$ thanks to the gap in the spectrum of $F$ both below 1 and above $-1$.
  In fact this argument indicates  a better bound of order $1/\varphi\sim 1/\varrho$
 and not only  its square in \eqref{stab bound}.

  
 For the general case, when $\bbm$ is not constant, such a simple argument does not work, 
  since the rotation angles $\varphi_j$  from $m_j =e^{i\varphi_j} |m_j|$  now depend on the coordinate $j$, so 
 there is no simple geometric relation between the spectrum of $F$ and that of $\bbm^2S$. In fact the optimal bound in general
 is $1/\varrho^2$ and not $1/\varrho$.

 To obtain it, we still use the identity
 \be\label{trick}
  (1-\bbm^2S)\bw =e^{2i\bvarphi} |\bbm|\big(e^{-2i\bvarphi}- F\big)|\bbm|^{-1}\bw ,
  \ee
  and focus on inverting $e^{-2i\bvarphi}- F$. We have the following general lemma:
  
  \begin{lemma}\label{lm:UT}
  Let $T$ be hermitian with $\|T \|_2\le 1$ and with top normalized eigenvector $\bbf$, i.e.  $T\bbf =\|T \|_2\bbf$.
  For any unitary operator $U$ we have
\be\label{UTeq}
\Big\|\frac{1}{U-T}\Big\|_2\le \frac{C}{\mbox{Gap}(T)\; \cdot\big| 1-\| T\|_2 \langle \bbf, U\bbf\rangle\big|}.
\ee
\end{lemma}
A simple calculation shows that this lemma applied to $T=F$ and $U = \big( |\bbm|/\bbm\big)^2$ yields
the bound $C/\varrho^2$ for the inverse of $e^{-2i\bvarphi}- F$ since 
\[
   \big| 1-\| T\|_2\langle \bbf, U\bbf\rangle\big| \ge \re \big[ 1 -\big\langle \frac{\bbm^2 \bbf^2}{|\bbm|^2}\big\rangle\big]
   =2\big\langle \frac{(\im \bbm)^2 \bbf^2}{|\bbm|^2}\big\rangle\sim \langle \im \bbm\rangle^2.
\]
This proves the $\ell^2$-stability  bound \eqref{stab bound} in Lemma~\ref{lm:stab}. Improving this  bound  to the stability bound \eqref{stab bound infty} in $\ell^\infty$ is left as the following exercise.
\end{proof}

\begin{exercise}\label{stabex}
By using
$|\bbm(z)|\sim 1$ and \eqref{sbound}, prove \eqref{stab bound infty} from \eqref{stab bound}. Hint: show 
that for any matrix $R$ such that $1-R$ is invertible, we have
\begin{equation}\label{Rid}
    \frac{1}{1-R} = 1 + R + R  \frac{1}{1-R} R,
\end{equation}
and apply this 
with $R= \bbm^2 S$.
\end{exercise}


\begin{proof}[{\it Sketch of proof of Lemma~\ref{lm:UT}}] For details, see Appendix B of \cite{AEK1short}. The  idea is that one needs a lower bound on
$\| (U-T)\bw\|_2$ for any $\ell^2$-normalized $\bw$. Split $\bw$ as $\bw = \langle \bbf, \bw\rangle \bbf + P\bw$, where $P$ is 
the orthogonal projection to the complement of $\bbf$. We will frequently use that
\be\label{gapuse}
   \| TP\bw\|_2\le \big[ \| T\|-\mbox{Gap}(T)\big] \|P\bw\|_2,
\ee
following from the definition of the gap.
 Setting  $\alpha : = \big| 1-\| T\|_2 \langle \bbf, U\bbf\rangle\big|$, we distinguish 
three cases
\begin{itemize}
\item[(i)]   $16\| P\bw\|^2_2\ge \alpha$;
\item [(ii)] $16\| P\bw\|^2_2< \alpha$ and $\alpha \ge \| PU\bbf\|^2_2$;
\item [(iii)] $16\| P\bw\|^2_2< \alpha$ and $\alpha < \| PU\bbf\|^2_2$.
\end{itemize}
In regime (i) we use a crude triangle inequality $\| (U-T)\bw\|_2 \ge \|\bw\|_2 - \| T\bw\|_2$, the splitting of $w$ and \eqref{gapuse}.
In regime (ii) we first project $(U-T)\bw$ onto the $\bbf$ direction: $\| (U-T)\bw\|_2 \ge |\langle \bbf, (1-U^*T)\bw\rangle|$ and estimate.
Finally in regime (iii) we first project $(U-T)\bw$ onto the $P$ direction $\| (U-T)\bw\|_2 \ge \| P(U-T)\bw\|_2 $ and estimate. 

\begin{exercise} Complete the analysis of all these three regimes and finish the proof of Lemma~\ref{lm:UT}. \qedhere
\end{exercise}

\end{proof}

\section{Analysis of the matrix Dyson equation}\label{sec:matrix}

\subsection{Properties of the solution to the MDE}
In this section we analyze the matrix Dyson equation introduced in \eqref{mde4}
\be\label{mde5}
      I + (z+\cS[M])M=0, \qquad \im M > 0, \quad \im z>0,   \qquad (MDE)
\ee
where we assume that $\cS: \C^{N\times N}\to \C^{N\times N}$ is a symmetric and positivity preserving linear map.
In many aspects the analysis  goes parallel to that of the vector Dyson equation and we will highlight only
the main complications due to the matrix character of this problem.

The proof of the  existence and uniqueness result, Theorem~\ref{mdeexuniq}, is analogous 
to the vector case  using the Caratheodory metric, so we omit it, see \cite{Helton2007-OSE}. 
The Stieltjes transform representation \eqref{rep matrix} can also be proved
by reducing it to the scalar case (Exercise~\ref{exist}).
The self-consistent density of states is defined as before:
\[
  \varrho(\rd\tau) =\frac{1}{\pi} \langle V(\rd\tau)\rangle = \frac{1}{\pi N} \tr V(\rd\tau),
\]
and its harmonic extension is again denoted by $\varrho(z) =\frac{1}{\pi}\langle \im M(z)\rangle$.

From now on we assume the flatness condition \eqref{flatness} on $\cS$. We have the analogue of 
Theorem \ref{thmbulk} on various bounds on $M$ that can be proven in  a similar manner.
The role of the $\ell^2$-norm, $\| m\|_2$  in the vector case will be played by the  (normalized) Hilbert-Schmidt norm, i.e.
$\| M\|_{hs}: = \big( \frac{1}{N}\tr MM^* \big)^{1/2}$ as it comes from
the natural scalar product
structure on matrices. The role of the supremum norm of   $|\bbm|$ in the vector case will be played by the operator norm $\| M\|_2$
in the matrix case and similarly the supremum norm  of $1/|\bbm|$ is replaced with $\| M^{-1}\|_2$.

\begin{theorem}\label{thmbulkmat}[Bounds on $M$] 
Assuming the flatness condition \eqref{flatness}, we have
\be\label{MBOUND}
\| M\|_{hs} \lesssim 1, \quad \| M(z)\|_2 \lesssim \frac{1}{\varrho(z)+\mbox{dist}(z,\mbox{supp}(\varrho))}, \qquad \| M^{-1}(z)\|_2\lesssim 1+|z|
\ee
and
\be\label{imboundlower}
\qquad \varrho(z) \lesssim \im M(z)\lesssim (1+|z|)^2 \| M(z)\|^2_2 \varrho(z) 
\ee
where $\| T\|_{hs}: = \big( \frac{1}{N}\tr TT^* \big)^{1/2}$ is the normalized Hilbert-Schmidt norm.
\end{theorem}

\begin{exercise}\label{exist}
Prove that if $M(z)$ is an analytic matrix-valued function on the upper half plane, $z\in \Cp$, such that
$\im M(z)>0$, and $i\eta M(i\eta) \to - I$ as $\eta\to \infty$, then $M(z)$ has a Stieltjes transform representation 
of the form \eqref{rep matrix}. Hint: Reduce the problem to the scalar case  by considering the quadratic form $\langle \bw, M(z) \bw\rangle$
for $\bw\in \C$.
\end{exercise}

\begin{exercise}
Prove Theorem~\ref{thmbulkmat} by mimicking the corresponding proof for the vector case
but watching out for  the non commutativity of the matrices.
\end{exercise}

\subsection{The saturated self-energy matrix}

We have seen in the vector Dyson equation that the stability operator $1-\bbm^2 S$ played a central role
both in establishing regularity of the self-consistent density of states and also in establishing the local law. What
is the matrix analogue of this operator? Is there any analogue for the saturated self-energy operator $F$ defined in 
Definition~\ref{def:F} ?

 The matrix responsible for the stability can be easily found, mimicking the calculation \eqref{diff} by differentiating \eqref{mde5}
 wrt. $z$
  \begin{align}\label{diff1}
   I + (z+\cS[M])M=0  \quad \Longrightarrow  & \quad  (I+ \cS [\partial_z M])M + (z+\cS[M])\partial_zM=0 \\
    \quad \Longrightarrow  & \quad 
    \partial_z M  = (1-M\cS[\cdot] M)^{-1} M^2. \nonumber
 \end{align}
where we took the inverse of the ``super operator'' $1- M\cS[\cdot ]M$.
We introduce the notation $\cC_T$ for the operator  {\bf ``sandwiching by a matrix $T$''},
that acts on any matrix $R$ as 
\[ \cC_T[R]:=TRT.
\]
With this notation we have $1- M\cS[\cdot ]M= 1- \cC_M\cS $  that acts on $N\times N$ matrices 
as $(1- \cC_M\cS)[R] = R - M\cS(R) M$.

The boundedness of the inverse of the  stability operator,  $1-\bbm^2S$ in the vector case, relied crucially on 
finding a symmetrized version of the operator $\bbm^2 S$, the saturated self-energy operator
(Definition~\ref{Fdef}), for which spectral theory can be applied, see 
the identity \eqref{trick}.   This will be the heart of the proof in the following section where we control the 
spectral norm of the inverse
of the stability operator.
Note that  spectral theory in the matrix setup means to work with the Hilbert space of
matrices, equipped with the Hilbert-Schmidt scalar product. We denote by $\| \cdot\|_{sp}: = \| \cdot \|_{hs\to hs}$
the corresponding norm of superoperators viewed as linear maps on this Hilbert space.

\subsection{Bound on the stability operator}

The key technical result of the analysis of the MDE is the following lemma:

\begin{lemma}\label{lm:stabM} Assuming the flatness condition \eqref{flatness}, we have, for $|z|\le C$,
\be\label{CMS}
   \Big\| (1- \cC_{M(z)}\cS)^{-1}\Big\|_{sp} \lesssim \frac{1}{ \big[\varrho(z) + \mbox{dist}(z, \mbox{supp}(\varrho))\big]^C}
\ee
with some universal constant ($C=100$ would do). 
\end{lemma}

Similarly to the argument in Section~\ref{sec:reg} for the vector case, the bound \eqref{CMS} directly implies H\"older 
regularity of the solution and it implies \eqref{matrixholder}. It is also  the key estimate in the random matrix part 
of the proof of the local law.

 
\begin{proof}[{\it Proof of Lemma~\ref{lm:stabM}.}] In the vector case, the saturated self-energy matrix $F$ naturally emerged from taking the imaginary part of the 
Dyson equation and recognizing a Perron-Frobenius type eigenvector of the form $\frac{\im \bbm}{|\bbm|}$, see \eqref{perron}.
This structure was essential to establish the bound $\|F\|_2\le 1$.
 We proceed similarly for the matrix case to find the analogous super operator $\cF$ that has to be
 symmetric and positivity preserving in addition to having a ``useful'' Perron-Frobenius eigenequation. The imaginary
part of the MDE in the form
\[
   - \frac{1}{M} = z+\cS[M]
\]
is given by
\be\label{impart1}
    \frac{1}{M^*} \im M \frac{1}{M} = \eta + \cS[\im M], \quad \Longrightarrow \quad  \im M = \eta M^*M + M^* \cS[\im M]M.
\ee
What is the analogue of $\frac{\im \bbm}{|\bbm|}$ in this equation that is positive, but this time as a matrix? 
``Dividing by $|M|$'' is a quite ambiguous operation, not just because the matrix multiplication is not commutative, but
also for the fact that for non-normal matrices, the  absolute value of a general matrix  $R$ is not defined in a canonical way.
The standard definition is $|R| = \sqrt{R^*R}$, which leads to the polar decomposition of the 
usual form $R=U|R|$ with some unitary $U$,   but the alternative definition $\sqrt{RR^*}$ would also be
equally justified. But they are not the same, and this ambiguity would destroy the symmetry of the attempted 
super operator $\cF$  if done naively.

Instead of guessing the right form, we just look for the matrix version of $\frac{\im \bbm}{|\bbm|}$ in the form $\frac{1}{Q}\im M \frac{1}{Q^*}$
with some matrix $Q$ yet to be found. Then we can rewrite \eqref{impart1} (for $\eta=0$ for simplicity) as
\[
   \frac{1}{Q}\im M \frac{1}{Q^*} 
= \frac{1}{Q} M^* \frac{1}{Q^*} Q^* \cS \Big[  Q  \;\frac{1}{Q} (\im M) \frac{1}{Q^*} \; Q^*\Big] Q\frac{1}{Q} M \frac{1}{Q^*}.
\]
We write it in the form
\[
     X=  Y^*  Q^* \cS [ Q X Q^*] Q Y, \quad \mbox{with}\quad X :=  \frac{1}{Q} \; (\im M) \frac{1}{Q^*}, \quad Y:= \frac{1}{Q} M \frac{1}{Q^*}
\]
i.e. 
\[ 
     X= Y^* \cF[X] Y, \qquad \mbox{with} \quad \cF[\cdot] := Q^* \cS[ Q\cdot Q^*] Q.
\]
With an appropriate $Q$, this operator will be the correct saturated self-energy operator.
Notice that $\cF$ is positivity preserving. 

To get the Perron-Frobenius structure, we need to ``get rid'' of the $Y$ and $Y^*$ above; we have a good chance
if we require that $Y$ be unitary, $YY^*=Y^*Y=I$. The good news is that $X=\im Y$ and 
if $Y$ is unitary, then  $X$ and $Y$ commute (check this fact). We thus arrive at
\[
      X = \cF [X].
\]
Thus the  Perron-Frobenius argument  applies and we get that  $\cF$ is bounded in spectral norm:
\[ 
   \| \cF\|_{sp}\le 1
\]
Actually, if $\eta>0$, then we get a strict inequality.

Using the definition of $\cF$ and that $M = QYQ^*$ with some unitary $Y$, we can also
 write the operator $\cC_M\cS$ appearing in the stability operator in terms of $\cF$. Indeed,
for any matrix $R$
\[
    M\cS[R] M =  Q Y Q^*\cS\Big[Q \frac{1}{Q} R \frac{1}{Q^*} Q^* \Big] QY Q^* =  QY \cF\Big[ \frac{1}{Q} R \frac{1}{Q^*}\Big] Y Q^*
\]
so
\[ 
   R - M\cS[R] M  = Q\Big( 1 - Y\cF [\cdot ] Y\Big)\big[ \frac{1}{Q} R \frac{1}{Q^*}\big] Q^*.
\]
Thus
\be\label{keysymm1}
   I - \cC_M \cS = \cK_Q (I- \cC_Y \cF) \cK^{-1}_{Q},
\ee
where for any matrix $T$ we defined the super operator $\cK_T$ acting on any matrix $R$
as $\cK_T[R]: = TRT^*$ to be the symmetrized analogue of the  sandwiching operator $\cC_T$.
The formula \eqref{keysymm1} is the matrix analogue of \eqref{keysymm}.

Thus, assuming  that $Q\sim 1$ in a sense that $\|Q\|_2\lesssim 1$ and $\| Q^{-1}\|_2 \lesssim 1$,  we have
\[
   \mbox{$I - \cC_M \cS$ is  stable  \quad $\Longleftrightarrow$ \quad $I- \cC_Y \cF$ is stable   \quad $\Longleftrightarrow$\quad  $\cC_{Y^*}- \cF$ is stable} \nc
\]
bringing our stability operator into the form of a ``unitary minus bounded self-adjoint''
to which   Lemma~\ref{lm:UT} (in  the Hilbert space of matrices)
will apply. 

To complete this argument, 
all we need is a ``symmetric polar decomposition''  of $M$ in the form $M = QYQ^*$, where $Y$ is unitary and $Q\sim 1$
knowing that $M\sim 1$. We will give this decomposition explicitly. Write 
$M=A+iB$ with $A :=\re M$ and $B:=\im M>0$. Then we can write
\[
      M = \sqrt{B} \Big( \frac{1}{\sqrt{B}}A \frac{1}{\sqrt{B}} +i \Big)\sqrt{B}
 \]
and now we make the middle factor unitary by dividing its absolute value:
\[
   M = \sqrt{B} W Y W \sqrt{B}= : QYQ^*
\]
\[
  W: = \Bigg[ 1 + \Big(\frac{1}{\sqrt{B}} A \frac{1}{\sqrt{B}}\Big)^2 \Bigg]^{\frac{1}{4}},
  \quad Y:=\frac{ \frac{1}{\sqrt{B}} A \frac{1}{\sqrt{B}} +i}{W^2}.
\]
In the regime, where $c\le B\le C$ and $\|A\|_2\le C$, we have 
\[
   Q = \sqrt{B} W \sim 1
\]
in the sense that $\|Q\|_2\lesssim 1$ and $\|Q^{-1}\|_2\lesssim 1$. In our application, we use
the upper bound \eqref{MBOUND} for $\|M\|_2$ and the lower bound on $B=\im M$ from \eqref{imboundlower}.
This gives a control on both $\|Q\|_2$ and $\| Q^{-1}\|_2$ as a certain power of $\varrho(z)$ and
this will be responsible for parts of the powers collected in  the right hand side of \eqref{CMS}.
In this proof here we focus only on the bulk, so we do not intend to gain the additional term $\mbox{dist}(z, \mbox{supp}\varrho)$
that requires a slightly different argument. The result is
\be\label{tof}
\Big\| \frac{1}{1- \cC_M\cS}\Big\|_{sp} \le \frac{1}{\varrho(z)^C} \Big\| \frac{1}{\cU- \cF}\Big\|_{sp}
\ee
with $\cU:=\cC_{Y^*}$.

We remark that  $\cF$ can also be written as follows: 
\be\label{CF}
  \cF =\cK_Q^* \cS \cK_Q = \cC_W \cC_{ \sqrt{\mbox{\small Im} M}}\cS \cC_{\sqrt{\mbox{\small Im}  M}}\cC_W . 
\ee

Finally, we need to  invert $\cC_{Y^*}- \cF$ effectively with the help of Lemma~\ref{lm:UT}.
Since $\cF$ is positivity preserving,
a Perron-Frobenius type theorem (called the Krein-Rutman theorem in  more general Banach spaces)
applied to $\cF$ yields that it has a normalized {\it eigenmatrix} $F$ with eigenvalue $\|\cF\|_{sp}\le 1$.
 The following lemma collects information on $\cF$ and $F$, similarly to Lemma~\ref{PFlemma}:
 \begin{lemma}
 Assume the flatness condition \eqref{flatness} and let $\cF$ be defined by \eqref{CF}. 
 Then $\cF$ has a unique normalized eigenmatrix corresponding to its largest eigenvalue
 \[
   \cF[F] = \|\cF\|_{sp} F, \qquad  \|F\|_{hs}=1,  \qquad \|\cF\|_{sp}\le 1.
\]
Furthermore
\[
 \|\cF\|_{sp}=1 - \frac{ \langle F, \cC_W[\im M]\rangle}{\langle F, W^{-2}\rangle} \im z,
\]
 the eigenmatrix $F$ has bounds
 \[
       \frac{1}{\|M\|^7_2}\le      F\le \| M\|^6_2
 \]
 and $\cF$ has a spectral gap:
 \be\label{cfgap}
   \mbox{Spec}\big(\cF/\|\cF\|_{sp}\big) \subset [ -1+\theta, 1-\theta] \cup \{1\}, \qquad \theta \ge \|M\|^{-42}_2
 \ee
 (the explicit powers do not play any significant role).
\end{lemma}

We  omit the proof of this lemma (see Lemma 4.6 of \cite{AEK5}), its proof is similar but more involved than
that of Lemma~\ref{PFlemma}, especially the noncommutative analogue of Lemma~\ref{lm:gap} needs substantial changes
(this is given in Lemma A.3 in \cite{AEK5}).

Armed with the bounds on $\cal F$ and $F$, we can use Lemma~\ref{lm:UT} with $T$ playing the role of $\cF$ and
$\cU:=\cC_{Y^*}$ playing the role of $U$:
  \[
 \Big\| \frac{1}{\cU- \cF}\Big\|_{sp} \lesssim \frac{1}{\mbox{Gap}(\cF)\; \; \big| 1-  \|\cF\|\langle F, \cU(F)\rangle\big|}.
\]
We already had a bound on the gap of $\cF$ in \eqref{cfgap}. 
As a last step, we prove the estimate
\[ \big| 1-  \langle F, \cU(F)\rangle\big|= \big|1- \langle F,Y^*FY^*\rangle\big|\ge \frac{\varrho^2(z)}{\| M\|^4_2}\ge \varrho(z)^6.
\]
\begin{exercise} Prove these last two bounds by using $ \big|1- \langle F,Y^*FY^*\rangle\big| \ge \langle F, \cC_{\im Y^*} F\rangle$,
using the definition of $Y$ and various bounds on $M$ from  Theorem~\ref{thmbulkmat}.
\end{exercise}

Combining \eqref{tof} with  these last bounds and with the bound on the gap of $\cF$ \eqref{cfgap}
we  complete the proof of Lemma~\ref{lm:stabM} (without the 
$\mbox{dist}(z, \mbox{supp}\varrho)$ part). 
\end{proof}

\begin{exercise}
Prove the matrix analogue of the unconditional bound \eqref{trivl2}, i.e. if $M$ solves the MDE \eqref{mde5},
where  we only assume that $\cS$ is symmetric and  positivity preserving, then
$
    \| M\|_{hs} \le\frac{2}{|z|}.
$
(Hint: use the representation $M=QYQ^*$ to express $YQ^*Q = -\frac{1}{z}(1+Y \cF(Y))$ and take Hilbert-Schmidt norm on both sides.)
\end{exercise}

\section{Ideas of the proof of the local laws}\label{sec:ideas}

In this section we sketch the proof of the local laws. We will present the more general correlated case,
i.e. Theorem~\ref{thm:univcorr} and we will focus on the entrywise  local law \eqref{perch16cor}.

\subsection{Structure of the proof}
 In Section~\ref{sec:resolvent} around \eqref{egg} we already outlined the main idea.
Starting from $HG= I+zG$, we have the identity 
\be\label{iz}
    I+ (z+ \cS[G])G =  D, \qquad D:= HG + \cS [G]G,
\ee
and we compare it with the matrix Dyson equation
\be\label{mis}
    I + (z+ \cS[M])M =0.
\ee
The first (probabilistic) part of the proof is a good bound on $D$, the second (deterministic) part is to use the stability of the MDE 
to conclude from these two equations that $G-M$ is small.  The first question is in which norm 
should one estimate these quantities?

Since $D$ is still random, it is not realistic to estimate it in operator norm, in fact $\| D\|_2\gtrsim 1/\eta$ with high probability.
To see this,  consider the simplest Wigner case,
\[ 
  D = I +\Big( z +  \frac{1}{N}\tr G\Big) G.
\]
Let $\lambda$ be the closest eigenvalue to $\re z$ with normalized eigenvector $\bu$. Note that
 typically $|\re z- \lambda|\lesssim 1/N$ and $\eta\gg 1/N$, thus $\|G\bu\|_2 = 1/|\lambda- z|  \sim 1/\eta$
 (suppose that $\re z$ is away from zero). From the local law we know that $\frac{1}{N}\tr G\sim m_{sc}\sim 1$
 and $z+m_{sc}\sim 1$. Thus
 \[
    \| D\bu\|_2 = \Big\| \bu + \Big( z +  \frac{1}{N}\tr G\Big) G\bu\Big\|_2\sim \| G\bu\|_2\sim 1/\eta.
 \]
 The appropriate weaker norm is the entrywise maximum norm defined by
 \[
    \| T\|_{\max} : = \max_{ij} |T_{ij}|.
 \]
 
 \subsection{Probabilistic part of the proof}
In the maximum norm we have the following
\begin{theorem}\label{thm:D}
Under the  conditions of Theorem~\ref{thm:univcorr}, for any $\gamma, \e, D>0$ we have the following high
probability statement for some $z=E+i\eta$ with $|z|\le 1000$, 
$\eta\ge N^{-1+\gamma}$: 
\be\label{Dmax}
  \P\biggl(  \|D(z)\|_{\max}\ge \frac{N^\e}{\sqrt{N\eta}} \biggr) \le \frac{C}{N^D},
\ee
i.e. all matrix elements $D_{ij}$ are small simultaneously for all spectral parameters.
\end{theorem}
We will omit the proof, which is a tedious calculation and whose basic ingredients were sketched in 
Section~\ref{sec:resolvent}. For the Wigner type matrices or for correlated matrices with fast (exponential)
correlation decay as in \cite{AEK5} one may use the Schur complement method  together
with concentration estimates on quadratic functionals of independent or essentially independent random vectors (Section~\ref{sec:schur}).
For more general correlations or if nonzero expectation of $H$ is allowed, then we may use the cumulant method (Section~\ref{sec:cum}).
In both cases, one establishes a high moment bound  on $\E |D_{ij}|^p$ via a detailed expansion and then one concludes
a high probability bound via Markov inequality.

\subsection{Deterministic part of the proof}

In the second (deterministic) part of the proof we compare \eqref{iz} and \eqref{mis}. From these two equations we have
\be\label{diffz}
(I-M\cS[\cdot]M) [G-M] =  MD + M \cS[G-M](G-M),
\ee
so by inverting the super operator $I-M\cS[\cdot]M = I-\cC_M\cS$, we get
\be\label{G-M}
  G-M = \frac{1}{I-\cC_M\cS}\big[MD\big]  + \frac{1}{I-\cC_M\cS}\Big[ M \cS[G-M](G-M)\Big].
\ee
Not only is $\| M\|$ is bounded, see \eqref{MBOUND}, but also both
\be\label{maxnorm}
   \| M\|_{\infty}: = \max_i \sum_j |M_{ij}|\qquad \mbox{and}\quad    \| M\|_{1}: = \max_j \sum_i |M_{ij}|
\ee
are bounded. This information is obvious for Wigner type matrices, when $M$ is diagonal. For correlated
matrices with fast correlation decay it requires a somewhat involved additional proof that we do not repeat here,
see Theorem 2.5 of \cite{AEK5}. Slow decay needs another argument \cite{EKScorrelated}.

Furthermore, 
we know that in the bulk spectrum the inverse of the  stability operator is bounded in spectral norm \eqref{CMS}, i.e. 
when the stability operator is considered mapping matrices with Hilbert Schmidt norm. 
We may also consider its norm in the other two natural norms, i.e. when the space of matrices
is equipped with the maximum norm \eqref{maxnorm} and the Euclidean matrix norm $\|\cdot \|$.
 We find  the boundedness of the inverse of the stability operator in these two other
norms as well since we can prove (see Exercise~\ref{ex1})
\be\label{SSb}
    \Big\| \frac{1}{I-\cC_M\cS} \Big\|_{\infty}  + \Big\| \frac{1}{I-\cC_M\cS} \Big\|_{2}
     \lesssim   \Big\| \frac{1}{I-\cC_M\cS} \Big\|_{sp}. 
\ee
Note that the bound on the first term in the left hand side is the analogue of the estimate from Exercise~\ref{stabex}. 
Using all this , we obtain from \eqref{G-M} that
\[ 
   \| G-M\|_{\max}\lesssim \|D \|_{\max} +  \| G-M\|_{\max}^2,
\]
where $\lesssim$ includes factors of $\varrho(z)^{-C}$, which are harmless in the bulk.
From this quadratic inequality  we easily obtain that
\be\label{GMD}
   \| G-M\|_{\max}\lesssim \|D \|_{\max} ,
\ee
assuming a weak bound  $ \| G-M\|_{\max}\ll 1$. This latter information is obtained by a continuity argument in the imaginary part
of the spectral parameter. We fix an $E$ in the bulk, $\varrho(E)>0$ and consider $(G-M)(E+i\eta)$ as a function of $\eta$.
For large $\eta$ we know that both $G$ and $M$ are bounded by $1/\eta$, hence they are small, so the
weak bound  $ \| G-M\|_{\max}\ll 1$ holds. Then we conclude that 
\eqref{GMD} holds for large $\eta$. Since $\| D\|_{\max}$ is small, at least with very high probability, see \eqref{Dmax}, 
we obtain that the strong bound
\be
\label{strong}
     \| G-M\|_{\max}\lesssim \frac{N^\e}{\sqrt{N\eta}}
\ee
also holds. Now we may reduce the value of $\eta$ a bit using the fact that  the function $\eta\to (G-M)(E+i\eta)$ is Lipschitz continuous with
Lipschitz constant $C/\eta^2$. So we know that $ \| G-M\|_{\max}\ll 1$ for this smaller $\eta$ value as well. 
Thus  \eqref{GMD} can again be applied and together with \eqref{Dmax} we get the strong bound  \eqref{strong}
for this reduced $\eta$ as well. We continue this ``small-step'' reduction  as long as the strong bound  implies  the weak bound, i.e. as long as
$N\eta\gg N^{2\e}$, i.e. $\eta \gg N^{-1+2\e}$. Since $\e>0$ is arbitrary we can go down to 
the scales $\eta\ge N^{-1+\gamma}$ for any $\gamma>0$.
 Some care is needed in this argument, since the smallness of $\| D\|_{\max}$ holds only with high probability, 
 so in every step we lose a set of small probability. This is, however, affordable by the union bound since
 the probability of the events where $D$ is not controlled is very small, see \eqref{Dmax}.

The proof of the averaged law \eqref{perch13cor} is similar. Instead of the maximum norm, we use averaged quantities of the
form $\langle TD \rangle =\frac{1}{N}\tr TD$. In the first, probabilistic step  instead of \eqref{Dmax} we prove
that for any fixed deterministic matrix $T$ we have
\[
   \big| \langle TD \rangle\big| \le \frac{N^\e}{N\eta}\| T\|
\] 
with very high probability. Notice that averaged quantities can be estimated with
an  additional $(N\eta)^{-1/2}$ power  better; this is the main reason why averaged law \eqref{perch13cor}
has a stronger control than the entrywise or the isotropic laws.

\begin{exercise}\label{ex1}
Prove \eqref{SSb}. Hint: consider the  identity \eqref{Rid} with $R= \cC_M\cS$
and use the smoothing properties of the self-energy operation $\cS$ following from \eqref{flatness}
and the boundedness of $M$ in all three relevant norms.
\end{exercise}

%
%
%
%
%
%
%

\bibspread
\def\cydot{\leavevmode\raise.4ex\hbox{.}} \def\cprime{$'$} \def\cprime{$'$}
  \def\cprime{$'$}


\end{document}